\documentclass[11pt,english,letterpage]{amsart}
\usepackage{amssymb,bm,mathrsfs,stmaryrd,mathabx}
\usepackage[all]{xy}
\usepackage{hyperref}
\usepackage{verbatim}

\theoremstyle{plain}
\newtheorem{theorem}{Theorem}[subsection]
\newtheorem{bigtheorem}{Theorem}[section]
\newtheorem{proposition}[theorem]{Proposition}
\newtheorem{lemma}[theorem]{Lemma}
\newtheorem{corollary}[theorem]{Corollary}

\theoremstyle{definition}
\newtheorem{definition}[theorem]{Definition}

\theoremstyle{remark}
\newtheorem{remark}[theorem]{Remark}

\numberwithin{equation}{subsection}

\newcommand{\map}[1]{\xrightarrow{#1}}

\newcommand{\iso}{\cong}
\newcommand{\define}{\stackrel{\mathrm{def}}{=}}

\newcommand{\Gal}{\mathrm{Gal}}
\newcommand{\Hom}{\mathrm{Hom}}

\newcommand{\Aut}{\mathrm{Aut}}
\newcommand{\End}{\mathrm{End}}
\newcommand{\Spec}{\mathrm{Spec}}
\newcommand{\Q}{\mathbb Q}
\newcommand{\Z}{\mathbb Z}
\newcommand{\R}{\mathbb R}

\newcommand{\C}{\mathbb C}
\newcommand{\F}{\mathbb F}
\newcommand{\A}{\mathbb A}
\newcommand{\co}{\mathcal O}
\newcommand{\alg}{\mathrm{alg}}

\newcommand{\Lie}{\mathrm{Lie}}
\newcommand{\Fil}{\mathrm{Fil}}
\newcommand{\Pap}{\mathrm{Pap}}
\newcommand{\Kra}{\mathrm{Kra}}

\newcommand{\Falt}{\mathrm{Falt}}
\newcommand{\Pic}{\mathrm{Pic}}

\newcommand{\ord}{\mathrm{ord}}

\newcommand{\GL}{\mathrm{GL}}
\newcommand{\GU}{\mathrm{GU}}
\newcommand{\SO}{\mathrm{SO}}
\newcommand{\SL}{\mathrm{SL}}
\newcommand{\GSpin}{\mathrm{GSpin}}

\newcommand{\kk}{{\bm{k}}}
\newcommand{\dR}{\mathrm{dR}}

\newcommand{\zxz}[4]{\begin{pmatrix} #1 & #2 \\ #3 & #4 \end{pmatrix}}

\newcommand{\calO}{\mathcal{O}}

\newcommand{\tot}{\mathrm{tot}}
\newcommand{\reg}{\mathrm{reg}}

\newcommand{\bs}{\backslash}

\newcommand{\Cha}{\widehat{\operatorname{Ch}}}
\newcommand{\scal}{\text{\rm sc}}
\newcommand{\tr}{\operatorname{tr}}

\newcommand{\CL}{\operatorname{CL}}

\author[J.~Bruinier]{Jan H.~Bruinier}
\address{Fachbereich Mathematik, Technische Universit\"at Darmstadt,  
D-64289 Darmstadt, Germany}
\email{bruinier@mathematik.tu-darmstadt.de}

\author[B.~Howard]{Benjamin Howard}
\address{Department of Mathematics, Boston College, 140 Commonwealth Ave, Chestnut Hill, MA 02467, USA}
\email{howardbe@bc.edu}

\author[S.~Kudla]{Stephen S.~Kudla}
\address{Department of Mathematics, University of Toronto, 40 St. George St., BA6290, Toronto, ON M5S 2E4, Canada}
\email{skudla@math.toronto.edu}

\author[M.~Rapoport]{Michael Rapoport}
\address{
Mathematisches Institut der Universit\"at Bonn, Endenicher Allee 60, 53115 Bonn, Germany, and 
 Department of Mathematics, University of Maryland, College Park, MD 20742, USA 
}
\email{rapoport@math.uni-bonn.de}

\author[T.~Yang]{Tonghai Yang}
\address{Department of Mathematics, University of Wisconsin Madison, Van Vleck Hall, Madison, WI 53706, USA}
\email{thyang@math.wisc.edu}

\title[Modularity of unitary generating series  II]{Modularity of generating series of divisors on  unitary Shimura varieties II: arithmetic applications}

\begin{document}

 \begin{abstract}
We prove two formulas  in the style of the Gross-Zagier theorem, relating derivatives of $L$-functions to arithmetic intersection pairings on a unitary Shimura variety.  We also prove a special case of Colmez's conjecture on the Faltings heights of abelian varieties with complex multiplication. These results are derived from the authors' earlier results on the modularity of generating series of divisors on unitary Shimura varieties.
\end{abstract}



\thanks{J.B.~was  supported in part by DFG grant BR-2163/4-2. 
B.H.~was supported in part by NSF grants DMS-1501583 and DMS-1801905. 
M.R.~was supported in part by the Deutsche Forschungsgemeinschaft through the grant SFB/TR 45. 
S.K.~was supported by an NSERC Discovery Grant. 
T.Y.~was  supported in part by NSF grant DMS-1500743.}

\maketitle

\setcounter{tocdepth}{1}
\tableofcontents



\section{Introduction}


Fix  an integer $n\ge 3$, and a quadratic imaginary field $\kk\subset \C$  of odd discriminant $\mathrm{disc}(\kk)=-D$.  Let $\chi_\kk : \A^\times \to \{ \pm 1\}$ be the associated quadratic character,  let  $\mathfrak{d}_\kk\subset \co_\kk$ denote the different of $\kk$,  let $h_\kk$ be  the class number of $\kk$, and let $w_\kk$ be the number of roots of unity in $\kk$.

By a \emph{hermitian $\co_\kk$-lattice} we mean a projective $\co_\kk$-module of finite rank endowed with a nondegenerate hermitian form.


\subsection{Arithmetic theta lifts}


Suppose  we are given a pair $(\mathfrak{a}_0,\mathfrak{a})$ in which
\begin{itemize}
\item
$\mathfrak{a}_0$ is a  self-dual hermitian $\co_\kk$-lattice of signature $(1,0)$,
\item
$\mathfrak{a}$ is a  self-dual hermitian  $\co_\kk$-lattice of signature $(n-1,1)$.
\end{itemize}
This pair determines hermitian $\kk$-spaces $W_0 =  \mathfrak{a}_{ 0 \Q} $ and $W=\mathfrak{a}_\Q$.

From this data we constructed in \cite{BHKRY} a  smooth Deligne-Mumford stack  $\mathrm{Sh}(G,\mathcal{D})$ of dimension $n-1$ over $\kk$  with complex points
\[
\mathrm{Sh}(G,\mathcal{D})(\C) = G(\Q) \backslash \mathcal{D} \times G(\A_f)/K.
\]
The reductive group
$G\subset \GU(W_0 )\times \GU(W)$
is the largest subgroup on which the two similitude characters agree, and   $K\subset G(\A_f)$ is the largest subgroup stabilizing the $\widehat{\Z}$-lattices    $\widehat{\mathfrak{a}}_0 \subset W_0(\A_f)$ and $\widehat{\mathfrak{a} } \subset W(\A_f)$.

We also defined in \cite[\S 2.3]{BHKRY} an integral model
\begin{equation}\label{kramer inclusion}
\mathcal{S}_\Kra \subset \mathcal{M}_{(1,0)} \times_{\co_\kk} \mathcal{M}_{(n-1,1)}^\Kra
\end{equation}
of $\mathrm{Sh}(G,\mathcal{D})$.  It is regular and flat over $\co_\kk$, and admits a canonical toroidal compactification $\mathcal{S}_\Kra \hookrightarrow \mathcal{S}_\Kra^*$ whose boundary is a smooth divisor.

The main result of \cite{BHKRY} is the construction of a formal generating series of arithmetic divisors
\begin{equation}\label{modular divisors}
\widehat{\phi} (\tau) = \sum_{m\ge 0} \widehat{\mathcal{Z}}_\Kra^\mathrm{total}(m)\cdot q^m \in \widehat{\mathrm{Ch}}^1_\Q(\mathcal{S}_\Kra^*) [[q]]
\end{equation}
valued in the Gillet-Soul\'e codimension one arithmetic Chow group with rational coefficients, extended to allow log-log Green functions at the boundary as in  \cite{BKK,BBK}, and the proof that this generating series is  modular of weight $n$, level $\Gamma_0(D)$,  and character $\chi_\kk^n$.     The modularity result implies that the coefficients span a finite-dimensional subspace of the arithmetic Chow group
 \cite[Remark 7.1.2]{BHKRY}.

After passing to the arithmetic Chow group with complex coefficients,  for any classical modular form
\[
g  \in S_n(\Gamma_0(D) , \chi_\kk^n )
\]
we may form the Petersson inner product
\[
\langle  \widehat{\phi} ,g \rangle_{\mathrm{Pet}} = \int_{ \Gamma_0(D) \backslash \mathcal{H} }
 \overline{ g (\tau)}   \cdot \widehat{\phi} (\tau) \, \frac{du\, dv}{ v^{2-n}}
\]
where $\tau=u+iv$.
As in  \cite{Ku04}, define   the \emph{arithmetic theta lift}
\begin{equation}\label{ATL}
\widehat{\theta}(g ) = \langle  \widehat{\phi} , g  \rangle_\mathrm{Pet} \in \widehat{\mathrm{Ch}}^1_\C( \mathcal{S}_\Kra^* )  .
\end{equation}

Armed with the construction of the arithmetic theta lift (\ref{ATL}), we are now able to complete the program of \cite{Ho1, Ho2, BHY} to prove Gross-Zagier style formulas relating arithmetic intersections to derivatives of $L$-functions.

The Shimura variety $\mathcal{S}_\Kra^*$ carries different families of codimension $n-1$ cycles constructed from complex multiplication points, and our results show that the arithmetic intersections of these families with  arithmetic lifts are related to central derivatives of $L$-functions.


\subsection{Central derivatives and small CM points}


  In \S \ref{s:small GZ}  we construct an \'etale and proper Deligne-Mumford stack $\mathcal{Y}_\mathrm{sm}$ over $\co_\kk$, along with a morphism
  \[
  \mathcal{Y}_\mathrm{sm} \to \mathcal{S}_\Kra^*.
  \]
This is  the \emph{small CM cycle}.
 Intersecting  arithmetic divisors against $\mathcal{Y}_\mathrm{sm}$ defines    a  linear functional
\[
[ \, -   :  \mathcal{Y}_\mathrm{sm}] : \widehat{\mathrm{Ch}}^1_\C( \mathcal{S}_\Kra^* ) \to \C,
\]
and our first main result computes the image of the arithmetic theta lift (\ref{ATL}) under this linear functional.

The statement involves the  convolution $L$-function $L( \tilde{g},\theta_\Lambda,s)$ of two modular forms
\[
\tilde{g}\in S_n( \overline{\omega}_L ) ,\qquad \theta_\Lambda \in M_{n-1}( \omega_\Lambda^\vee )
\]
valued in finite-dimensional representations of $\SL_2(\Z)$.  We refer the reader to \S \ref{ss:convolution} for the precise definitions.
Here we note only that $\tilde{g}$ is the image of $g$ under an induction map
\begin{equation}\label{induction map}
S_n(\Gamma_0(D) , \chi^n_\kk ) \to S_n( \overline{\omega}_L )
\end{equation}
from scalar-valued forms to vector-valued forms, that $\theta_\Lambda$ is the theta function attached to a quadratic space $\Lambda$ over $\Z$ of signature $(2n-2,0)$, and that the $L$-function $L( \tilde{g} ,\theta_\Lambda,s)$ vanishes at its center of symmetry $s=0$.

\begin{bigtheorem}\label{Theorem A}
The arithmetic theta lift (\ref{ATL}) satisfies
\[
[ \widehat{\theta}(g ) : \mathcal{Y}_\mathrm{sm} ] = -  \deg_\C (\mathcal{Y}_\mathrm{sm}) \cdot  \frac{d}{ds} L( \tilde{g} , \theta_\Lambda,s) \big|_{s=0}.
\]
Here we have defined
\[
\deg_\C (\mathcal{Y}_\mathrm{sm})   = \sum_{ y\in  \mathcal{Y}_\mathrm{sm}(\C) }  \frac{1}{| \Aut(y) |},
\]
where the sum is over the finitely many isomorphism classes of the groupoid of complex points of $\mathcal{Y}_\mathrm{sm}$, viewed as an $\co_\kk$-stack.
\end{bigtheorem}

The proof is given in \S \ref{s:small GZ}, by combining the modularity result of \cite{BHKRY} with the main result   of \cite{BHY}.
In \S \ref{s:convolution} we provide alternative formulations of Theorem \ref{Theorem A} that involve the usual convolution $L$-function of scalar-valued modular forms, as opposed to the vector-valued forms $\tilde{g}$ and $\theta_\Lambda$.  See especially Theorem \ref{maintheo2}.


\subsection{Central derivatives and big CM points}
\label{ss:intro big GZ}


Fix a totally real field $F$ of degree $n$, and define a CM field \[E=\kk\otimes_\Q F.\]  Let $\Phi \subset \Hom(E,\C)$ be a CM type of signature $(n-1,1)$, in the sense that there is a unique $\varphi^\mathrm{sp} \in \Phi$, called the \emph{special embedding}, whose restriction to $\kk$ agrees with the complex conjugate of the inclusion $\kk\subset \C$.  The reflex field of the pair $(E,\Phi)$ is
\[
E_\Phi = \varphi^\mathrm{sp} (E) \subset \C,
\]
and we denote by $\co_\Phi \subset E_\Phi$  its ring of integers.

We define in \S \ref{ss:big cm} an \'etale and proper Deligne-Mumford stack $\mathcal{Y}_\mathrm{big}$ over $\co_\Phi$, along with a morphism of $\co_\kk$-stacks
\[
\mathcal{Y}_\mathrm{big} \to \mathcal{S}_\Kra^* .
\]
This is the \emph{big CM cycle}.   Here we view $\mathcal{Y}_\mathrm{big}$ as an $\co_\kk$-stack using the inclusion $\co_\kk\subset \co_\Phi$ of subrings of $\C$ (which is the complex conjugate of the special embedding $\varphi^\mathrm{sp} : \co_\kk \to \co_\Phi$).
 Intersecting  arithmetic divisors against $\mathcal{Y}_\mathrm{big}$ defines    a  linear functional
\[
[ \, -   :  \mathcal{Y}_\mathrm{big}] : \widehat{\mathrm{Ch}}^1_\C( \mathcal{S}_\Kra^* ) \to \C.
\]

Our second main result relates the image of the arithmetic theta lift (\ref{ATL}) under this linear functional to the central derivative of a generalized $L$-function defined as the Petersson inner product  $\langle E(s) , \tilde{g} \rangle_\mathrm{Pet}$.  The modular form $\tilde{g}(\tau)$ is, once again, the image of $g(\tau)$ under the induction map (\ref{induction map}).  The modular form $E(\tau,s)$ is defined as the restriction via the diagonal embedding $\mathcal{H} \to \mathcal{H}^n$ of  a weight one Hilbert modular Eisenstein series valued in the space of the contragredient  representation $\omega_L^\vee$.  See \S \ref{ss:weird L} for details.

\begin{bigtheorem}\label{Theorem B}
Assume that the discriminants of $\kk/\Q$ and $F/\Q$ are odd and relatively prime.
The arithmetic theta lift (\ref{ATL}) satisfies
\[
[ \widehat{\theta}(g ) : \mathcal{Y}_\mathrm{big} ]
=  \frac{-1}{n} \cdot \deg_\C (\mathcal{Y}_\mathrm{big})  \cdot   \frac{d}{ds} \langle E(s) , \tilde{g} \rangle_\mathrm{Pet} \big|_{s=0} .
\]
Here we have defined
\[
\deg_\C (\mathcal{Y}_\mathrm{big})   = \sum_{ y\in  \mathcal{Y}_\mathrm{big}(\C) }  \frac{1}{| \Aut(y) |},
\]
where the sum is over the finitely many isomorphism classes of the groupoid of complex points of $\mathcal{Y}_\mathrm{big}$, viewed as an $\co_\kk$-stack.
\end{bigtheorem}

The proof is given in \S \ref{s:big GZ}, by combining the modularity result of  \cite{BHKRY} with the intersection calculations of \cite{BKY,Ho1,Ho2}.


\subsection{Colmez's conjecture}


Suppose $E$ is a CM field with maximal totally real subfield $F$.
Let $D_E$ and $D_F$ be the absolute discriminants of $E$ and $F$, set
$
\Gamma_\R(s) = \pi^{-s/2}\Gamma(s/2),
$
and define the completed $L$-function
\[
\Lambda(s,\chi_E) =  \left|\frac{ D_E}{D_F}\right|^{ \frac{s}{2}}  \Gamma_\R(s+1)^{ [F:\Q] }  L(s,\chi_E)
\]
of the character $\chi_E: \A_F^\times \to \{ \pm 1\}$ determined by $E/F$.
 It satisfies the functional equation $\Lambda(1-s,\chi_E) = \Lambda(s,\chi_E)$, and
\[
\frac{\Lambda'(0,\chi_E)}{\Lambda(0,\chi_E)} =
\frac{L'(0,\chi_E)}{L(0,\chi_E)} + \frac{1}{2} \log  \left|\frac{ D_E}{D_F}\right|  - \frac{ [F:\Q] }{2} \log(4\pi e^\gamma) ,
\]
where $\gamma = -\Gamma'(1)$ is the Euler-Mascheroni constant.

Suppose $A$ is an abelian variety over $\C$ with complex multiplication by $\co_E$ and CM type $\Phi$.
In particular $A$ is defined over the algebraic closure of $\Q$ in $\C$.
It is a theorem of Colmez  \cite{Colmez}  that the Faltings height
\[
h^\Falt_{(E,\Phi)} = h^\Falt(A)
\]
 depends only on the pair $(E,\Phi)$, and not on $A$ itself.
Moreover, Colmez gave a conjectural formula for this Faltings height in terms of logarithmic derivatives of Artin $L$-functions.
In the special case where $E=\kk$, Colmez's conjecture reduces to the well-known Chowla-Selberg formula
\begin{equation}\label{chowla-selberg}
h^\Falt_\kk   = - \frac{1}{2}  \cdot  \frac{  \Lambda'(0,\chi_\kk)  }{ \Lambda(0,\chi_\kk) } - \frac{1}{4} \cdot  \log(16\pi^3e^\gamma),
\end{equation}
where we omit the CM type $\{ \mathrm{id} \} \subset \Hom(\kk,\C)$ from the notation.

Now suppose we are in the special case of \S \ref{ss:intro big GZ}, where \[E=\kk\otimes_\Q F\] and  $\Phi \subset \Hom(E,\C)$ has  signature $(n-1,1)$.   In this case, Colmez's conjecture simplifies to the equality of the following theorem.

\begin{bigtheorem}[\cite{Yang-Yin}]\label{Theorem C}
For a pair $(E,\Phi)$ as above,
\[
h^\Falt_{(E,\Phi)}  =
-\frac{2}{n}  \cdot \frac{\Lambda'(0, \chi_E) }{\Lambda(0, \chi_E)}
+  \frac{4-n}{2}  \cdot  \frac{\Lambda'(0, \chi_\kk) }{\Lambda(0, \chi_\kk)}
-  \frac{n}{4} \cdot \log(16 \pi^3 e^\gamma).
\]
\end{bigtheorem}

In \cite[\S 2.4]{BHKRY} we defined the line bundle of weight one modular forms $\bm{\omega}$ on $\mathcal{S}_\Kra^*$. 
 It was endowed it with a hermitian metric in \cite[\S 7.2]{BHKRY}, and the resulting metrized line bundle determines a class
\[
\widehat{\bm{\omega}}   \in \widehat{\mathrm{Ch}}_\Q^1( \mathcal{S}^*_\Kra ).
\]
The constant term of (\ref{modular divisors}) is
\begin{equation}\label{constant term}
\widehat{\mathcal{Z}}^\tot_\Kra(0) =  - \widehat{\bm{\omega}}   + ( \mathrm{Exc}, -\log(D))
\end{equation}
where $\mathrm{Exc}$ is the \emph{exceptional locus} of $\mathcal{S}_\Kra^*$ appearing in  \cite[Theorem 2.3.4]{BHKRY}. 
It is a smooth effective Cartier divisor supported in characteristics dividing $D$, and we view it as an arithmetic divisor by endowing it with the constant Green function $-\log(D)$ in the complex fiber.

\begin{bigtheorem}\label{Theorem D}
The metrized line bundle $\widehat{\bm{\omega}}$ satisfies
\[
[ \widehat{\bm{\omega}}  : \mathcal{Y}_\mathrm{big} ]
=  \frac{- 2}{n}\cdot  \deg_\C(\mathcal{Y}_\mathrm{big})  \cdot \frac{  \Lambda'(0,\chi_E)}{\Lambda(0,\chi_E)} .
\]
\end{bigtheorem}

Theorem \ref{Theorem C} is proved   in \cite{Yang-Yin}  as a consequence of the average version of Colmez's  conjecture  \cite{AGHMP-2,YZ, Ho3}.
 Note that the proof  in \cite{Yang-Yin}  does not require our standing hypothesis that $\mathrm{disc}(\kk)$ is odd.  Of course the assumption that $\mathrm{disc}(\kk)$ is odd is still needed for Theorem \ref{Theorem D}, as it is only under these hypotheses that we have even defined the integral model $\mathcal{S}_\Kra^*$ and its line bundle of weight one modular forms.

 In  \S \ref{s:colmez} we will show that Theorems \ref{Theorem C} and  \ref{Theorem D} are equivalent.
 One can interpret this in one of two ways.  As  Theorem \ref{Theorem C} is already known, this equivalence proves Theorem D.
 On the other hand,   in  \S \ref{ss:big derivative} will give an independent proof of Theorem \ref{Theorem D}  under the additional assumption that   the discriminants of $\kk$ and $F$ are odd and relatively prime.  In this way we obtain a new proof of Theorem \ref{Theorem C} under these extra hypotheses.


\subsection{The case $n=2$}

Throughout the introduction we have assumed that $n\ge 3$, and the reader might wonder how much of what we have written extends to the case $n=2$.

As explained in \cite[\S 1.6]{BHKRY},  when $n=2$ the proof of the modularity of  (\ref{modular divisors}) breaks down because there is no known integral model of $\mathrm{Sh}(G,\mathcal{D})$ whose reduction at the primes of $\co_\kk$ dividing $D$ is normal.  The existence of such a model when $n>2$ is used in [\emph{loc.~cit.}]  to compute the vertical components of divisors of Borcherds products.

When $n=2$, the Shimura variety $\mathrm{Sh}(G,\mathcal{D})$ is essentially a union of modular curves (if the $\kk$-hermitian space $W$ admits an isotropic line) or compact quaternionic Shimura curves (if $W$ is anisotropic).
In either case the analogues of Theorems \ref{Theorem A} and \ref{Theorem B}  are close in spirit to the Gross-Zagier theorem \cite{GZ} and its generalizations \cite{YZZ}.
In particular, the statement of Theorems \ref{Theorem A} is quite parallel to 
the key result Theorem 6.1 in  \cite[Section~1.6]{GZ}.
If we interchange in the computation of  $[ \widehat{\theta}(g ) : \mathcal{Y}_\mathrm{sm} ]$ the order of taking the Petersson inner product and the height pairing, this quantity is very analogous to the left hand side of Theorem 6.1 in \cite{GZ}. Both quantities are expressed as central derivatives of a Rankin convolution $L$-function of $g$ and a binary theta function which is determined by the CM cycle in question. If $g$ is a newform, then $\widehat{\theta}(g )$ should lie in a $g$-isotypical component and the height pairing in our Theorem 
 \ref{Theorem A} should be proportional to the height of the $g$-isotypical component of (a twist of) $\mathcal{Y}_\mathrm{sm}$. It would be interesting to make such a comparison precise.
However, note that there are substantial differences as well.
While we work with unitary Shimura varieties and CM points whose discriminants are equal to the level, Gross and Zagier work with $\GL_2$ Shimura varieties and CM points whose discriminants 
are coprime to the level. 
%

Theorem \ref{Theorem C} is true as stated when $n=2$, and is proved in \cite{Yang-Yin}.  
Indeed, Colmez's conjecture is known for all quartic CM fields.  
If the quartic CM field is Galois over $\Q$, then the Galois group is abelian and Colmez's conjecture is known by work of Colmez \cite{Colmez} and Obus \cite{Obus}.  
In the non-Galois case the CM types form a single $\mathrm{Aut}(\C/\Q)$-orbit;  as Colmez's conjecture is constant on such orbits,   
 the full Colmez conjecture follows from the average case proved in \cite{AGHMP-2} and \cite{YZ}.

Theorem \ref{Theorem D} is also true as stated when $n=2$.
Indeed,  when we prove the equivalence of Theorems \ref{Theorem C} and \ref{Theorem D}  in \S \ref{s:colmez}  we only assume $n\ge 2$.
 

\subsection{Thanks} The results of this paper are the outcome of a long term project, begun initially in Bonn in June of 2013, and supported in a crucial way 
by three weeklong meetings at AIM, in Palo Alto (May of 2014) and San Jose (November of 2015 and 2016), as part of their AIM SQuaRE's program. The opportunity to spend these periods of intensely focused 
efforts on the problems involved was essential. We would like to thank the University of Bonn and AIM for their support. 



\section{Small CM cycles and derivatives of $L$-functions}
\label{s:small GZ}


In this section we combine the results of \cite{BHKRY} and \cite{BHY} to prove Theorem \ref{Theorem A}.
Although we will restrict to $n\ge 3$ in \S \ref{ss:Aproof}, we allow $n\ge 2$ until that point.


\subsection{A Shimura variety of dimension zero}


Define a rank three torus $T_\mathrm{sm}$ over $\Q$ as the fiber product
\[
\xymatrix{
{T_\mathrm{sm}}  \ar[rrr]\ar[d]  &  &  & {  \mathbb{G}_m } \ar[d]^{\mathrm{diag.}} \\
{\mathrm{Res}_{\kk/\Q}\mathbb{G}_m \times \mathrm{Res}_{\kk/\Q}\mathbb{G}_m  } \ar[rrr]_{\mathrm{Nm} \times \mathrm{Nm}}  &   &   &{  \mathbb{G}_m \times \mathbb{G}_m. }
}
\]
Its group of $\Q$-points is
\[
T_\mathrm{sm}(\Q) \iso \{ (x,y) \in \kk^\times \times \kk^\times : x\overline{x}=y\overline{y} \}.
\]

The fixed embedding $\kk \subset \C$ identifies Deligne's torus $\mathbb{S}$ with the real algebraic group
$(\mathrm{Res}_{\kk/\Q}\mathbb{G}_m)_\R$, and   the  diagonal inclusion
\[
\mathbb{S} \hookrightarrow  (\mathrm{Res}_{\kk/\Q}\mathbb{G}_m)_\R \times ( \mathrm{Res}_{\kk/\Q}\mathbb{G}_m )_\R
\]
factors through a morphism
$
h_\mathrm{sm} : \mathbb{S} \to T_{\mathrm{sm} , \R }.
$
The pair $(T_\mathrm{sm},\{h_\mathrm{sm} \})$ is  a Shimura datum, which, along with the compact open subgroup
\[
K_\mathrm{sm} = T_\mathrm{sm}(\A_f) \cap ( \widehat{\co}_\kk^\times \times \widehat{\co}_\kk^\times),
\]
determines a $0$-dimensional $\kk$-stack  $\mathrm{Sh}(T_\mathrm{sm})$ with complex points
\[
\mathrm{Sh}(T_\mathrm{sm}) (\C) = T_\mathrm{sm}(\Q) \backslash \{ h_\mathrm{sm}\} \times T_\mathrm{sm}(\A_f) / K_\mathrm{sm}.
\]


\subsection{The small CM cycle}
\label{ss:small cycle construction}


The Shimura  variety just constructed  has  a moduli interpretation,  which allows us to construct an integral model.
The interpretation we have in mind requires first choosing  a triple $(\mathfrak{a}_0,\mathfrak{a}_1,\mathfrak{b})$ in which
\begin{itemize}
\item
$\mathfrak{a}_0$ is a  self-dual hermitian $\co_\kk$-lattice of signature $(1,0)$,
\item
$\mathfrak{a}_1$ is a  self-dual hermitian $\co_\kk$-lattice of signature $(0,1)$,
\item
$\mathfrak{b}$ is a  self-dual hermitian $\co_\kk$-lattice of signature $(n-1,0)$.
\end{itemize}
The hermitian forms on $\mathfrak{a}_0$ and $\mathfrak{b}$ induce a hermitian form of signature $(n-1,0)$ on the   projective $\co_\kk$-module
\[
\Lambda = \Hom_{\co_\kk} (  \mathfrak{a}_0 , \mathfrak{b}   ),
\]
as explained in   \cite[\S 2.1]{BHY} or  \cite[(2.1.5)]{BHKRY}.

Recall from \cite[\S 3.1]{BHY} or \cite[\S 2.3]{BHKRY} the $\co_\kk$-stacks  $\mathcal{M}_{(p,0)}$ and $\mathcal{M}_{(0,p)}$.
Both  parametrize  abelian schemes  $A\to S$ of relative dimension $p\ge 1$ over $\co_\kk$-schemes, endowed with principal polarizations and $\co_\kk$-actions.   For the first moduli problem we impose the  signature $(p,0)$ condition that $\co_\kk$ acts on the $\co_S$-module $\Lie(A)$ via the structure morphism $\co_\kk \to \co_S$.  For the second we impose the signature $(0,p)$ condition that the action is by  the complex conjugate of the structure morphism.  Both of these stacks are \'etale and proper over $\co_\kk$ by \cite[Proposition 2.1.2]{Ho2}.

\begin{remark}
The generic fibers of $\mathcal{M}_{(1,0)}$ and $\mathcal{M}_{(0,1)}$ are the Shimura varieties associated to  $\mathfrak{a}_{0\Q}$ and $\mathfrak{a}_{1\Q}$, while the generic fiber of $\mathcal{M}_{(n-1,0)}$ contains the Shimura variety associated to $\mathfrak{b}_\Q$ as an open and closed substack.  
For more precise information, see \cite[Proposition 2.13]{KR2} and the lemma that precedes it.
\end{remark}

Denote by $\widetilde{\mathcal{Y}}_\mathrm{sm}$ the functor that associates to every $\co_\kk$-scheme $S$ the groupoid of quadruples $(A_0,A_1,B,\eta)$ in which
\begin{equation}\label{presmall}
( A_0,A_1,B ) \in  \mathcal{M}_{(1,0)}(S)  \times  \mathcal{M}_{(0,1)}(S)  \times  \mathcal{M}_{(n-1,0)} (S),
\end{equation}
and
\begin{equation}\label{eta level}
\eta : \underline{\Hom}_{\co_\kk}(A_0,B) \iso \underline{\Lambda}
\end{equation}
is an isomorphism of \'etale sheaves of hermitian $\co_\kk$-modules, where the hermitian form on the left hand side is defined as in \cite[(2.5.1)]{BHKRY}.
We impose the further condition  that   for every geometric point $s  \to S$,  and every prime $\ell\neq \mathrm{char}(s)$, there is an isomorphism of hermitian $\co_{\kk,\ell}$-lattices
 \begin{equation}\label{small genus}
 \Hom_{\co_\kk} ( A_{0 s} [\ell^\infty]  ,  A_{1s} [\ell^\infty] ) \iso \Hom_{\co_\kk}( \mathfrak{a}_0,\mathfrak{a}_1 ) \otimes_\Z \Z_\ell.
 \end{equation}

 \begin{lemma}\label{lem:small genus complete}
 If
 \[
 s \to \mathcal{M}_{(1,0)}  \times_{\co_\kk}  \mathcal{M}_{(0,1)}  \times_{\co_\kk}  \mathcal{M}_{(n-1,0)}
 \]
is a geometric point of characteristic $0$ such that (\ref{small genus}) holds for all primes $\ell$ except possibly one, then it holds for the remaining prime as well.
 \end{lemma}

\begin{proof}
The proof is identical to  \cite[Lemma 2.2.2]{BHKRY}.
\end{proof}

\begin{proposition}\label{prop:representability}
The functor $\widetilde{\mathcal{Y}}_\mathrm{sm}$ is represented by a Deligne-Mumford stack, \'etale and proper over $\co_\kk$,
and there is a canonical isomorphism of $\kk$-stacks
\begin{equation}\label{small uniformization}
 \mathrm{Sh}(T_\mathrm{sm}) \iso \widetilde{\mathcal{Y}}_{\mathrm{sm}/\kk} .
\end{equation}
\end{proposition}

\begin{proof}
For any $\co_\kk$-scheme $S$, let  $\mathcal{N}(S)$ be the groupoid of triples (\ref{presmall}) satisfying (\ref{small genus}) for every geometric point $s\to S$ and every prime $\ell\neq \mathrm{char}(s)$.  In other words, the definition is the same as $\widetilde{\mathcal{Y}}_\mathrm{sm}$ except that we omit the datum (\ref{eta level}) from the moduli problem.

We interrupt the proof of Proposition \ref{prop:representability} for a lemma.

\begin{lemma}\label{lem:small cm components}
The functor $\mathcal{N}$ is represented by an open and closed substack
\[
\mathcal{N} \subset  \mathcal{M}_{(1,0)}  \times_{\co_\kk}  \mathcal{M}_{(0,1)}  \times_{\co_\kk}  \mathcal{M}_{(n-1,0)}.
\]
\end{lemma}

\begin{proof}
This is \cite[Proposition 5.2]{BHY}.  As the proof  there is left to the reader, we indicate the idea. Let
 \[
 \mathcal{B} \subset  \mathcal{M}_{(1,0)}  \times_{\co_\kk}  \mathcal{M}_{(0,1)}  \times_{\co_\kk}  \mathcal{M}_{(n-1,0)}
\]
be one connected component, and suppose there is a geometric point  $s\to \mathcal{B}$ of characteristic $p$ such that (\ref{small genus}) holds for all $\ell\neq p$.  The geometric fibers of the $\ell$-adic sheaf  $\underline{\Hom}_{\co_\kk} ( A_0 [\ell^\infty]  ,  A_1 [\ell^\infty] )$ on
\[
\mathcal{B}_{(p)} = \mathcal{B} \times_{\Spec(\Z)} \Spec(\Z_{(p)})
\]
are all isomorphic,   and therefore  (\ref{small genus}) holds for all geometric points $s\to \mathcal{B}_{(p)}$ and all $\ell\neq p$.
In particular, using Lemma \ref{lem:small genus complete}, if $s \to \mathcal{B}$ is a geometric point of characteristic $0$, then (\ref{small genus}) holds for every prime $\ell$.  Having proved this, one can reverse the argument to see that (\ref{small genus}) holds for \emph{every} geometric point $s\to \mathcal{B}$ and every $\ell \neq \mathrm{char}(s)$.   Thus if the condition (\ref{small genus}) holds at one geometric point, it holds at all geometric points on the same connected component.
\end{proof}

We now return to the proof of Proposition \ref{prop:representability}.
As noted above, the stacks $\mathcal{M}_{(p,0)}$ and $\mathcal{M}_{(0,p)}$ are \'etale and proper over $\co_\kk$, and hence the
same is true of $\mathcal{N}$.

Let $(A_0,A_1,B)$ be the universal object over $\mathcal{N}$.  Combining \cite[Theorem 5.1]{BHY} and \cite[Corollary 6.9]{hida}, the \'etale sheaf $\underline{\Hom}_{\co_\kk}(A_0,B)$  is represented by a Deligne-Mumford stack whose connected components
are finite \'etale over  $\mathcal{N}$.  Fixing a geometric point $s\to \mathcal{N}$, we obtain a representation of $\pi_1^{et}(\mathcal{N},s)$ on a finitely generated $\co_\kk$-module $\Hom_{\co_\kk}(A_{0s},B_s)$, and the kernel of this representation cuts out a
finite \'etale cover  $\mathcal{N}' \to \mathcal{N}$
over which the sheaf $\underline{\Hom}_{\co_\kk}(A_0,B)$ becomes constant.

It is now easy to see that the functor $\widetilde{\mathcal{Y}}_\mathrm{sm}$ is represented by the disjoint union of finitely many copies of the maximal open and closed substack of $\mathcal{N}'$ over which there exists an isomorphism (\ref{eta level}).

It remains to construct the isomorphism (\ref{small uniformization}).
The natural actions of $\co_\kk$ on $\mathfrak{a}_0$ and $\mathfrak{b}$, along with the \emph{complex conjugate} of the natural action of $\co_\kk$ on $\mathfrak{a}_1$, determine a morphism of reductive groups
\[
  \mathrm{Res}_{\kk/\Q}\mathbb{G}_m \times \mathrm{Res}_{\kk/\Q}\mathbb{G}_m
\map{( w,z) \mapsto (w, \overline{z},z) } \GU(\mathfrak{a}_{0\Q}) \times \GU(\mathfrak{a}_{1\Q}) \times \GU(\mathfrak{b}_\Q ).
\]
Restricting this morphism to the subtorus $T_\mathrm{sm}$ defines a morphism
\[
\mathbb{S} \map{ h_\mathrm{sm}  }    T_{\mathrm{sm},\R}
\to   \GU(\mathfrak{a}_{0\R}) \times \GU(\mathfrak{a}_{1\R}) \times \GU(\mathfrak{b}_\R ),
\]
endowing  the real vector spaces $\mathfrak{a}_{0\R}$, $\mathfrak{a}_{1\R}$, and $\mathfrak{b}_\R$ with  complex  structures.

The isomorphism (\ref{small uniformization}) on complex points sends a pair
\[
( h_\mathrm{sm} , g) \in \mathrm{Sh}(T_\mathrm{sm})(\C)
\]
 to the quadruple  $(A_0,A_1,B,\eta)$ defined by
\[
A_0(\C) = \mathfrak{a}_{0\R} /  g \mathfrak{a}_0,\quad
A_1(\C) = \mathfrak{a}_{1\R} /g  \mathfrak{a}_1,\quad
B(\C) = \mathfrak{b}_{\R} / g \mathfrak{b} ,
\]
endowed with their natural $\co_\kk$-actions and polarizations as in the proof of \cite[Proposition 2.2.1]{BHKRY}.  The datum $\eta$ is the canonical identification
\[
\Hom_{\co_\kk}(A_0,B) = \Hom_{\co_\kk}(g \mathfrak{a}_0,g \mathfrak{b}) =\Hom_{\co_\kk}( \mathfrak{a}_0, \mathfrak{b})  = \Lambda.
\]
It follows from the theory of canonical models that this isomorphism on complex points descends to an isomorphism of $\kk$-stacks, completing the proof of Proposition \ref{prop:representability}.
\end{proof}


The finite  group $\Aut(\Lambda)$ of automorphisms of the hermitian lattice $\Lambda$ acts on $\widetilde{\mathcal{Y}}_\mathrm{sm}$ by
\[
\gamma* (A_0,A_1,B,\eta) = (A_0,A_1,B,\gamma\circ \eta),
\]
 allowing us to form the  stack quotient
$
\mathcal{Y}_\mathrm{sm} = \Aut(\Lambda) \backslash\widetilde{\mathcal{Y}}_\mathrm{sm}.
$
The forgetful map
\[
\widetilde{\mathcal{Y}}_\mathrm{sm}  \to   \mathcal{M}_{(1,0)} \times \mathcal{M}_{(0,1)} \times \mathcal{M}_{(n-1,0)}
\]
 (all fiber products over $\co_\kk$)  factors through an open and closed immersion
 \[
 \mathcal{Y}_\mathrm{sm} \to \mathcal{M}_{(1,0)} \times \mathcal{M}_{(0,1)} \times \mathcal{M}_{(n-1,0)} 
 \]
whose image is the open and closed substack $\mathcal{N}$ of Lemma \ref{lem:small cm components}.

The triple $(\mathfrak{a}_0,\mathfrak{a}_1,\mathfrak{b})$ determines a pair $(\mathfrak{a}_0,\mathfrak{a})$ as in the introduction, simply by setting $\mathfrak{a}=\mathfrak{a}_1\oplus \mathfrak{b}$.  This data determines a unitary Shimura variety with  integral model $\mathcal{S}_\Kra$ as in  (\ref{kramer inclusion}), and there is a commutative diagram
\[
\xymatrix{
{ \mathcal{Y}_\mathrm{sm}  }  \ar[rr] \ar[d]_\pi  & &   {  \mathcal{M}_{(1,0)} \times \mathcal{M}_{(0,1)} \times \mathcal{M}_{(n-1,0)} }   \ar[d]   \\
 { \mathcal{S}_\Kra  } \ar[rr]^{\subset} & & {  \mathcal{M}_{(1,0)} \times \mathcal{M}_{(n-1,1)}^\Kra } .
}
\]
The vertical arrow on the right sends
\[
(A_0,A_1,B) \mapsto (A_0, A_1\times B),
\]
 and the arrow $\pi$ is defined by the commutativity of the diagram.

\begin{remark}
In order for $A_1\times B$ to define a point of $\mathcal{M}_{(n-1,1)}^\Kra$,  we must endow its Lie algebra  with a codimension one subsheaf  \[
\mathcal{F}_{A_1\times B} \subset \Lie(A_1\times B)
\]
 satisfying Kr\"amer's condition \cite[\S 2.3]{BHKRY}.
We choose
$
\mathcal{F}_{A_1\times B} = \Lie(B).
$
\end{remark}

\begin{definition}\label{def:small cm definition}
Composing the morphism $\pi$ in the diagram above with the inclusion of $\mathcal{S}_\Kra$ into its toroidal compactification, we obtain a morphism of $\co_\kk$-stacks
\[
\pi : \mathcal{Y}_\mathrm{sm} \to \mathcal{S}_\Kra^*
\]
called   the  \emph{small CM cycle}.
\end{definition}

As in \cite[Definition 3.1.8]{Ho2}, there is a linear functional
\[
\widehat{\mathrm{Ch}}^1_\C(\mathcal{S}^*_\Kra) \to \C
\]
called  \emph{arithmetic degree along} $\mathcal{Y}_\mathrm{sm}$ and denoted $\widehat{\mathcal{Z}} \mapsto [ \widehat{\mathcal{Z}} : \mathcal{Y}_\mathrm{sm} ]$,  defined as the composition 
\[
 \widehat{\mathrm{Ch}}^1_\C(\mathcal{S}^*_\Kra) \map{\pi^*}  \widehat{\mathrm{Ch}}^1_\C(\mathcal{Y}_\mathrm{sm}) \map{ \widehat{\deg} } \C .
\]
The first arrow is pullback of arithmetic divisors.  The second arrow  (\emph{arithmetic degree}) is normalized as follows:
An  irreducible  divisor $\mathcal{Z}\subset \mathcal{Y}_\mathrm{sm}$ is necessarily supported in finitely many nonzero characteristics, and hence any $\C$-valued function $\mathrm{Gr}(\mathcal{Z} , \cdot)$ on the finite set $\mathcal{Y}_\mathrm{sm}(\C)$ defines a Green function for it.   
The arithmetic degree of the arithmetic divisor 
\[
(\mathcal{Z},\mathrm{Gr}(\mathcal{Z} , \cdot)) \in  \widehat{\mathrm{Ch}}^1_\C(\mathcal{Y}_\mathrm{sm})
\]
is defined to be
\[
\widehat{\deg} (\mathcal{Z},\mathrm{Gr}(\mathcal{Z},\cdot) ) =
\sum_{ \mathfrak{q} \subset \co_\kk }
\sum_{z\in \mathcal{Z}(\F_\mathfrak{q}^\alg) } \frac{ \log ( \mathrm{N}(\mathfrak{q}))}{\#\Aut_\mathcal{X}(z)}
+   \sum_{z\in \mathcal{Y}_\mathrm{sm}(\C)} 
\frac{  \mathrm{Gr}(\mathcal{Z} , z) }{ \#\Aut_{\mathcal{Y}_\mathrm{sm}(\C) }(z) } ,
\]
where $\F_\mathfrak{q}^\alg$ is an algebraic closure
of $\co_\kk/\mathfrak{q}$, and $\mathrm{N}(\mathfrak{q}) = \# (\co_\kk/\mathfrak{q})$.

\begin{remark}\label{rem:no one-half}
The above definition of arithmetic degree does not include a  factor of $1/2$ in front of the archimedean contribution,
seemingly in disagreement with the usual definition  (see \cite[\S 3.4.3]{GS} for example).  In fact there is no disagreement.
Our convention is that $\mathcal{Y}_\mathrm{sm}(\C)$ means the complex points of $\mathcal{Y}_\mathrm{sm}(\C)$ as a $\kk$-stack, whereas in the usual definition it would be regarded as a $\Q$-stack.
Thus the usual definition includes a sum over  twice as many complex points, but with a $1/2$ in front.
\end{remark}

\begin{remark}
The small CM cycle  arises from a morphism of Shimura varieties.
Indeed, there is a  morphism of Shimura data $(T_\mathrm{sm},\{h_\mathrm{sm}\}) \to (G,\mathcal{D})$,
and the induced morphism of Shimura varieties sits in a commutative diagram
\[
\xymatrix{
{  \mathrm{Sh}(T_\mathrm{sm}) }  \ar[rr]  \ar[d]_\iso &  & { \mathrm{Sh}(G,\mathcal{D}) } \ar[d]^\iso  \\
{    \widetilde{\mathcal{Y}}_{ \mathrm{sm} /\kk}}  \ar[r]  & {  \mathcal{Y}_{\mathrm{sm}/\kk}  }  \ar[r]^\pi  & {  \mathcal{S}_{\Kra/\kk} } .
}
\]
\end{remark}

\begin{proposition}\label{prop:small degree}
The degree $\deg_\C (\mathcal{Y}_\mathrm{sm})$ of Theorem \ref{Theorem A} satisfies
\[
\deg_\C (\mathcal{Y}_\mathrm{sm})  =  ( h_\kk / w_\kk)^2  \cdot  \frac{2^{1-o(D)} }{  | \Aut (\Lambda) |},
\]
where $o(D)$ is the number of distinct prime divisors of $D$.
\end{proposition}

\begin{proof}
This is an elementary calculation.  Briefly, the groupoid $ \mathcal{Y}_\mathrm{sm}(\C)$ has $2^{1-o(D)} h_\kk^2$ isomorphism classes of points, and each point has  the same automorphism group $\co^\times_\kk \times \co_\kk^\times \times U(\Lambda)$.
\end{proof}

Recall from  (\ref{constant term}) that the constant term of (\ref{modular divisors}) is 
\[
\widehat{\mathcal{Z}}^\tot_\Kra(0) =  - \widehat{\bm{\omega}}   + ( \mathrm{Exc}, -\log(D)) 
\]
where $\widehat{\bm{\omega}}$ is the metrized line bundle of weight one modular forms.
The  exceptional locus $\mathrm{Exc} \subset \mathcal{S}_\Kra$ was defined in \cite[\S 2.3]{BHKRY}.   
It is a reduced effective Cartier divisor supported in characteristics dividing $D$, and can  be characterized as follows.
The integral model $\mathcal{S}_\Kra$ carries over it an abelian scheme $A \to \mathcal{S}_\Kra$ of relative dimension $n$ endowed with an action of $\co_\kk$.  This abelian scheme is obtained by pulling back  the universal object from the second factor of the fiber product in (\ref{kramer inclusion}).  If we let  $\delta\in \co_\kk$ be a fixed square root of $-D$, then $\mathrm{Exc}$ is the reduced stack underlying closed substack of $\mathcal{S}_\Kra$ defined by $\delta \cdot \Lie(A)=0$.

\begin{proposition}\label{prop:no error}
The constant term (\ref{constant term}) satisfies
\[
 [ \widehat{\mathcal{Z}}_\Kra^\tot(0) : \mathcal{Y}_\mathrm{sm} ] = - [\widehat{\bm{\omega}} : \mathcal{Y}_\mathrm{sm}]
=  2 \deg_\C(\mathcal{Y}_\mathrm{sm}) \cdot \frac{\Lambda'(0,\chi_\kk)}{ \Lambda(0,\chi_\kk) } .
\]
\end{proposition}

\begin{proof}
The second equality was proved in the course of proving \cite[Theorem 6.4]{BHY}.  We note that the argument  uses the Chowla-Selberg formula (\ref{chowla-selberg}) in an essential way.

The first equality is equivalent to
\[
[ ( \mathrm{Exc}  , - \log(D))   : \mathcal{Y}_\mathrm{sm} ]  = 0,
\]
and so it suffices to prove
\begin{equation}\label{error vanishes}
[ (0  ,  \log(D))   : \mathcal{Y}_\mathrm{sm} ]   =   \deg_\C(\mathcal{Y}_\mathrm{sm}) \cdot \log(D) = [ ( \mathrm{Exc}  , 0 )   : \mathcal{Y}_\mathrm{sm} ] .
\end{equation}
The first equality in (\ref{error vanishes}) is obvious from the definitions.
To prove the second equality, we first prove
\begin{equation}\label{small physical intersection}
 \mathcal{Y}_\mathrm{sm} \times_{\mathcal{S}_\Kra} \mathrm{Exc} =  \mathcal{Y}_{\mathrm{sm}} \times_{\Spec(\co_\kk)} \Spec( \co_\kk / \mathfrak{d}_\kk ) .
\end{equation}

As the exceptional locus $\mathrm{Exc} \subset \mathcal{S}_\Kra$ is reduced and supported in characteristics dividing $D$, it satisfies
\[
\mathrm{Exc} \subset \mathcal{S}_\Kra \times_{ \Spec(\co_\kk)  } \Spec( \co_\kk / \mathfrak{d}_\kk ).
\]
This implies  the inclusion $\subset$ in (\ref{small physical intersection}).
As $\mathcal{Y}_\mathrm{sm}$ is \'etale over $\co_\kk$, the right hand side of (\ref{small physical intersection}) is reduced, and hence so is the left hand side.
To prove that equality holds in (\ref{small physical intersection}), it now suffices  to check the inclusion $\supset$ on the level of geometric points.

As above, let  $\delta\in \co_\kk$ be a square root of $-D$.    Suppose $p\mid D$ is a prime, $\mathfrak{p}\subset \co_\kk$ is the unique prime above it, and $\F_\mathfrak{p}^\alg$ is an algebraic closure of its residue field.  Suppose we have a point
$y\in \mathcal{Y}_\mathrm{sm}(\F_\mathfrak{p}^\alg)$  corresponding to a triple $(A_0,A_1,B)$ over $\F_\mathfrak{p}^\alg$.
As $\delta=0$ in $\F_\mathfrak{p}^\alg$, the signature conditions imply that the endomorphism $\delta\in\co_\kk$ kills the Lie algebras of $A_0$, $A_1$, and $B$.
In particular $\delta$ kills the Lie algebra of $A_1\times B$, which is the pullback via
\[
\pi : \mathcal{Y}_\mathrm{sm} \to \mathcal{S}_\Kra
\]
of the universal $A \to \mathcal{S}_\Kra$.
Using the characterization of $\mathrm{Exc}$ recalled above, we find that that  $\pi(y) \in \mathrm{Exc}$.
This proves (\ref{small physical intersection}).

The equality (\ref{small physical intersection}), and the fact that both sides of that equality are reduced, implies that
\[
[ ( \mathrm{Exc}  , 0 )   : \mathcal{Y}_\mathrm{sm} ]  =
 \sum_{p\mid D } \log(p)  \sum_{ y\in \mathcal{Y}_\mathrm{sm}(\F_\mathfrak{p}^\alg)   }  \frac{1}{ |\Aut(y) | } .
\]
On the other hand, the \'etaleness of $\mathcal{Y}_\mathrm{sm} \to \Spec(\co_\kk)$ implies that the right hand side is equal to
\[
 \sum_{p\mid D } \log(p)  \sum_{ y\in \mathcal{Y}_\mathrm{sm}(\C)   }  \frac{1}{ |\Aut(y) | } = \log(D) \cdot \deg_\C(\mathcal{Y}_\mathrm{sm}),
\]
completing the proof of the second equality in (\ref{error vanishes}).
\end{proof}


\subsection{The convolution $L$-function}
\label{ss:convolution}


Recall that we have  defined a hermitian $\co_\kk$-lattice $\Lambda=\Hom_{\co_\kk}(\mathfrak{a}_0,\mathfrak{b})$ of signature $(n-1,0)$.
We  also define hermitian $\co_\kk$-lattices
\[
L_0=\Hom_{\co_\kk}(\mathfrak{a}_0,\mathfrak{a}_1) ,\quad L = \Hom_{\co_\kk}(\mathfrak{a}_0,\mathfrak{a}),
\]
of signature $(1,0)$ and $(n-1,1)$, so that
$
L \iso  L_0 \oplus \Lambda.
$

The hermitian form $\langle\cdot,\cdot\rangle : L\times L \to \co_\kk$ determines a $\Z$-valued quadratic form
$Q(x)=\langle x,x\rangle$ on $L$, and we denote in the same way its restrictions to $L_0$ and $\Lambda$.
The dual lattice of $L$ with respect to the $\Z$-bilinear form
\begin{equation}\label{bilinear}
[x_1,x_2]=Q(x_1+x_2) - Q(x_1) - Q(x_2)
\end{equation}
is $L'=\mathfrak{d}_\kk^{-1}L$.

As in \cite[\S 2.2]{BHY} we denote by   $S_L=\C[L'/L]$  the space of complex-valued functions on $L'/L$, and by
$
\omega_L : \SL_2(\Z) \to \Aut_\C (S_L)
$
the Weil representation. There is a complex conjugate representation $\overline{\omega}_L$ on $S_L$ defined by
\[
\overline{\omega}_L(\gamma)  \phi = \overline{  \omega_L(\gamma)\overline{\phi} }.
\]

Suppose we begin with  a  classical  scalar-valued  cusp form
\begin{equation*}
g (\tau)  = \sum_{m> 0} c(m)q^m \in  S_n( \Gamma_0(D) , \chi^n_\kk ),
\end{equation*}
Such a form determines a vector-valued form
\begin{equation}\label{vector induction}
\tilde{g}(\tau) = \sum_{\gamma\in \Gamma_0(D) \backslash \SL_2(\Z) } ( g |_n \gamma)(\tau) \cdot  \overline{ \omega_L(\gamma^{-1})\phi_0}
\in S_n( \overline{\omega}_L),
\end{equation}
where $\phi_0\in S_L$ is the characteristic function of the trivial coset.  This construction defines  the induction map (\ref{induction map}).
The form $\tilde{g}(\tau)$ has a $q$-expansion
\[
\tilde{g}(\tau) =  \sum_{m>0} \tilde{c}(m)  q^m
\]
with coefficients $\tilde{c}(m) \in S_L$.

There is  a similar Weil representation
$
\omega_\Lambda: \SL_2(\Z) \to \Aut_\C (S_\Lambda),
$
 and for every $m\in \Q$ we define a linear functional $R_\Lambda(m)\in S_\Lambda^\vee$ by
\[
R_\Lambda(m)(\phi) = \sum_{ \substack{  x\in  \Lambda'  \\ \langle x,x\rangle =m  }} \phi(x)
\]
where $\phi \in S_\Lambda$ and $\langle \cdot,\cdot\rangle:\Lambda_\Q \times \Lambda_\Q \to \kk$ is the $\Q$-linear extension of the hermitian form on $\Lambda$.  The theta series
\[
\theta_\Lambda(\tau) = \sum_{m\in \Q} R_\Lambda(m) q^m \in M_{n-1}( \omega_\Lambda^\vee)
\]
is a modular form valued in the contragredient representation $S_\Lambda^\vee$.

As in \cite[\S 5.3]{BHY} or \cite[\S 4.4]{BY},  we define the \emph{Rankin-Selberg convolution $L$-function}
\begin{align}
\label{eq:defl}
 L( \tilde{g}, \theta_\Lambda,s)
  = \Gamma \left(  \frac{s}{2}  + n -1   \right)     \sum_{m\ge 0} \frac{    \{   \overline{   \tilde{c}(m) } , R_\Lambda(m)   \}    }{ (4\pi m)^{ \frac{s}{2}  + n -1 }}.
\end{align}
Here $\{ \cdot\, ,\cdot\} :S_L \times S_L^\vee \to \C$ is the tautological pairing.
The inclusion \[ \Lambda' / \Lambda \to  L'  /L\]  induces a linear map  $S_L\to S_\Lambda$
by restriction of functions, and we use the dual $S_\Lambda^\vee \to S_L^\vee$ to view $R_\Lambda(m)$ as an element of $S_L^\vee$.

 \begin{remark}
 The convolution $L$-function satisfies a functional equation in $s\mapsto -s$, forcing $L(  \tilde{g} ,  \theta_\Lambda,0)=0.$
\end{remark}

 \begin{remark}
 In this generality, neither  the cusp form $g$ nor the theta series $\theta_\Lambda$ is a Hecke eigenform.
Thus the convolution $L$-function (\ref{eq:defl}) cannot be expected to have an Euler product expansion.
\end{remark}


\subsection{A preliminary central derivative formula}
\label{ss:first small derivative}


We now recall  the main result of \cite{BHY}, and explain the connection between the cycles and Shimura varieties here and in that work.

Define hermitian $\widehat{\co}_\kk$-lattices
\[
\mathbb{L}_{0,f}    = \Hom_{\co_\kk}( \mathfrak{a}_0 , \mathfrak{a}_1) \otimes_{\Z} \widehat{\Z}   ,\quad
\mathbb{L}_f      = \Hom_{\co_\kk}( \mathfrak{a}_0 , \mathfrak{a}) \otimes_{\Z} \widehat{\Z} ,
\]
and let $\mathbb{L}_{0,\infty}$ and $\mathbb{L}_\infty$ be $\kk_\R$-hermitian spaces of signatures $(1,0)$ and $(n,0)$, respectively.
In the terminology of \cite[\S 2.1]{BHY}, the pairs
\[
\mathbb{L}_0=(\mathbb{L}_{0,\infty} ,\mathbb{L}_{0,f}),\quad
\mathbb{L}=(\mathbb{L}_\infty,\mathbb{L}_f)
\]
are \emph{incoherent hermitian $(\kk_\R, \widehat{\co}_\kk)$-modules}.
Our  small CM cycle is related to the cycle of \cite[\S 5.1]{BHY} by
\[
\xymatrix{
{   \mathcal{Y}_\mathrm{sm}  } \ar[r]\ar@{=}[d]    & {  \mathcal{S}_\Kra }  \ar@{=}[d]   \\
{  \mathcal{Y}_{ (\mathbb{L}_0,\Lambda) } }  \ar[r]    &   {   \mathcal{M}_\mathbb{L} ,}
}
\]
and the metrized line bundle  $\widehat{\bm{\omega}}^{-1}$ of \cite{BHKRY} agrees with the  metrized cotautological bundle $\widehat{\mathcal{T}}_{\mathbb{L}}$ of \cite{BHY}.

Let $\Delta$ be the automorphism group of the finite abelian group $ L'/L$ endowed with the quadratic form
$
 L'/L \to \Q/\Z
$
obtained by reduction of  $Q : L \to \Z$.  The tautological action of $\Delta$ on  $S_L=\C[L'/L]$ commutes with the Weil representation $\omega_L$,
and hence $\Delta$ acts on all spaces of  modular forms valued in the representation $\omega_L$.

Let $H_{2-n}(\omega_L)$ be the space of harmonic Maass forms of \cite[\S 2.2]{BHY}.
Every $f\in H_{2-n}(\omega_L)$ has a \emph{holomorphic part}
\[
f^+(\tau) = \sum_{ \substack{  m\in \Q \\  m \gg -\infty  } } c_f^+(m)\cdot q^m ,
\]
which is a formal $q$-expansion with coefficients in $S_L$.  Let  $c_f^+(0,0)$ be the value of  $c_f^+(0)\in S_L$ at the trivial coset. 

As in  \cite{BF} or \cite[\S 3.1]{BY}, there is a $\Delta$-equivariant, surjective, conjugate linear differential operator
\[
\xi : H_{2-n}(\omega_L) \to S_n(\overline{\omega}_L),
\]
and the construction of  \cite[(4.15)]{BHY} defines a linear functional
\begin{equation}\label{fake theta}
\widehat{\mathcal{Z}}  :  H_{2-n}(\omega_L)^\Delta \to \widehat{\mathrm{Ch}}_\C^1(\mathcal{S}^*_\Kra).
\end{equation}
These  are related by the main result of \cite{BHY}, which we now state.

\begin{theorem}[\cite{BHY}] \label{thm:BHY main}
The equality
\[
[  \widehat{\mathcal{Z}} ( f )  :  \mathcal{Y}_\mathrm{sm}  ]   -   c_f^+(0,0) \cdot  [  \widehat{\bm{\omega}} : \mathcal{Y}_\mathrm{sm} ]
= - \deg_\C( \mathcal{Y}_\mathrm{sm} )  \cdot  L'( \xi (f) , \theta_\Lambda,0)
\]
holds for any $\Delta$-invariant $f\in H_{2-n}(\omega_L)$.
\end{theorem}


\subsection{The proof of Theorem \ref{Theorem A}}
\label{ss:Aproof}


Throughout \S \ref{ss:Aproof} we assume $n\ge 3$.
Under this assumption the linear functional (\ref{fake theta}) is closely related to the coefficients of the generating series (\ref{modular divisors}).  Indeed,  If $m$ is a positive integer,   \cite[Lemma 3.10]{BHY} shows that there  is a unique
\[
f_m \in H_{2-n}(\omega_L)^\Delta
\]
  with holomorphic part
\begin{equation}\label{simple harmonic}
f_m^+(\tau) = \phi_0 \cdot q^{-m} + O(1),
\end{equation}
where $\phi_0 \in S_L$ is the characteristic function of the trivial coset.
Applying the above linear functional to this form recovers the $m^\mathrm{th}$ coefficient
\[\widehat{\mathcal{Z}}_\Kra^\tot ( m ) =\widehat{\mathcal{Z}} ( f_m ) \]
of the generating series (\ref{modular divisors}).

The following proposition explains the connection between the linear functional (\ref{fake theta}) and the arithmetic theta lift  (\ref{ATL}).

\begin{proposition}\label{prop:good harmonic choice}
For every    $g \in S_n(\Gamma_0(D) , \chi^n_\kk )$  there is a $\Delta$-invariant form $f\in H_{2-n}(\omega_L)$ such that
\begin{equation}\label{harmonic theta}
\widehat{\theta}(g) =  \widehat{\mathcal{Z}} ( f )   + c_f^+(0,0) \cdot   \widehat{\mathcal{Z}}_\Kra^\tot(0) ,
\end{equation}
and  such that $\xi(f)$ is equal to the form  $\tilde{g} \in S_n(\overline{\omega}_L)$ defined by  (\ref{vector induction}).
Moreover, we may choose $f$ to be a linear combination of the forms $f_m$ characterized by (\ref{simple harmonic}).
\end{proposition}

\begin{proof}
Consider the space $H_{2-n}^\infty(\Gamma_0(D) , \chi^n_\kk )$ of harmonic Maass forms of \cite[\S 7.2]{BHKRY}.
The constructions of \cite{BF} provide us with a surjective conjugate linear differential operator
\[
\xi: H_{2-n}^\infty(\Gamma_0(D) , \chi^n_\kk )\to S_n(\Gamma_0(D) , \chi^n_\kk ),
\]
and we choose an $f_0\in H_{2-n}^\infty(\Gamma_0(D) , \chi^n_\kk )$ such that $\xi(f_0)=g$.
 It is easily seen that $f_0$ may be chosen to  vanish at all cusps of $\Gamma_0(D)$ different from $\infty$.  
 This can, for instance, be attained by adding a suitable weakly holomorphic form in the space $M_{2-n}^{!,\infty}(\Gamma_0(D) , \chi^n_\kk )$ of \cite[\S 4.2]{BHKRY}.
The Fourier expansion of the holomorphic part of $f_0$ is denoted
\[
f_0^+(\tau) = \sum_{m\in \Q} c_0^+(m) q^m.
\]

As in (\ref{vector induction}), the form $f_0$ determines an $S_L$-valued harmonic Maass form
\[
f(\tau) = \sum_{\gamma\in \Gamma_0(D) \backslash \SL_2(\Z) } ( f_0 |_{2-n} \gamma)(\tau) \cdot   \omega_L(\gamma^{-1})\phi_0
\in H_{2-n}(\omega_L)^\Delta .
\]
As the $\xi$-operator is equivariant for the action of $\SL_2(\Z)$, we have $\xi(f)=\tilde{g}$.
According to \cite[Proposition 6.1.2]{BHKRY}, which holds analogously for harmonic Maass forms, the coefficients of the holomorphic part $f^+$ satisfy
\[
c_f^+(m,\mu) = \begin{cases}
c_0^+(m) & \mbox{if }\mu=0 \\
0 & \mbox{otherwise}
\end{cases}
\]
for all $m\le 0$.  This equality  implies that
\[
f = \sum_{m>0} c^+_0(-m) f_m,
\]
where $f_m \in H_{2-n}(\omega_L)^\Delta$ is the harmonic form characterized by  (\ref{simple harmonic}).
Indeed, the difference between the two forms is a harmonic form $h$ whose holomorphic part $\sum_{m\ge 0} c_h^+(m) q^m$
has no principal part.  It follows from \cite[Theorem 3.6]{BF} that such a harmonic form is actually holomorphic, and therefore vanishes because the weight  is negative.

The above decomposition of $f$ as a linear combination of the $f_m$'s implies that
\[
\widehat{\mathcal{Z}}  (f) = \sum_{m>0} c^+_0(-m) \cdot  \widehat{\mathcal{Z}}_\Kra^\tot(m) \in
\widehat{\mathrm{Ch}}_\C^1( \mathcal{S}_\Kra^* ),
\]
and consequently
\begin{align*}
\widehat{\theta}(g) 
&= \langle \widehat{\phi},\xi(f_0)  \rangle_\mathrm{Pet}\\
&=\{ f_0,\widehat{\phi}\}\\
&=\sum_{m\geq 0} c_0^+(-m) \cdot  \widehat{\mathcal{Z}}_\Kra^\tot ( m )\\
&=
\widehat{\mathcal{Z}} ( f )   + c_f^+(0,0) \cdot  \widehat{\mathcal{Z}}_\Kra^\tot ( 0 ) .
\end{align*}
Here, in the second line,  we have used the bilinear pairing
\[
\{\cdot,\cdot\}:H_{2-n}^\infty(\Gamma_0(D) , \chi^n_\kk )\times  M_n(\Gamma_0(D) , \chi_\kk^n ) \to \C
\]
analogous to \cite[Proposition 3.5]{BF}, and the fact that $f_0$ vanishes at all cusps different from $\infty$.
 \end{proof}

\begin{remark}
It is incorrectly claimed in \cite[\S 1.3]{BHY} that (\ref{harmonic theta}) holds for \emph{every} form $f$ with $\xi(f)=\tilde{g}$.
\end{remark}

The following is stated in the introduction as Theorem \ref{Theorem A}.

\begin{theorem}
If $g\in S_n(\Gamma_0(D) , \chi^n_\kk)$ and  $\tilde{g} \in S_n(\overline{\omega}_L)$ are related by (\ref{vector induction}), then
\[
[ \widehat{\theta}(g) : \mathcal{Y}_\mathrm{sm} ] = -  \deg_\C (\mathcal{Y}_\mathrm{sm}) \cdot  L'( \tilde{g} , \theta_\Lambda,0).
\]
\end{theorem}

\begin{proof}
Choosing $f$ as in Proposition \ref{prop:good harmonic choice}, and using the first equality of Proposition \ref{prop:no error},  yields
\[
[ \widehat{\theta}(g) : \mathcal{Y}_\mathrm{sm} ]
=  [ \widehat{\mathcal{Z}} ( f )  : \mathcal{Y}_\mathrm{sm} ]   - c_f^+(0,0) \cdot   [  \widehat{\bm{\omega}}  : \mathcal{Y}_\mathrm{sm} ] .
\]
Thus the  claim follows  from  Theorem \ref{thm:BHY main}.
\end{proof}


\section{Further results on the convolution $L$-function}
\label{s:convolution}


In this section we specialize to the case where $g\in S_n(\Gamma_0(D),\chi_\kk^n)$ is a new eigenform, and express the convolution $L$-function \eqref{eq:defl} associated to the vector valued cusp form $\tilde g$  in terms of the usual $L$-function associated to $g$.

This allows us, in Theorem \ref{maintheo2} below, to rewrite  Theorem \ref{Theorem A} of the introduction
in a way that avoids vector-valued modular forms.  When $n$ is even, it also allows us to formulate a version of 
Theorem \ref{Theorem A} in which the $L$-function has an Euler product.

We assume $n\ge 2$ until we reach \S \ref{ss:smallredux}, at which point we restrict to $n\ge 3$.


\subsection{Atkin-Lehner operators}


Recall that $\chi_\kk$ is the idele class character associated to the quadratic field $\kk$.
If we view $\chi_\kk$ as a Dirichlet character modulo $D$, then any factorization $D=Q_1 Q_2$ induces
a factorization
\[
\chi_\kk = \chi_{Q_1}\chi_{Q_2}
\]
where $\chi_{Q_i}:(\Z/Q_i\Z)^\times \to \C^\times$ is a quadratic Dirichlet character.

Fix a normalized  cuspidal new  eigenform 
 \[
 g(\tau) =\sum_{m>0} c(m) q^m \in S_{n}(\Gamma_0(D), \chi_\kk^n).
 \]
As in \cite[Section 4.1]{BHKRY}, for each positive divisor $Q\mid D$, fix a matrix
\[
R_Q =\left( \begin{matrix} \alpha & \beta \\ \frac{D}{Q} \gamma & Q \delta \end{matrix} \right) \in  \Gamma_0(D/Q)
\]
with $\alpha,\beta,\gamma,\delta \in\Z$, and  define the  Atkin-Lehner operator
\[
W_Q =\begin{pmatrix}  Q \alpha & \beta \\ D \gamma & Q \delta \end{pmatrix} =
R_Q \begin{pmatrix}  Q   \\  & 1 \end{pmatrix}.
\]
The cusp form
\begin{align*}
g_Q(\tau) &= \chi^n_Q(\beta) \chi_{D/Q}^n(\alpha)  \cdot   g|_n W_Q\\
& =\sum_{m>0} c_Q(m) q^m,
\end{align*}
is then independent of the choice of $\alpha$, $\beta$, $\gamma$, $\delta$.

Let $\epsilon_Q(g)$ be the fourth root of unity
\[
\epsilon_Q(g) = \prod_{\substack{q \mid Q\\\text{$q$ prime}}} \chi^n_Q( Q/q) \cdot \lambda_q,
\]
where
\[
 \lambda_q =
 \overline{c(q)} \cdot \begin{cases}
  -q^{1-\frac{n}2}  &\text{if $n \equiv 0 \pmod 2$}
  \\
   \delta_q q^{\frac{1-n}2}  &\text{if $n \equiv 1 \pmod 2$,}
   \end{cases}
\]
and $\delta_q$ is defined by
\begin{equation*}
\delta_q = \begin{cases}
1 & \mbox{if } q\equiv 1\pmod{4} \\
i & \mbox{if } q\equiv 3\pmod{4}.
\end{cases}
\end{equation*}
According to \cite[Theorem 2]{Asai}, we have
\begin{align*}
\nonumber
c_Q(m) &=\epsilon_Q(g)\chi_Q^n(m) c(m)  && \text{if  $(m, Q)=1$,}
\\
c_Q(m) &= \epsilon_Q(g)\chi_{D/Q}^n(m) \overline{c(m)} &&  \text{if $(m, D/Q)=1$,}\\
\nonumber
c_Q(m_1 m_2) &= \epsilon_Q(g)^{-1}c_Q(m_1) c_Q(m_2)  && \text{if  $(m_1, m_2)=1$.}
\end{align*}

\begin{remark}\label{rem:g_Q even}
If  $n$ is even, then the Fourier coefficients of $g$ are totally real. It follows that
$g_Q=\epsilon_Q(g)g$ for every divisor $Q\mid D$.  Furthermore,
\[
\epsilon_Q(g)=\prod_{q\mid Q}\big(-q^{1-\frac{n}{2} } c(q)\big) =\pm 1.
\]
\end{remark}


\subsection{Twisting theta functions}
\label{ss:theta twist}


Let $(\mathfrak{a}_0,\mathfrak{a}_1,\mathfrak{b})$ be a triple of self-dual hermitian $\co_\kk$-lattices of signatures $(1,0)$, $(0,1)$, and $(n-1,0)$, as in \S \ref{ss:small cycle construction}, and recall that from this data we constructed  hermitian $\co_\kk$-lattices 
\begin{equation}\label{aux lattices 1}
\mathfrak{a} = \mathfrak{a}_1\oplus \mathfrak{b},\quad 
 L=\Hom_{\co_\kk}(\mathfrak{a}_0, \mathfrak{a})
\end{equation}
of  signature $(n-1,1)$.  We also define 
\begin{equation}\label{aux lattices 2}
L_1=\Hom_{\co_\kk}(\mathfrak{a}_0,\mathfrak{a}_1),\quad 
\Lambda=\Hom_{\co_\kk}( \mathfrak{a}_0,\mathfrak{b}),
\end{equation}
 so that $L = L_1 \oplus \Lambda.$

Let $\GU(\Lambda)$ be the unitary  similitude group associated with  $\Lambda$, viewed as an algebraic group over $\Z$.
 For any $\Z$-algebra $R$ its $R$-valued points are given by
\[
\GU(\Lambda)(R)=\{ h\in \GL_{\calO_\kk}(\Lambda_R):\; \text{$\langle h x, h y\rangle = \nu(h)\langle x,y\rangle$ $\forall x,y\in \Lambda_R$}\},
\]
where $\nu(h)\in R^\times$ denotes the similitude factor of $h$.
Note the relation
\begin{align}
\label{eq:relnu}
\mathrm{Nm}_{\kk/\Q}(\det(h))=\nu(h)^{n-1}.
\end{align}
For $h\in \GU(\Lambda)(\R)$ the similitude factor $\nu(h)$ belongs to $\R_{>0}$.

As $\Lambda$ is positive definite, the set
\[
X_\Lambda = \GU(\Lambda)(\Q)\bs \GU(\Lambda)(\A_f)/ \GU(\Lambda)(\widehat\Z)
\]
is finite.  Denoting by 
\[
\CL(\kk)=\kk^\times\bs \widehat\kk^\times /\widehat \calO_\kk^\times
\] 
the ideal class group of $\kk$,  the natural map $\mathrm{Res}_{\kk/\Q} \mathbb{G}_m \to \GU(\Lambda)$
to the center induces an action 
\begin{align}
\label{eq:clact}
\CL(\kk)\times X_\Lambda \longrightarrow X_\Lambda.
\end{align}

As in the proof of \cite[Proposition 2.1.1]{BHKRY},    any $h\in\GU(\Lambda)(\A_f)$ determines an $\co_\kk$-lattice
\[
\Lambda_h= \Lambda_\Q \cap h\widehat\Lambda.
\]
This  lattice is not self-dual under the hermitian form $\langle-,-\rangle$ on $\Lambda_\Q$.
However,  there is  a unique positive rational number  $\mathrm{rat}( \nu(h) )$ such that 
\[
\frac{ \nu(h) } { \mathrm{rat}( \nu(h) )  }  \in   \widehat \Z^{\times},
\]
and the lattice $\Lambda_h$ is self-dual under the rescaled hermitian form 
\[
\langle x,y\rangle_{h} = \frac{1}{\mathrm{rat}( \nu (h))}  \cdot  \langle x,y\rangle.
\]
If $h\in \GU(\Lambda)(\widehat \Z)$  then $\Lambda_h=\Lambda$. If $h\in \GU(\Lambda)(\Q)$, then $\Lambda_h\cong \Lambda$ as hermitian $\co_\kk$-modules.
Hence  $h\mapsto \Lambda_h$ defines a function from $X_\Lambda$ to the set of isometry classes of  self-dual hermitian $\calO_\kk$-module of signature $(n-1,0)$.

 Similarly, for any $h\in\GU(\Lambda)(\A_f)$ we define a self-dual hermitian $\calO_\kk$-lattice of signature $(0,1)$ by endowing
\[
L_{1,h}= L_{1\Q} \cap \det(h) \widehat L_1
\]
with the hermitian form
\[
\langle x,y\rangle_{h} = \frac{1}{ \mathrm{rat}(  \nu(h) ) ^{n-1}} \cdot  \langle x,y\rangle.
\]
The assignment
$h\mapsto L_{1,h}$ defines a map from $X_\Lambda$ to the set of isometry classes of  self-dual hermitian $\calO_\kk$-lattices of signature $(0,1)$.

\begin{lemma} \label{lem:ideal twist}  
For any $h\in \GU(\Lambda)(\A_f )$ the hermitian $\co_\kk$-lattice 
\[
L_h=L_{1,h}\oplus \Lambda_h
\] 
is isomorphic everywhere locally to $L$.
Moreover,  $L_h$ and $L$ become isomorphism after tensoring with $\Q$.
\end{lemma}

\begin{proof}
Let $p$ be a prime. As in \cite[\S 1.8]{BHKRY}, a $\kk_p$-hermitian space  is determined by its
dimension and invariant. The relations
\begin{align*}
\det ( \Lambda_{h} \otimes_\Z\Q) &= \mathrm{rat}( \nu(h) ) ^{1-n}  \cdot \det (\Lambda \otimes_\Z\Q),\\
\det ( L_{1,h} \otimes_\Z\Q) &=  \mathrm{rat}(\nu(h))^{1-n}\cdot \det (L_1\otimes_\Z\Q),
\end{align*}
combined with
\eqref{eq:relnu}, imply that $L\otimes_\Z \Q$ and $L_h\otimes_\Z \Q$ have the same invariant everywhere locally.  As they both have signature $(n-1,1)$, they are isomorphic everywhere locally, and hence isomorphic globally.

A result of Jacobowitz \cite{Jac} shows that any two self-dual lattices  in $L \otimes_\Z\Q$ are  isomorphic everywhere locally, and hence it follows from the previous paragraph that $L$ and $L_h$ are isomorphic everywhere locally.
\end{proof}

Define a linear map 
\[
M_{n-1}(\omega_\Lambda^\vee) \to M_{n-1}(\Gamma_0(D),\chi_\kk^{n-1} )
\]
from $S_\Lambda^\vee$-valued modular forms to scalar-valued modular forms
by evaluation at the characteristic function $\phi_0\in S_\Lambda$ of the trivial coset $0\in \Lambda' / \Lambda$.  
This map takes
the vector valued theta series $\theta_\Lambda \in M_{n-1}(\omega_\Lambda^\vee)$ of \S \ref{ss:convolution}
to the scalar valued theta series
\[
\theta_\Lambda^\scal(\tau) = \sum_{m \in \Z_{\ge 0}} R_\Lambda^\scal(m) \cdot q^m,
\]
where $R_\Lambda^\scal(m)$  is the number of ways to represent $m$ by $\Lambda$.

Let $\eta$  be an algebraic automorphic form for $\GU(\Lambda)$ which  is trivial at $\infty$ and  right $\GU_\Lambda(\widehat\Z)$-invariant.  In other words,  a function
\[
\eta:X_\Lambda\longrightarrow \C.
\]
Throughout we assume that under the action \eqref{eq:clact}  the function $\eta$
transforms with a character $\chi_\eta:\CL(\kk)\to \C^\times$, that is,
\begin{align}
\label{eq:etamult}
\eta(\alpha h)=\chi_\eta(\alpha)\eta(h).
\end{align}

We associate a theta function to $\eta$ by setting
\[
\theta_{\eta, \Lambda}^\scal
=  \sum_{ h\in X_\Lambda} \frac{\eta(h)}{|\Aut(\Lambda_h)|} \cdot  \theta_{\Lambda_h}^\scal
 \in M_{n-1}(\Gamma_0(D), \chi_\kk^{n-1}).
\]
This form is cuspidal when the character $\chi_\eta$ is non-trivial. We denote its Fourier expansion by
\[
\theta_{\eta, \Lambda}^\scal(\tau) = \sum_{m \ge  0} R^\scal_{\eta, \Lambda}(m) \cdot q^m.
\]
Similarly, we may define
\[
\theta_{\eta, \Lambda}(\tau)
= \sum_{ h\in X_\Lambda}  \frac{\eta(h)}{|\Aut(\Lambda_h)|} \cdot  \theta_{\Lambda_{h}} (\tau),
\]
but this is only a formal sum: as $h$ varies the forms $\theta_{\Lambda_{h}}$
take values in the varying spaces $S_{\Lambda_h}^\vee$.

Lemma \ref{lem:ideal twist} allows us to identify $S_L\iso S_{L_h}$, and hence make sense of the $L$-function
$L(\tilde{g} ,\theta_{\Lambda_h} ,s)$ as in \eqref{eq:defl}.  In the next subsection we will compare
\begin{equation}\label{veccon}
L(\tilde{g} ,\theta_{\eta,\Lambda} ,s) = \sum_{h\in X_\Lambda} \frac{\eta(h)}{|\Aut(\Lambda_h)|}
\cdot L(\tilde{g} ,\theta_{\Lambda_h} ,s).
\end{equation}
to the usual convolution $L$-function
\begin{equation}\label{scacon}
L(g, \theta_{\eta, \Lambda}^\scal, s) =
\Gamma\big(\frac{s}2+ n-1\big)\sum_{m =1 }^\infty \frac{\overline{c(m)}
R^\scal_{\eta, \Lambda}(m)}{(4\pi m)^{\frac{s}2+n-1}}
\end{equation}
of the scalar-valued forms $g$ and $\theta^\scal_{\eta,\Lambda}$.


\subsection{Rankin-Selberg $L$-functions  for scalar and vector  valued forms}


In this subsection we prove a precise relation between (\ref{veccon}) and (\ref{scacon}).
First, we give an explicit formula for the Fourier coefficients $a(m,\mu)$  of $\tilde g$ in terms of those of $g$ analogous to \cite[Proposition 6.1.2]{BHKRY}.

For a prime $p$ dividing $D$  define
\begin{equation}\label{def-gamma-p}
\gamma_p =
 \delta_p^{-n} \cdot (D,p)_p^n \cdot \operatorname{inv}_p(V_p) \in \{ \pm 1, \pm i\} ,
\end{equation}
where $\operatorname{inv}_p(V_p)$ is the invariant  of $V_p=\Hom_\kk(W_0,W)\otimes_\Q\Q_p$ in the sense of  \cite[(1.8.3)]{BHKRY} and $\delta_p\in \{1,i\}$ is as before.
It is equal to the local Weil index of the Weil representation of $\SL_2(\Z_p)$ on $S_{L_p}\subset S(V_p)$, where $V_p$ is viewed as a quadratic space by taking the trace of the hermitian form.
This is explained in more detail in \cite[Section 8.1]{BHKRY}. For any $Q$ dividing $D$ we define
\begin{equation}\label{gamma def}
\gamma_Q  = \prod_{q\mid Q} \gamma_q .
\end{equation}

\begin{remark}
If $n$ is even and $p\mid D$, then \eqref{def-gamma-p} simplifies to
\[
\gamma_p= \left(\frac{-1}{p} \right)^{n/2} \mathrm{inv}_p(V_p).
\]
\end{remark}

For any  $\mu\in L'/L$ define $Q_\mu\mid D$ by
\[
Q_\mu = \prod_{\substack{p \mid D\\ \mu_p\neq 0}} p ,
\]
where $\mu_p$ is the image of $\mu$ in $L_p'/L_p$.
Let $\phi_\mu\in S_L$ be the characteristic function of $\mu$.

\begin{proposition}
\label{prop:scalar1.3}
  For all $m\in \Q$ the coefficients $\tilde{a}(m)\in S_L$ of $\tilde g$ satisfy
\[
\tilde a(m, \mu) =\begin{cases}
   \sum_{Q_\mu \mid Q \mid D}  Q^{1-n}\overline{\gamma_Q}   \cdot c_Q(mQ)  &\text{if $m \equiv -Q(\mu)  \pmod{ \Z}$,}
   \\[1ex]
    0  &\text{otherwise}.
    \end{cases}
\]
\end{proposition}

\begin{proof}
The first formula is a special case of  results of Scheithauer \cite[Section 5]{Scheithauer}.
It can also be proved in the same way as Proposition 6.1.2 of \cite{BHKRY}. 
The complex conjugation over $\gamma_Q$ arises because of the fact that  $\tilde g$ transforms with the complex conjugate representation $\overline{\omega}_L$. 
The additional factor $Q^{1-n}$ is due to the fact that we work here in weight $n$.
\end{proof}

\begin{proposition}
\label{prop:scalar-vector}
The convolution $L$-function \eqref{eq:defl}
satisfies
\[
L(\tilde g, \theta_\Lambda, s) =
\sum_{Q|D} Q^{\frac{s}2}  \gamma_Q \cdot L(g_Q, \theta_{\Lambda_\mathfrak q}^\scal, s)
,
\]
where $\mathfrak q\in \widehat \kk^\times$ is such that $\mathfrak q^2  \widehat\calO_\kk^\times=Q \widehat\calO_\kk^\times$.
Moreover, for any $\eta:X_\Lambda\to\C$ satisfying \eqref{eq:etamult} the $L$-functions (\ref{veccon}) and (\ref{scacon}) are related by
\[
L(\tilde g, \theta_{\eta, \Lambda}, s) =\sum_{Q|D} Q^{\frac{s}2}  \gamma_Q \cdot  \chi_\eta(\mathfrak{q}^{-1}) L(g_Q, \theta_{\eta, \Lambda}^\scal, s)
 .
\]
\end{proposition}

\begin{proof}
Proposition \ref{prop:scalar1.3} implies
\begin{align*}
\frac{L(\tilde g, \theta_\Lambda, s) }{\Gamma(\frac{s}2 +n -1)}
 &= \sum_{\mu \in  \Lambda' /\Lambda} \sum_{m \in \Q_{>0}}
 \sum_{Q_\mu \mid Q  \mid D}  Q^{1-n}
\gamma_Q
 \cdot \frac{ \overline{c_Q(mQ)}  R_{\Lambda}(m, \phi_\mu)}{(4\pi m)^{\frac{s}2 +n-1}}
 \\
 &=
  \sum_{Q \mid D}Q^{1-n}  \gamma_Q
  \sum_{m \in \frac{1}Q \Z_{>0}} \frac{\overline{c_Q(mQ)}}{ (4\pi m)^{\frac{s}2+n-1}}
   \sum_{\substack{\mu \in \Lambda'/\Lambda \\ Q_\mu |Q}} R_{\Lambda}(m, \phi_\mu)
 \\
  &=  \sum_{Q|D} Q^{\frac{s}2} \gamma_Q
  \sum_{m \in \Z_{>0}} \frac{\overline{c_Q(m)}}{ (4\pi m)^{\frac{s}2+n-1}}
  \sum_{\substack{\mu \in \Lambda'/\Lambda \\ Q_\mu \mid Q}}
  R_{\Lambda} (m/Q, \phi_\mu ).
\end{align*}
The first claim now follows from the relation
\[
\sum_{\substack{\mu \in \Lambda' /\Lambda \\ Q_\mu \mid Q}} R_{\Lambda}( m / Q, \mu)
= R_{\Lambda_{{\mathfrak q}^{-1}}}(m, 0)= R_{\Lambda_{{\mathfrak q}}}(m, 0).
\]

For the second claim,  if we  replace $\Lambda$ by $\Lambda_{h}$ and
$L_1$ by $L_{1,h}$ for $h\in X_\Lambda$, then  $L$ and   $\gamma_Q$
remain unchanged.  The above calculations therefore imply that
\begin{align*}
L(\tilde g, \theta_{\eta, \Lambda}, s)
&=  \sum_{Q \mid D}  \gamma_Q  Q^{\frac{s}2}
  \sum_{h\in X_\Lambda} \frac{\eta(h)}{|\Aut(\Lambda_h)|}  L(g_Q, \theta_{\Lambda_{\mathfrak{q}h}}^\scal, s)
 \\
 &=  \sum_{Q \mid D}   \gamma_Q  Q^{\frac{s}2}
  \sum_{h\in X_\Lambda} \frac{\eta(\mathfrak{q}^{-1}h)}{|\Aut(\Lambda_h)|}  L(g_Q, \theta_{\Lambda_{h}}^\scal, s)
 \\
 &=
 \sum_{Q|D} \gamma_Q Q^{\frac{s}2} \cdot \chi_\eta(\mathfrak{q}^{-1}) L(g_Q, \theta_{\eta, \Lambda}^\scal, s),
\end{align*}
where we have used \eqref{eq:etamult} and the fact that $|\Aut(\Lambda_h)| = |\Aut(\Lambda_{\mathfrak{q}h})|$. 
\end{proof}

\begin{corollary} \label{cor:9.2.6}
If  $n$ is even, then
\[
L(\tilde g, \theta_{\eta, \Lambda}, s) =
 L(g, \theta_{\eta, \Lambda}^\scal, s) \cdot
 \prod_{p|D} \big(1+ \chi_\eta(\mathfrak{p}^{-1}) \epsilon_p(g)\gamma_p p^{\frac{s}2}  \big ).
\]
\end{corollary}

\begin{proof}
This is immediate from Proposition \ref{prop:scalar-vector} and Remark \ref{rem:g_Q even}.
\end{proof}


\subsection{Small CM cycles and derivatives of $L$-functions, revisited}
\label{ss:smallredux}


Now we are ready to state a variant of Theorem \ref{Theorem A} using only scalar valued modular forms.
Assume $n\ge 3$.

Every  $h\in X_\Lambda$ determines a codimension $n-1$ cycle
\begin{equation}\label{twisty small cm}
\mathcal{Y}_{\mathrm{sm},h} \to  \mathcal{S}^*_\Kra
\end{equation}
as follows.  From the triple $(\mathfrak{a}_0,\mathfrak{a}_1,\mathfrak{b})$ fixed in \S \ref{ss:theta twist} and the hermitian $\co_\kk$-lattices
$L_h=L_{1,h}\oplus \Lambda_h$ of Lemma \ref{lem:ideal twist}, we  denote by $\mathfrak{a}_{1,h}$ and  $\mathfrak{b}_h$ the unique hermitian $\co_\kk$-lattices satisfying
\[
L_{1,h} \iso \Hom_{\co_\kk}(\mathfrak{a}_0,\mathfrak{a}_{1,h}),\quad 
\Lambda_h \iso \Hom_{\co_\kk}( \mathfrak{a}_0,\mathfrak{b}_h),
\]
and set $\mathfrak{a}_h=\mathfrak{a}_{1,h}\oplus \mathfrak{b}_h$ so that $L_h\iso \Hom_{\co_\kk}(\mathfrak{a}_0,\mathfrak{a}_h)$.
  Compare with (\ref{aux lattices 1}) and (\ref{aux lattices 2}).  
  
 Repeating the construction of the small CM cycle $\mathcal{Y}_\mathrm{sm}$ with the triple $(\mathfrak{a}_0,\mathfrak{a}_1,\mathfrak{b})$ replaced by   $(\mathfrak{a}_0,\mathfrak{a}_{1,h},\mathfrak{b}_h)$ results in a proper \'etale $\co_\kk$-stack $\mathcal{Y}_{\mathrm{sm},h}$.
 Repeating the construction of the Shimura variety $\mathcal{S}_\Kra$ with the triple $(\mathfrak{a}_0,\mathfrak{a})$ replaced by $(\mathfrak{a}_0,\mathfrak{a}_h)$ results in a new Shimura variety $\mathcal{S}_{\Kra,h}$, along with a finite and unramified morphism
 \[
 \mathcal{Y}_{\mathrm{sm},h} \to \mathcal{S}_{\Kra,h}.
 \]

It  follows from Lemma \ref{lem:ideal twist} that $\mathfrak{a}$ and $\mathfrak{a}_h$ are isomorphic everywhere locally, 
and examination of the moduli problem defining $\mathcal{S}_\Kra$ in \cite[\S 2.3]{BHKRY} shows that $\mathcal{S}_\Kra$ depends only the everywhere local data determined by the pair $(\mathfrak{a}_0,\mathfrak{a})$, and not on the actual global $\co_\kk$-hermitian lattices.
Therefore, there is a canonical morphism of $\co_\kk$-stacks
 \[
 \mathcal{Y}_{\mathrm{sm},h}  \to  \mathcal{S}_{\Kra,h} \iso \mathcal{S}_\Kra
 \] 
 in which the isomorphism is simply the identity functor on the moduli problems.    
 In the end, this amounts to simply repeating the construction of 
 $\mathcal{Y}_\mathrm{sm} \to \mathcal{S}_\Kra$ from Definition \ref{def:small cm definition} word-for-word, but replacing $\Lambda$ by $\Lambda_h$ everywhere.
This defines the desired cycle (\ref{twisty small cm}).

Each algebraic automorphic form $\eta:X_\Lambda\to\C$ satisfying \eqref{eq:etamult}
now determines a cycle
\begin{equation*}
 \eta \mathcal{Y}_{\mathrm{sm}} = \sum_{h\in X_\Lambda}
  \eta(h)   \cdot  \mathcal{Y}_{\mathrm{sm},h}
\end{equation*}
on $\mathcal{S}^*_\Kra$ with complex coefficients, and a corresponding linear functional
\[
  [ - :  \eta \mathcal{Y}_{\mathrm{sm}}]  :   \Cha_\C^1(\mathcal{S}^*_\Kra) \to \C.
\]

\begin{theorem} \label{maintheo2}
The arithmetic theta lift (\ref{ATL}) satisfies
\[
[ \widehat{\theta}(g) : \mathcal{Y}_\mathrm{sm} ]
=
- \deg_\C (\mathcal{Y}_\mathrm{sm})
\cdot \frac{d}{ds}\Big[
\sum_{Q|D} Q^{\frac{s}2}  \gamma_Q L(g_Q, \theta_{ \Lambda_{\mathfrak q}}^\scal, s)\Big]\big|_{s=0}
 ,
\]
where $\mathfrak q\in \widehat \kk^\times$ is such that $\mathfrak q^2  \widehat\calO_\kk^\times=Q \widehat\calO_\kk^\times$.
Moreover,  if $n$ is even  and $\eta:X_\Lambda\to\C$ satisfies \eqref{eq:etamult}, then
\begin{align*}
&[ \widehat{\theta}(g ) :\eta \mathcal{Y}_\mathrm{sm} ]\\
&=
 - 2^{1-o(d_\kk)}  \left(  h_\kk / w_\kk \right) ^2 \cdot \frac{d}{ds}
\Big[ L(g, \theta_{\eta, \Lambda}^\scal,s)  \cdot
\prod_{p |D} \big(1+ \chi_\eta(\mathfrak p^{-1})\epsilon_p(g)\gamma_p p^{\frac{s}2}  \big) \Big] \big|_{s=0},
\end{align*}
where $\mathfrak p\in \widehat \kk^\times$ such that $\mathfrak p^2  \widehat\calO_\kk^\times=p \widehat\calO_\kk^\times$.
Note that in the first formula the sum is over all positive divisors $Q\mid D$, while in the second the product is over
the prime divisors $p\mid D$. 
\end{theorem}

\begin{proof}
The first assertion follows from Theorem \ref{Theorem A} and Proposition \ref{prop:scalar-vector}. 

For the second assertion, applying  Theorem \ref{Theorem A}  to 
\[
\mathcal{Y}_{\mathrm{sm},h} \to \mathcal{S}_{\Kra,h}^* \iso \mathcal{S}_\Kra^*
\]
yields 
\[
[ \widehat{\theta}(g) : \mathcal{Y}_{\mathrm{sm},h} ]
=
- \deg_\C (\mathcal{Y}_{\mathrm{sm},h})
\cdot \frac{d}{ds}   L(\tilde{g},\theta_{\Lambda_h} ,s) \big|_{s=0}.
\]
Combining this with  Proposition \ref{prop:small degree} yields
\begin{align*}
[ \widehat{\theta}(g) :  \eta \mathcal{Y}_{\mathrm{sm}} ]  
& = 
-  2^{1-o(d_\kk)}  \left(  h_\kk / w_\kk \right) ^2   \cdot \frac{d}{ds}     L(\tilde{g},\theta_{\eta,\Lambda} ,s)    \big|_{s=0},
\end{align*}
and an application of  Corollary \ref{cor:9.2.6} completes the proof.
\end{proof}

\begin{remark}
Since the $L$-function \eqref{veccon} vanishes at $s=0$, the same must be true for the expressions in brackets on the right hand sides of the equalities of the above theorem. In particular, when $n$ is even, then  either $L(g, \theta_{\eta, \Lambda}^\scal,s)$ or at least one of the factors
\[
1+ \chi_\eta(\mathfrak p^{-1})\epsilon_p(g)\gamma_p p^{\frac{s}2}  
\]
(for a prime $p\mid D$)  vanishes at $s=0$. If we pick the newform $g$ such that the latter local factors are nonvanishing, then $L(g, \theta_{\eta, \Lambda}^\scal,0)=0$ and we obtain
\begin{align*}
[ \widehat{\theta}(g) :\eta \mathcal{Y}_\mathrm{sm} ]=
 - 2^{1-o(d_\kk)} \frac{h_\kk^2}{w_\kk^2} \cdot
\prod_{p |D} \big(1+ \chi_\eta(\mathfrak p^{-1})\epsilon_p(g)\gamma_p \big)\cdot L'(g, \theta_{\eta, \Lambda}^\scal,0) .
\end{align*}
\end{remark}


\section{Big CM cycles and derivatives of $L$-functions}
\label{s:big GZ}


In this section we prove Theorem \ref{Theorem B}  by combining  results of \cite{BHKRY}  and \cite{Ho1,Ho2,BKY}.
We asume $n\ge 2$ until \S \ref{ss:big harmonic intersection}, at which point we restrict to $n\ge 3$.


\subsection{A Shimura variety of dimension zero}


Let $F$ be a totally real field of degree $n$, and define a CM field
$
E=\kk\otimes_\Q F.
$
Define a rank $n+2$ torus $T_\mathrm{big}$ over $\Q$ as the fiber product
\[
\xymatrix{
{T_\mathrm{big}}  \ar[rrr]\ar[d]  &  &  & {  \mathbb{G}_m } \ar[d]^{\mathrm{diag.}} \\
{\mathrm{Res}_{\kk/\Q}\mathbb{G}_m \times \mathrm{Res}_{E/\Q}\mathbb{G}_m  } \ar[rrr]_{\mathrm{Nm} \times \mathrm{Nm}}  &   &   &{  \mathbb{G}_m \times  \mathrm{Res}_{F / \Q} \mathbb{G}_m. }
}
\]
Its group of $\Q$-points is
\[
T_\mathrm{big}(\Q) \iso \{ (x,y) \in \kk^\times \times E^\times : x\overline{x}=y\overline{y} \}.
\]

\begin{remark}\label{rem:split group}
There is an isomorphism
\[
T_\mathrm{big}(\Q)  \iso \kk^\times \times \mathrm{ker} ( \mathrm{Nm} : E^\times \to F^\times )
\]
defined by $(x,y) \mapsto (x, x^{-1} y)$.  It is clear that this arises from an isomorphism
\[
T_\mathrm{big} \iso    \mathrm{Res}_{\kk/\Q}\mathbb{G}_m
 \times \mathrm{ker}\big( \mathrm{Nm} :  \mathrm{Res}_{E/\Q}\mathbb{G}_m \to \mathrm{Res}_{F/\Q}\mathbb{G}_m \big).
\]
\end{remark}

As in the discussion preceding Theorem \ref{Theorem B},  let $\Phi \subset \Hom_\Q(E,\C)$  be a CM type of signature $(n-1,1)$,  let
  \[
  \varphi^\mathrm{sp}:E \to \C
  \]
   be its special element, and let $\co_\Phi$ be the ring of integers of $E_\Phi=\varphi^\mathrm{sp}(E)$.

The CM type $\Phi$ determines an isomorphism $ \C^n \iso E_\R $, and hence an embedding $\C^\times \to E_\R^\times$ arising  from a morphism of real algebraic groups $\mathbb{S} \to (\mathrm{Res}_{E/\Q}\mathbb{G}_m )_\R$.
This induces a morphism
\[
\mathbb{S} \to (\mathrm{Res}_{\kk/\Q}\mathbb{G}_m)_\R  \times (\mathrm{Res}_{E/\Q}\mathbb{G}_m )_\R,
\]
which factors through a morphism
\[
h_\mathrm{big} : \mathbb{S} \to T_{\mathrm{big} , \R }.
\]
The pair $(T_\mathrm{big},\{h_\mathrm{big} \})$ is  a Shimura datum, which, along with the compact open subgroup
\[
K_\mathrm{big} = T_\mathrm{big}(\A_f) \cap ( \widehat{\co}_\kk^\times \times \widehat{\co}_E^\times),
\]
determines a $0$-dimensional $E_\Phi$-stack  $\mathrm{Sh}(T_\mathrm{big})$ with complex points
\[
\mathrm{Sh}(T_\mathrm{big}) (\C) = T_\mathrm{big}(\Q) \backslash \{ h_\mathrm{big}\} \times T_\mathrm{big}(\A_f) / K_\mathrm{big}.
\]


\subsection{The big CM cycle}
\label{ss:big cm}


The Shimura  variety just constructed has a  moduli interpretation, which we will use to construct an integral model.
The interpretation we have in mind requires first choosing  a triple $(\mathfrak{a}_0,\mathfrak{a} , i_E )$ in which
\begin{itemize}
\item
$\mathfrak{a}_0$ is a self-dual hermitian $\co_\kk$-lattice of signature $(1,0)$,
\item
$\mathfrak{a}$ is a self-dual hermitian $\co_\kk$-lattice of signature $(n-1,1)$,
\item
$i_E : \co_E \to \End_{\co_\kk} (\mathfrak{a})$ is an action extending the  action of $\co_\kk$.
\end{itemize}
Denoting by $H:\mathfrak{a} \times \mathfrak{a} \to \co_\kk$ the hermitian form, we require further that
\[ H( i_E(x) a,b)  = H(  a, i_E(\overline{x}) b )\]  for all $x\in \co_E$ and $a,b\in \mathfrak{a}$, and that in the decomposition
\[
\mathfrak{a}_\R \iso \bigoplus_{ \varphi_F : F \to \R } \mathfrak{a}\otimes_{\co_F ,\varphi_F } \R
\]
the summand indexed by $\varphi_F = \varphi^\mathrm{sp}|_F$ is negative definite (which, by the signature condition, implies that the other summands are positive definite).

\begin{remark}
In general such a triple need not exist.  In the applications will assume that the discriminants of $\kk/\Q$ and $F/\Q$ are odd and relatively prime, and in this case one can construct such a triple using the argument of \cite[Proposition 3.1.6]{Ho1}.
\end{remark}

We now define a moduli space of abelian varieties with complex multiplication by $\co_E$ and type $\Phi$, as in \cite[\S 3.1]{Ho1}.
Denote by $\mathcal{CM}_\Phi$ the functor that associates to every $\co_\Phi$-scheme $S$ the groupoid of triples $(A,\iota,\psi)$ in which
\begin{itemize}
\item
$A \to S$ is an abelian scheme of dimenension $n$,
\item
$\iota : \co_E \to \End(A)$ is an $\co_E$-action,
\item
$\psi : A\to A^\vee$ is a principal polarization such that \[\iota(x)^\vee \circ \psi= \psi\circ \iota(\overline{x})\] for all $x\in \co_E$.
\end{itemize}
We also impose the  \emph{$\Phi$-determinant condition} that every $x\in \co_E$ acts on $\Lie(A)$ with characteristic polynomial equal to the image of
\[
\prod_{\varphi \in \Phi} (T-\varphi(x)) \in \co_\Phi[T]
\]
 in $\co_S[T]$.
 We usually abbreviate $A\in \mathcal{CM}_\Phi(S)$, and suppress the data $\iota$ and $\psi$ from the notation.
 By  \cite[Proposition 3.1.2]{Ho1}, the  functor $\mathcal{CM}_\Phi$ is represented by a   Deligne-Mumford stack, proper and \'etale over $\co_\Phi$ .

\begin{remark}
The $\Phi$-determinant condition defined above agrees with that of  \cite[\S 3.1]{Ho1}.  As in \cite[Proposition 2.1.3]{hartwig}, this is a consequence of Amitsur's formula, which can be found in  \cite[Theorem A]{amitsur} or \cite[Lemma 1.12]{chenevier}.
\end{remark}

Define an open and closed substack
\[
\mathcal{Y}_\mathrm{big} \subset  \mathcal{M}_{(1,0)}  \times_{\co_\kk}  \mathcal{CM}_\Phi
\]
as the union of connected components $\mathcal{B} \subset \mathcal{M}_{(1,0)}  \times_{\co_\kk}  \mathcal{CM}_\Phi$ satisfying the following property: for every complex point $y = (A_0,A) \in \mathcal{B}(\C)$, and for all primes $\ell$,
there is an $\co_E$-linear isomorphism of hermitian $\co_{\kk,\ell}$-lattices
 \begin{equation}\label{big genus}
 \Hom_{\co_{\kk , \ell}  } (A_0[\ell^\infty] , A [\ell^\infty] )
  \iso \Hom_{\co_\kk}( \mathfrak{a}_0,\mathfrak{a} ) \otimes_\Z  \Q_\ell .
 \end{equation}

\begin{remark}\label{rem:betti hermitian}
To verify that a connected component 
$\mathcal{B} \subset \mathcal{M}_{(1,0)}  \times_{\co_\kk}  \mathcal{CM}_\Phi$ is contained in $\mathcal{Y}_\mathrm{big}$, it suffices to check that (\ref{big genus}) holds for one complex point $y\in \mathcal{B}(\C)$.
This is a consequence of  the  main theorem of complex multiplication and the fact that the points of $\mathcal{B}(\C)$ form a single $\Aut(\C/E_\Phi)$-orbit.
\end{remark}

\begin{proposition}\label{prop:big integral}
There is a canonical isomorphism of $E_\Phi$-stacks
\[
 \mathrm{Sh}(T_\mathrm{big})   \iso   \mathcal{Y}_{\mathrm{big}/E_\Phi}  .
\]
\end{proposition}

\begin{proof}
The natural actions of $\co_\kk$ and $\co_E$ on $\mathfrak{a}_0$ and $\mathfrak{a}$ determine an action of the subtorus
\[
T_\mathrm{big} \subset   \mathrm{Res}_{\kk/\Q}\mathbb{G}_m \times \mathrm{Res}_{E/\Q}\mathbb{G}_m
\]
on $\mathfrak{a}_{0\Q}$ and $\mathfrak{a}_\Q$, and the morphism $h_\mathrm{big}  : \mathbb{S}  \to   T_{\mathrm{sm},\R}$
endows each  of the real vector spaces $\mathfrak{a}_{0\R}$ and $\mathfrak{a}_{\R}$ with a complex  structure.

The desired isomorphism  on complex points sends
\[
( h_\mathrm{big} , g) \in \mathrm{Sh}(T_\mathrm{sm})(\C)
\]
 to the pair    $(A_0,A)$ defined by
\[
A_0(\C) = \mathfrak{a}_{0\R} /  g \mathfrak{a}_0,\quad
A(\C) = \mathfrak{a}_{\R} / g \mathfrak{a}.
\]
The elliptic curve $A_0$ is endowed with its natural $\co_\kk$-action and its unique principal ploarization.
The abelian variety $A$ is endowed with its natural $\co_E$-action, and the polarization induced by the symplectic form determined by its $\co_\kk$-hermitian form, as in the proof of \cite[Proposition 2.2.1]{BHKRY}.

It follows from the theory of canonical models that this isomorphism on complex points descends to an isomorphism of $E_\Phi$-stacks.
\end{proof}

The triple $(\mathfrak{a}_0,\mathfrak{a},i_E)$ determines a  pair $(\mathfrak{a}_0,\mathfrak{a})$ as in the introduction, which determines a unitary Shimura variety  with  integral model  $\mathcal{S}_\Kra$ as in (\ref{kramer inclusion}).
Recalling that $\co_\kk \subset \co_\Phi$ as subrings of $\C$, we now view both $\mathcal{Y}_\mathrm{big}$ and $\mathcal{CM}_\Phi$ as $\co_\kk$-stacks.
There is a commutative diagram
\[
\xymatrix{
{ \mathcal{Y}_\mathrm{big}  }  \ar[r] \ar[d]_\pi  &    {  \mathcal{M}_{(1,0)} \times \mathcal{CM}_\Phi }   \ar[d]   \\
 { \mathcal{S}_\Kra  } \ar[r] & {  \mathcal{M}_{(1,0)} \times \mathcal{M}_{(n-1,1)}^\Kra }
}
\]
(all fiber products are over $\co_\kk$), in which the vertical arrow on the right is the identity on the first factor and ``forget complex multiplication'' on the second.
The arrow $\pi$ is defined by the commutativity of the diagram.

\begin{remark}
In order to define the morphism
\[
\mathcal{CM}_\Phi \to \mathcal{M}_{(n-1,1)}^\Kra
\]
in the diagram above, we must  endow a point $A\in \mathcal{CM}_\Phi(S)$ with a subsheaf $\mathcal{F}_A \subset \Lie(A)$ satisfying Kr\"amer's condition \cite[\S 2.3]{BHKRY}.   Using the morphism
\[
 \co_E \map{  \varphi^\mathrm{sp}}  \co_\Phi \to \co_S,
\]
denote by $J_{\varphi^\mathrm{sp}} \subset \co_E \otimes_\Z \co_S$ the kernel of
\[
\co_E \otimes_\Z \co_S \map{ x\otimes y\mapsto \varphi^\mathrm{sp}(x)\cdot y  }  \co_S.
\]
According to \cite[Lemma 4.1.2]{Ho2}, the subsheaf $\mathcal{F}_A= J_{\varphi^\mathrm{sp}} \Lie(A)$ has the desired properties.
\end{remark}

\begin{definition}\label{def:big cycle}
Composing the morphism $\pi$ in the diagram above with the inclusion of $\mathcal{S}_\Kra$ into its toroidal compactification, we obtain  a morphism of $\co_\kk$-stacks
\[
\pi : \mathcal{Y}_\mathrm{big} \to \mathcal{S}_\Kra^*,
\]
called  the  \emph{big CM cycle}.
\end{definition}

Exactly as in \S \ref{ss:small cycle construction},  the \emph{arithmetic degree along} $\mathcal{Y}_\mathrm{big}$ is the composition
\[
 \widehat{\mathrm{Ch}}^1_\C(\mathcal{S}^*_\Kra) \map{\pi^*}  \widehat{\mathrm{Ch}}^1_\C(\mathcal{Y}_\mathrm{big}) \map{ \widehat{\deg} } \C .
\]
We denote this linear functional by
$
\widehat{\mathcal{Z}} \mapsto [ \widehat{\mathcal{Z}} : \mathcal{Y}_\mathrm{big} ].
$

\begin{remark}
The big CM cycle  arises from a morphism of Shimura varieties.
Indeed, there is a morphism of Shimura data $(T_\mathrm{big},\{h_\mathrm{big}\}) \to (G,\mathcal{D})$, and
the induced morphism of Shimura varieties sits in a commutative diagram of $E_\Phi$-stacks
\[
\xymatrix{
{  \mathrm{Sh}(T_\mathrm{big}) }  \ar[r]  \ar[d]_\iso &   { \mathrm{Sh}(G,\mathcal{D})_{/E_\Phi} } \ar[d]^\iso  \\
{  \mathcal{Y}_{\mathrm{big}/ E_\Phi}  }  \ar[r]^\pi  & {  \mathcal{S}_{\Kra/ E_\Phi} } .
}
\]
\end{remark}

\begin{proposition}\label{prop:class number}
The degree $\deg_\C(\mathcal{Y}_\mathrm{big})$ of Theorem \ref{Theorem B} satisfies
\[
\frac{ 1 }{n} \cdot \deg_\C (\mathcal{Y}_\mathrm{big})
=
\frac{ h_\kk}{w_\kk} \cdot  \frac{ \Lambda(0,\chi_E)  }{ 2^{r-1} },
\]
where $r$ is the number of places of $F$ that ramify in $E$ (including all archimedean places).
\end{proposition}

\begin{proof}
It is clear from Proposition \ref{prop:big integral} that
\[
\frac{ 1 }{n} \cdot \deg_\C (\mathcal{Y}_\mathrm{big})
 =  \sum_{ y\in \mathrm{Sh}(T_\mathrm{big} ) (\C)}  \frac{1}{ |\Aut(y)| }
 = \frac{| T_\mathrm{big}(\Q) \backslash T_\mathrm{big}(\A_f) / K_\mathrm{big}      |}{|  T_\mathrm{big}(\Q) \cap K_\mathrm{big}  |}.
\]
Note that when we defined the degree on the left we counted the complex points of $\mathcal{Y}_\mathrm{big}$ viewed as an $\co_\kk$-stack, whereas in the middle expression we are viewing  $\mathrm{Sh}(T_\mathrm{big} )$ as an  $E_\Phi$-stack.  This is the reason for the correction  factor of  $n=[E_\Phi:\kk]$ on the left.

Let $E' \subset E^\times$ be the kernel of the norm map $\mathrm{Nm} : E^\times \to F^\times$, and define
\[
\widehat{E}' \subset \widehat{E}^\times ,\quad \widehat{\co}_E' \subset \widehat{\co}_E^\times
\]
similarly.  Note that $\mu(E)=E' \cap \widehat{\co}_E'$ is the group of roots of unity in $E$, whose order we denote by $w_E$.
Using the isomorphism
$
T_\mathrm{big}(\Q) \iso  \kk^\times \times E'
$
of Remark \ref{rem:split group}, we find
\begin{equation}\label{class number splitting}
\frac{| T_\mathrm{big}(\Q) \backslash T_\mathrm{big}(\A_f) / K_\mathrm{big}      |}{|  T_\mathrm{big}(\Q) \cap K_\mathrm{big}  |}
=   \frac{ h_\kk}{w_\kk} \cdot \frac{|  E'  \backslash \widehat{E}'  / \widehat{\co}_E'     | }{ w_E } .
\end{equation}

Denote by $C_F$ and $C_E$ the ideal class groups of $E$ and $F$, and by $\tilde{F}$ and  $\tilde{E}$ their Hilbert class fields.  
As $E/F$ is ramified at all archimedean places, $\tilde{F} \cap E=F$, and the natural map
  \[
  \Gal(\tilde{E}/E) \to \Gal(\tilde{F}/F)
  \]
  is surjective.  Hence, by class field theory, the norm 
  \[
  \mathrm{Nm} : C_E \to C_F
  \] 
  is surjective.  Denote its kernel by $B$, so that we have a short exact sequence
\[
1 \to B \to C_E \map{\mathrm{Nm}} C_F \to 1.
\]

Define a group
\[
\tilde{B} =  E^\times  \backslash  \left\{ ( \mathfrak{B} ,\beta) :
 \begin{array}{c}
 \mathfrak{B} \subset E \mbox{ is a fractional $\co_E$-ideal,  }  \\
 \beta \in F^\times  , \mbox{ and }  \mathrm{Nm} (  \mathfrak{B} ) = \beta \co_F
\end{array}
\right\},
\]
where the action of $E^\times$ is by $\alpha \cdot (\mathfrak{B},\beta) = (\alpha \mathfrak{B} , \alpha\overline{\alpha}\beta)$.  There is an evident short exact sequence
\[
1 \to     \mathrm{Nm}(\co_E^\times) \backslash  \co_F^\times   \map{ \beta \mapsto (\co_E,\beta) } \tilde{B} \to B\to 1.
\]

\begin{lemma}
We have $[  \co_E^\times  :   \mathrm{Nm}(\co_E^\times)   ]  =  2^{n-1} w_E$.
\end{lemma}

\begin{proof}
Let $Q= [  \co_E^\times : \mu(E) \co_F^\times ]$.  If $Q=1$ then
\[
[ \mathrm{Nm}(\co_E^\times) : \co_F^{\times, 2}] =1 \quad \mbox{and} \quad  [\co_E^\times: \co_F^\times] = \frac{1}{2} \cdot w_E,
\]
and so
\[
[  \co_F^\times  :   \mathrm{Nm}(\co_E^\times)   ]  = [  \co_F^\times  :  \co_F^{\times,2}   ]  = 2^n
=   \frac{2^{n-1} w_E} { [ \co_E^\times : \co_F^\times] },
\]
where the middle equality follows from Dirichlet's unit theorem.

If $Q>1$ then \cite[Theorem 4.12]{washington} and its proof show that $Q=2$, and that the image of the map $\phi : \co_E^\times \to \co_E^\times$ defined by $\phi(x)=x/\overline{x}$ is the  index two subgroup
$
\phi(\co_E^\times)= \mu(E)^2 \subset \mu(E).
$
From this it follows easily that
\[
[ \mathrm{Nm}(\co_E^\times) : \co_F^{\times, 2}] =2 \quad \mbox{and} \quad  [\co_E^\times: \co_F^\times] =  w_E,
\]
and so
\[
[  \co_F^\times  :   \mathrm{Nm}(\co_E^\times)   ]  = \frac{1}{2} \cdot  [  \co_F^\times  :  \co_F^{\times,2}   ]   = 2^{n-1}
=   \frac{2^{n-1} w_E} { [ \co_E^\times : \co_F^\times] }.
\]
\end{proof}

Combining the information we have so far gives
\begin{equation}\label{dirichlet}
|\tilde{B}|
=   [  \co_F^\times  :   \mathrm{Nm}(\co_E^\times)   ]    \cdot   |B|
=  \frac{2^{n-1} w_E} { [ \co_E^\times : \co_F^\times] } \cdot \frac{ |C_E| }{ | C_F| }
= w_E \cdot \Lambda(0,\chi_E),
\end{equation}
where the final equality is a  consequence of Dirichlet's class number formula.

\begin{lemma}
There is an  exact sequence
\[
1 \to E'  \backslash \widehat{E}'  / \widehat{\co}_E'   \to  \tilde{B} \to \{ \pm 1 \}^r \to  \{ \pm 1\} \to 1.
\]
\end{lemma}

\begin{proof}
Every $ x \in \widehat{E}'$ determines a fractional $\co_E$-ideal $\mathfrak{B}=x\co_E$ with $\mathrm{Nm}(\mathfrak{B}) = \co_F$, and the rule $x\mapsto  (\mathfrak{B},1)$ is easily seen to define an injection
\begin{equation}\label{principal genus}
 E'  \backslash \widehat{E}'  / \widehat{\co}_E'  \to \tilde{B}.
\end{equation}

Given a $(\mathfrak{B},\beta) \in \tilde{B}$, consider the elements $\chi_{E,v}(\beta) \in \{ \pm 1\}$ as $v$ runs over all places of $F$.
If $v$ is split in $E$ then certainly $\chi_{E,v}(\beta) =1$.  If $v$ is inert in $E$ then  $\mathrm{Nm}(\mathfrak{B}) = \beta \co_F$ implies that $\chi_{E,v}(\beta) =1$.    As the product over all $v$ of $\chi_{E,v}(\beta)$ is equal to $1$, we see that sending $(\mathfrak{B},\beta)$ to the tuple of $\chi_{E,v}(\beta)$ with $v$ ramified in $E$ defines a homomorphism
\begin{equation}\label{genus}
\tilde{B} \to \mathrm{ker}\big( \{ \pm 1\}^r \map{\mathrm{product} }  \{ \pm 1\} \big).
\end{equation}

To see that (\ref{genus}) is surjective, fix a tuple $(\epsilon_v)_v \in \{ \pm 1\}^r$ indexed by the places of $F$ ramified in $E$, and assume that $\prod_v\epsilon_v=1$.    Let $b\in \A_F^\times$ be any idele satisfying:
\begin{itemize}
\item
 If  $v$ is ramified in $E$ then $\chi_{E,v}( b_v) = \epsilon_v$.
\item
If $v$ is a finite place of $F$ then $b_v \in \co_{F,v}^\times$.
\end{itemize}
The second condition implies that $\chi_{E,v}( b_v) =1$  whenever $v$ is unramified in $E$, and hence
\[
\chi_E(b) = \prod_v \epsilon_v =1.
\]
Thus $b$ lies in the kernel of the reciprocity map
\[
\A_F^\times \to F^\times\backslash \A_F^\times/ \mathrm{Nm}(\A_E^\times) \iso \Gal(E/F),
\]
and so can be factored as  $b =   \beta^{-1} x\overline{x} $ for some $\beta\in F^\times$ and $x\in \A_E^\times$.  Setting $\mathfrak{B} = x\co_E$, the pair $(\mathfrak{B},\beta) \in \tilde{B}$ maps to $(\epsilon_v)_v$ under (\ref{genus}).

It only remains to show that the image of (\ref{principal genus}) is equal to the kernel of (\ref{genus}).  It is clear from the definitions that the composition
\[
E'  \backslash \widehat{E}'  / \widehat{\co}_E'  \to \tilde{B} \to \{ \pm 1\}^r
\]
is trivial, proving one inclusion.  For the other inclusion, suppose $(\mathfrak{B},\beta) \in \tilde{B}$ lies in the kernel of (\ref{genus}).  We have already seen that this implies that $\beta \in F^\times$ satisfies $\chi_{E,v}(\beta)=1$ for every place $v$ of $F$, and so $\beta$ is a norm from $E$ everywhere locally.  By the Hasse-Minkowski theorem, $\beta$ is a norm globally, say $\beta =\alpha \overline{\alpha}$ with $\alpha\in E^\times$.     In the group $\tilde{B}$, we therefore have the relation
\[
(\mathfrak{B},\beta) = \alpha^{-1}  (\mathfrak{B},\beta) =(\mathfrak{A},1)
\]
for a fractional $\co_E$-ideal $\mathfrak{A} =\alpha^{-1}\mathfrak{B}$ satisfying $\mathrm{Nm}(\mathfrak{A})=\co_F$.  Any such $\mathfrak{A}$ has the form $\mathfrak{A}  = x\co_E$ for some $x\in \widehat{E}'$, proving that $(\mathfrak{B},\beta)$ lies in the image of (\ref{principal genus}).
\end{proof}

Combining the lemma with (\ref{dirichlet}) gives
\[
\frac{|  E'  \backslash \widehat{E}'  / \widehat{\co}_E'     | }{ w_E }
=  \frac{ |\tilde{B}| } { 2^{r-1} w_E } = \frac{\Lambda(0,\chi_E) }{ 2^{r-1} },
\]
and combining this with (\ref{class number splitting}) completes the proof of Proposition \ref{prop:class number}.
\end{proof}

\begin{proposition}\label{prop:big no error}
Assume that the discriminants of $\kk$ and $F$ are relatively prime.
The constant term (\ref{constant term}) satisfies
\[
[ \widehat{\mathcal{Z}}_\Kra^\tot(0) : \mathcal{Y}_\mathrm{big} ] = - [ \widehat{\bm{\omega}} : \mathcal{Y}_\mathrm{big} ] .
\]
\end{proposition}

\begin{proof}
The stated equality is equivalent to
\[
[ ( \mathrm{Exc}  , - \log(D))   : \mathcal{Y}_\mathrm{big} ]  = 0,
\]
and so it suffices to prove
\[
[  ( 0 ,  \log(D) )  : \mathcal{Y}_\mathrm{big} ]    =   \deg_\C( \mathcal{Y}_\mathrm{big} )  \cdot \log(D) =[  ( \mathrm{Exc} , 0)  : \mathcal{Y}_\mathrm{big} ]  .
\]
The first equality is clear from the definitions.  To prove the second equality, we first argue that
\begin{equation}\label{big physical intersection}
 \mathcal{Y}_\mathrm{big} \times_{\mathcal{S}_\Kra} \mathrm{Exc} =  \mathcal{Y}_{\mathrm{big}} \times_{\Spec(\co_\kk)} \Spec( \co_\kk / \mathfrak{d}_\kk ) ,
\end{equation}
as in the proof of Proposition \ref{prop:no error}.

The inclusion $\subset$ of (\ref{big physical intersection}) is again clear from
\[
\mathrm{Exc} \subset \mathcal{S}_\Kra \times_{ \Spec(\co_\kk)  } \Spec( \co_\kk / \mathfrak{d}_\kk ).
\]
Recall that  $\mathcal{Y}_\mathrm{big} \to \Spec(\co_\Phi)$ is \'etale.  Our hypothesis on the discriminants of  $\kk$ and $F$ implies that
$\Spec(\co_\Phi) \to \Spec(\co_\kk)$ is \'etale at all primes dividing $\mathfrak{d}_\kk$, and hence the same is true for
$
\mathcal{Y}_\mathrm{big} \to \Spec(\co_\kk).
$
This implies that the right hand side of (\ref{big physical intersection}) is reduced, and hence so is the left hand side.  To prove equality in (\ref{big physical intersection}), it therefore suffices to prove the inclusion $\supset$ on the level of geometric points.

Suppose $\mathfrak{p}\mid \mathfrak{d}_\kk$ is prime, and let $\F_\mathfrak{p}^\alg$ be an algebraic closure of its residue field.  Suppose that $y\in \mathcal{Y}_\mathrm{big}(\F_\mathfrak{p}^\alg)$ corresponds to the pair $(A_0,A)$, so that $A\in \mathcal{CM}_\Phi( \F_\mathfrak{p}^\alg)$.    Let $W$ be the completed \'etale local ring of the geometric point
\[
\Spec(\F_\mathfrak{p}^\alg) \map{y} \mathcal{Y}_\mathrm{big} \to \Spec(\co_\Phi).
\]
More concretely, $W$ is the completion of the maximal unramified extension of $\co_{\kk,\mathfrak{p}}$, equipped with an injective ring homomorphism $\co_\Phi \to W$.
  Let $\C_\mathfrak{p}$ be the completion of an algebraic closure of the fraction field of $W$, and fix an isomorphism of $E_\Phi$-algebras $\C \iso \C_\mathfrak{p}$.

For every $\varphi\in \Phi$ the induced map $\co_E \to \C \iso \C_\mathfrak{p}$ takes values in the subring $W$, and the induced map
\[
\co_E \otimes_\Z W \to \prod_{\varphi \in \Phi} W
\]
is surjective (by our hypothesis that $\kk$ and $F$ have relatively prime discriminants).
Denote its kernel by  $J_\Phi\subset \co_E\otimes_\Z W$, and define an $\co_E \otimes_\Z W$-module
\[
\Lie_\Phi = ( \co_E \otimes_\Z W  ) / J_\Phi  \iso \prod_{\varphi \in \Phi} W.
\]
As in the proof of \cite[Lemma 4.1.2]{Ho2}, there is an isomorphism of $\co_E\otimes_\Z \F_\mathfrak{p}^\alg$-modules
\[
\Lie(A) \iso \Lie_\Phi \otimes_W \F_\mathfrak{p}^\alg \iso \prod_{\varphi \in \Phi}  \F_\mathfrak{p}^\alg.
\]

Let  $\delta\in \co_\kk$ be a square root of $-D$.  As the image of $\delta$ under
\[
\co_E \map{\varphi} W \to \F_\mathfrak{p}^\alg
\]
is $0$ for every $\varphi\in \Phi$, it follows  from  what was said above that $\delta$ annihilates $\Lie(A)$.
Exactly as in the proof of  Proposition \ref{prop:no error}, this implies that the image of $y$ under $\mathcal{Y}_\mathrm{big} \to \mathcal{S}_\Kra$ lies on the exceptional divisor.
This completes the proof of (\ref{big physical intersection}), and the remainder of the proof is exactly as in Proposition \ref{prop:no error}.
\end{proof}


\subsection{A generalized $L$-function}
\label{ss:weird L}


The action $i_E: \co_E \to \End_{\co_\kk}(\mathfrak{a})$ makes
\[
L=\Hom_{\co_\kk}(\mathfrak{a}_0,\mathfrak{a})
\]
into a projective $\co_E$-module of rank one, and the $\co_\kk$-hermitian form on $L$ defined by  \cite[(2.1.5)]{BHKRY} satisfies
$
\langle \alpha x_1,x_2\rangle = \langle x_1, \overline{\alpha} x_2\rangle
$
 for all $\alpha\in \co_E$ and $x_1,x_2\in L$.    It is a formal consequence of this that the $E$-vector space $\mathscr{V}=L\otimes_\Z\Q$ carries an $E$-hermitian form
\[
\langle-,-\rangle_\mathrm{big} : \mathscr{V} \times \mathscr{V}  \to E,
\]
uniquely determined by the property
\[
\langle x_1,x_2\rangle = \mathrm{Tr}_{E/\kk}\langle x_1,x_2\rangle_\mathrm{big}.
\]
This hermitian form has signature $(0,1)$ at $\varphi^\mathrm{sp}|_F$, and signature $(1,0)$ at all other archimedean places of $F$.

From the $E$-hermitian form we obtain an $F$-valued quadratic form
$\mathscr{Q}(x) = \langle x,x\rangle_\mathrm{big}$ on $\mathscr{V}$ with  signature $(0,2)$ at $\varphi^\mathrm{sp}|_F$, and signature $(2,0)$ at all other archimedean places of $F$.  The $\Q$-quadratic form
\begin{equation}\label{trace form}
Q(x) = \mathrm{Tr}_{F/\Q}\mathscr{Q}(x)
\end{equation}
is $\Z$-valued on $L\subset \mathscr{V}$, and agrees with the quadratic form  of \S \ref{ss:convolution}.
Let \[ \omega_L : \SL_2(\Z) \to \Aut_\C(S_L)\] be the Weil representation on the space $S_L=\C[L'/L]$, where $L'=\mathfrak{d}_\kk^{-1}L$ is the dual lattice of $L$ relative to the $\Z$-bilinear form (\ref{bilinear}).

Write each $\vec{\tau}\in F_\C$ in the form $\vec{\tau}=\vec{u}+i \vec{v}$ with $\vec{u},\vec{v}\in F_\R$, and set
\[
\mathcal{H}_F = \{ \vec{\tau}  \in F_\C : \vec{v} \mbox{ is totally positive} \}.
\]
Every Schwartz function $\phi \in S ( \widehat{\mathscr{V}})$ determines  an incoherent Hilbert modular Eisenstein series
\begin{equation}\label{basic eisenstein}
E( \vec{\tau}, s ,\phi) =  \sum_{ \alpha\in F } E_\alpha(\vec{v},s,\phi) \cdot q^\alpha
\end{equation}
on $\mathcal{H}_F$, as in  \cite[(4.4)]{BKY} and \cite[\S 6.1]{AGHMP-2}.
If we identify
\[
S_L = \C[L' /L] \subset S ( \widehat{\mathscr{V}})
\]
as the space of $\widehat{L}$-invariant functions supported on $\widehat{L}'$, then (\ref{basic eisenstein}) can be viewed as a
function $E(\vec{\tau},s)$ on $\mathcal{H}_F$ taking values in the complex dual $S_L^\vee$.

We quickly recall the construction of (\ref{basic eisenstein}).
If $v$ is an arichmedean place of $F$, denote by  $( \mathscr{C}_v , \mathscr{Q}_v)$ the unique positive definite rank $2$ quadratic space over $F_v$.
Set $\mathscr{C}_\infty = \prod_{v\mid \infty} \mathscr{C}_v$.  The  rank $2$ quadratic space
\[
\mathscr{C}  = \mathscr{C}_\infty \times \widehat{\mathscr{V}}
\]
over $\A_F$  is \emph{incoherent}, in the sense that it is not the adelization of any $F$-quadratic space.  In fact, $\mathscr{C}$ is isomorphic to $\mathscr{V}$ everywhere locally, except at the unique archimedean place $\varphi^\mathrm{sp}|_F$ at which $\mathscr{V}$ is negative definite.

Let  $\psi_\Q: \Q\backslash \A \to \C^\times$ be the standard additive character, and define
\[
\psi_F: F\backslash \A_F \to \C^\times
\]
by $\psi_F = \psi_\Q \circ \mathrm{Tr}_{F/\Q}$.     Denote by $I( s, \chi_E)$
the degenerate principal series representation of $\SL_2(\A_F)$ induced from the character $\chi_E | \cdot |^s$
on the subgroup $B\subset \SL_2$ of upper triangular matrices.  Thus $I(s,\chi_E)$ consists of all smooth functions $\Phi(g,s)$ on $\SL_2(\A_F)$
satisfying the transformation law
\[
\Phi \left(   \left( \begin{matrix}     a& b \\ & a^{-1}      \end{matrix}  \right)   g    , s   \right) = \chi_E(a) |a|^{s+1}  \Phi(g,s).
\]
The Weil representation $\omega_{\mathscr{C}}$ determined by the character $\psi_F$ defines  an action of
$\SL_2(\A_F)$ on $S( \mathscr{C} )$, and for any  Schwartz function
\[
\phi_\infty \otimes \phi  \in S(\mathscr{C}_\infty) \otimes S(\widehat{\mathscr{V}}) \iso S( \mathscr{C} )
\]
the function
\begin{equation}\label{SW construction}
\Phi (g,0) =  \omega_{\mathscr{C}}(g) ( \phi_\infty \otimes \phi ) (0)
\end{equation}
lies in the induced representation $I( 0, \chi_E)$.
It extends uniquely to a standard section $\Phi ( g , s )$ of $I(s,\chi_E)$, which determines an  Eisenstein series
\[
E(g ,s , \phi_\infty \otimes \phi ) = \sum_{  \gamma \in  B(F) \backslash \SL_2(F)   }\Phi ( \gamma g ,s)
\]
in the variable  $g\in \SL_2(\A_F)$.

We always choose $\phi \in S_L \subset S(\mathscr{V})$, and take the archimedean component  $\phi_\infty$ of our Schwartz function to be the Gaussian distribution
\[
\phi_\infty^{\bm{1}} = \otimes \phi^{\bm{1}}_v \in  \bigotimes_{v\mid \infty} S( \mathscr{C}_v)
\]
defined by $\phi^{\bm{1}}_v(x) = e^{-2\pi \mathscr{Q}_v(x)}$, so that the resulting Eisenstein series
 \[
 E(\vec{\tau} ,s,\phi)   =    \frac{1}{ \sqrt{  \mathrm{Nm} (\vec{v}) } }  \cdot E(g_{\vec{\tau}},s, \phi_\infty^{\bm{1}} \otimes \phi)
 \]
 has parallel weight $1$.   Here
 \[
 g_{\vec{\tau}} = \left( \begin{matrix}  1 & \vec{u} \\ 0 & 1   \end{matrix}\right)
 \left( \begin{matrix}  \sqrt{ \vec{v} }  &  \\ & 1/\sqrt{ \vec{v} } \end{matrix}\right) \in \SL_2( F_\R)
 \]
and $\mathrm{Nm}:F_\R^\times \to \R^\times$ is the norm.

A choice of ordering of the embeddings  $F\to \R$  fixes an isomorphism of $\mathcal{H}_F$ with the $n$-fold product of the complex upper half-plane with itself, and the diagonal inclusion $\mathcal{H} \hookrightarrow \mathcal{H}_F$ is independent of the choice of ordering.
 By restricting our Eisenstein series to the diagonal we obtain an $S_L^\vee$-valued function
\[
E(\tau,s) = E(\vec{\tau},s)|_\mathcal{H}
\]
in the variable $\tau\in\mathcal{H}$,   which transforms like a modular form of weight $n$ and representation $\omega_L^\vee$ under the full modular group $\SL_2(\Z)$.

Given a cusp form $\tilde{g} \in S_n(\overline{\omega}_L)$ valued in $S_L$, consider the Petersson inner product
\begin{equation}\label{fake L}
\langle E(s) , \tilde{g} \rangle_\mathrm{Pet} =  \int_{ \SL_2(\Z) \backslash \mathcal{H}  } \big\{  \overline{ \tilde{g} (\tau)}  ,  E(\tau,s)    \big\}  \, \frac{du\, dv}{v^{2-n}} ,
\end{equation}
where  $\{ \cdot\, ,\cdot \} : S_L \times S_L^\vee \to \C$ is the tautological pairing.
This is an unnormalized version of the  \emph{generalized $L$-function}
\[
\mathcal{L}(s, \tilde{g}) =   \Lambda(s+1,\chi_E) \cdot \langle E(s) , \tilde{g} \rangle_\mathrm{Pet}
\]
of  \cite[(1.2)]{BKY} or \cite[\S 6.3]{AGHMP-2}.

Let $F_+\subset F$ be the subset of totally positive elements.
The Eisenstein series $E(\vec{\tau},s)$   satisfies a functional equation in $s\mapsto -s$, forcing  it to vanish at $s=0$.
As in \cite[Proposition 4.6]{BKY} and \cite[\S 6.2]{AGHMP-2}, we can extract from the  central derivative $E'(\vec{\tau} , 0)$  a formal $q$-expansion
\[
a_F(0)    +  \sum_{ \alpha\in F_+ }   a_F(\alpha)   \cdot q^\alpha
\]
 If  $\alpha\in F_+$ then $E'_\alpha(\vec{v},0,\phi)$ is independent of $\vec{v}$, and we define $a_F(\alpha) \in S_L^\vee$ by
\[
a_F(\alpha,\phi) = \Lambda(0,\chi_E) \cdot E'_\alpha(\vec{v}, 0,\phi).
\]
We define $a_F(0)\in S_L^\vee$ by
\[
a_F(0,\phi) = \Lambda(0,\chi) \cdot  E'_0(\vec{v} ,0,\phi) -  \Lambda(0,\chi_E) \cdot   \phi(0)  \log \mathrm{Nm}(\vec{v})  .
\]
Again, this is independent of $\vec{v}$.

\begin{remark}
For  notational simplicity, we often denote by $a_F(\alpha ,\mu)$ the value of $a_F(\alpha) : S_L \to \C$ at the characteristic function of a coset $\mu\in L'/L$.
\end{remark}

%
%

 For any  nonzero $\alpha\in F$, define
 \[
 \mathrm{Diff}(   \mathscr{C} ,\alpha ) = \{ \mbox{places $v$ of $F$} :  \mathscr{C}_v \mbox{ does not represent }\alpha \}.
 \]
 This is a finite set of odd cardinality, and any $v\in  \mathrm{Diff}(   \mathscr{C} ,\alpha )$ is necessarily  nonsplit in $E$.
 We are really only interested in this set when $\alpha\in F_+$.    As $ \mathscr{C}$ is positive definite at all archimedean places, for such $\alpha$ we have
  \[
  \mathrm{Diff}(   \mathscr{C} ,\alpha ) = \{ \mbox{primes } \mathfrak{p} \subset \co_F : \mathscr{V}_\mathfrak{p} \mbox{ does not represent }\alpha \}.
 \]

 We will need explicit formulas for all $a_F(\alpha,\mu)$ with $\alpha\in F_+$, but only for the  trivial  coset $\mu=0$.
 These are provided by the following proposition.

 \begin{proposition}\label{prop:fourier}
Suppose $\alpha\in F_+$.
\begin{enumerate}
\item
If $| \mathrm{Diff}(\mathscr{C} , \alpha)| >1$ then $a_F(\alpha)=0$.
\item
If $\mathrm{Diff}(\mathscr{C} , \alpha) = \{ \mathfrak{p}\}$, then
\[
a_F(\alpha,0) = -2^{r-1} \cdot \rho(\alpha \mathfrak{d}_F \mathfrak{p}^{-\epsilon_\mathfrak{p}}) \cdot \ord_\mathfrak{p}(\alpha \mathfrak{p} \mathfrak{d}_F ) \cdot \log(\mathrm{N}(\mathfrak{p})),
\]
where the notation is as follows:
 $r$ is the number of places of $F$ ramified in $E$ (including all archimedean places), $\mathfrak{d}_F\subset \co_F$ is the different of $F$,  and
\[
\epsilon_\mathfrak{p} = \begin{cases}
1 & \mbox{if $\mathfrak{p}$ is inert in $E$} \\
0 & \mbox{if $\mathfrak{p}$ is ramified in $E$}.
\end{cases}
\]
Moreover, for any fractional $\co_F$-ideal $\mathfrak{b} \subset F$ we have set
 \[
 \rho(\mathfrak{b}) = | \{ \mbox{ideals }\mathfrak{B} \subset \co_E :  \mathfrak{B}\overline{\mathfrak{B}} = \mathfrak{b} \co_E \} |.
 \]
 In particular, $\rho(\mathfrak{b})=0$ unless $\mathfrak{b} \subset \co_F$.
\end{enumerate}
 \end{proposition}

\begin{proof}
Up to a change of notation, this is \cite[Proposition 4.2.1]{Ho1}, whose proof amounts to collecting together calculations of \cite{Yang}.  More general formulas can be found in  \cite[\S 7.1]{AGHMP-2} and  \cite[\S 4.6]{HY}.
\end{proof}

\begin{proposition}\label{prop:constant eisenstein}
Assume that the discriminants of $\kk$ and $F$ are relatively prime.  For any $\mu\in L'/L$ we have 
\[
a_F(0,\mu)  =  \begin{cases}
-2   \Lambda'(0,\chi_E )  & \mbox{if $\mu=0$} \\
0 & \mbox{otherwise.}
\end{cases}
\]
\end{proposition}

\begin{proof}
Let $\Phi_\mu = \prod_\mathfrak{p} \Phi_{\mu,\mathfrak{p}}$ be the standard section of $I(s,\chi_E)$ determined by the characteristic function $\phi_\mu \in S_L \subset S(\mathscr{V})$
of $\mu \in L'/L$.
According to \cite[Proposition 6.2.3]{AGHMP-2}, we then have
\begin{equation}\label{course constant term}
a_F(0,\mu) = -2 \phi_\mu(0) \Lambda'(0,\chi_E) 
- \Lambda(0,\chi_E) \cdot    \frac{d}{ds} \Big( \prod_\mathfrak{p} M_\mathfrak{p}(s,\phi_\mu) \Big) \Big|_{s=0},
\end{equation}
where the product is over all finite places $\mathfrak{p}$ of $F$, and the local factors on the right have the form
\begin{equation}\label{local M}
M_\mathfrak{p}(s,\phi_\mu) 
= c_\mathfrak{p} \cdot \frac{L_\mathfrak{p}(s+1,\chi_E) }{ L_\mathfrak{p} (s,\chi_E)} \cdot W_{0,\mathfrak{p}}(s,\Phi_\mu)
\end{equation}
for some constants $c_\mathfrak{p}$ independent of $s$.  Here, setting
 \[
w=  \left(\begin{smallmatrix}  0&-1\\ 1 & 0 \end{smallmatrix}\right)  ,\quad     n(b) = \left(\begin{smallmatrix}  1&b\\ 0 & 1 \end{smallmatrix}\right) , 
 \]
the function
\[
W_{0,\mathfrak{p}}(s,\Phi_\mu) 
= \int_{F_\mathfrak{p}} \Phi_{\mu,\mathfrak{p} } \left(  wn(b)   ,s \right) \, db
\] 
 is the value of the local Whittaker function $W_{0,\mathfrak{p}}(g,s,\Phi_\mu)$ at the identity in $\SL_2(F_\mathfrak{p})$.
 Our goal is to prove that  $M_\mathfrak{p}(s,\phi_\mu)$ is independent of $s$, and hence  both the particular value of $c_\mathfrak{p}$   and the choice of Haar measure on $F_\mathfrak{p}$ are irrelevant to us.

Fix a prime $\mathfrak{p}\subset \co_F$, and let $p$ be the rational prime below it.  
We may identify $\mathscr{V}_\mathfrak{p} \iso E_\mathfrak{p}$ in such a way that $L_\mathfrak{p}\iso\co_{E,\mathfrak{p}}$, and so that the $F_\mathfrak{p}$-valued quadratic form $\mathscr{Q}$ on $\mathscr{V}_\mathfrak{p} \iso E_\mathfrak{p}$ becomes
\[
\mathscr{Q}(x) = \beta x\overline{x}
\]
for some $\beta \in F_\mathfrak{p}^\times$.  
If  $\mathfrak{d}_F$ denotes the different of $F/\Q$, then
\begin{equation}\label{conductor calc}
\beta  \co_{F,\mathfrak{p}} = \mathfrak{d}_F^{-1} \co_{F,\mathfrak{p}} .
\end{equation}
Indeed, let $\mathfrak{d}_E$ be the different of $E/\Q$.
The lattice $L_\mathfrak{p}'=\mathfrak{d}_\kk^{-1} \co_{E,\mathfrak{p}}$ is  the dual lattice of $\co_{E,\mathfrak{p}}$ relative to the $\Q_p$-bilinear form $[x,y]=\mathrm{Tr}_{E_\mathfrak{p}/\Q_p}(\beta x\overline{y})$, which implies the first equality in
\[
\beta^{-1} \co_{E,\mathfrak{p}}= \mathfrak{d}_E \mathfrak{d}_\kk^{-1} \co_{E,\mathfrak{p}} =  \mathfrak{d}_F \co_{E,\mathfrak{p}}.
\]  
The second equality is a consequence of our assumption that the discriminants of $\kk$ and $F$ are relatively prime.

If we endow   $\mathscr{V}_\mathfrak{p} = E_\mathfrak{p}$ with the rescaled quadratic form
\[
\mathscr{Q}^\sharp(x) \define  \beta^{-1} \mathscr{Q}(x) = x\overline{x},
\]
and define a new additive character
\[
\psi^\sharp_{F,\mathfrak{p}}(x) \define \psi_{F,\mathfrak{p}}(\beta x)
\]
(unramified by (\ref{conductor calc})), we obtain a new  Weil representation
\[
\omega^\sharp : \SL_2(F_\mathfrak{p}) \to \Aut( S(\mathscr{V}_\mathfrak{p} ) ),
\]
and hence, as in (\ref{SW construction}),  a function
\[
S(\mathscr{V}_\mathfrak{p} )  \map{ \phi \mapsto \Phi_\mathfrak{p}^\sharp(s,g) } I_\mathfrak{p}(s, \chi_E)
\]
defined by first setting $\Phi_\mathfrak{p}^\sharp(0,g) =   \omega^\sharp(g)\phi (0)$, 
and then extending  to a standard section.

The local Schwartz function $\phi_{\mu,\mathfrak{p}} \in S(\mathscr{V}_\mathfrak{p})$
now  determines a standard section $\Phi^\sharp_{\mu,\mathfrak{p}}(g,s)$  of $I_\mathfrak{p}(s, \chi_E)$, and 
explicit formulas for the Weil representation, as in  \cite[(4.2.1)]{HY}, show that 
\[
\int_{F_\mathfrak{p}} \Phi_{\mu,\mathfrak{p} } \left(  wn(b)   ,s \right) \, db
=\int_{F_\mathfrak{p}} \Phi^\sharp_{\mu,\mathfrak{p} } \left(  wn(b)   ,s \right) \, db.
\]
What our discussion shows is that there is no harm in rescaling the quadratic form on $\mathscr{V}_\mathfrak{p}$ to make $\beta=1$, and simultaneously modifying the additive character $\psi_{F,\mathfrak{p}}$ to make it unramified.

After this rescaling, one can easily deduce explicit formulas for  $W_{0,\mathfrak{p}}(s,\Phi_\mu)$   from the literature.
Indeed,  if the local component $\mu_\mathfrak{p}\in L_\mathfrak{p}' / L_\mathfrak{p}$ is zero, then the calculations found in  \cite[\S 2]{Yang} imply that 
\[
W_{0,\mathfrak{p}}(s,\Phi_\mu)  =  \frac{ L_\mathfrak{p} (s,\chi_E)}{L_\mathfrak{p}(s+1,\chi_E) } 
\]
up to scaling by a nonzero constant independent of $s$. 
If instead $\mu_\mathfrak{p} \neq 0$ then  $\mathfrak{p}$ is ramified in $E$ (and in particular $p>2$), and it follows from the calculations found in the proof of \cite[Proposition 4.6.4]{HY} that  $W_{0,\mathfrak{p}}(s,\Phi_\mu)  = 0$.
In any case (\ref{local M}) is independent of $s$ for every $\mathfrak{p}$, and so the derivative in (\ref{course constant term}) vanishes.
\end{proof}


\subsection{A preliminary central derivative formula}
\label{ss:big harmonic intersection}


The entirety of \S \ref{ss:big harmonic intersection} is devoted to proving Theorem \ref{thm:big harmonic lift}, which a  big CM analogue of Theorem \ref{thm:BHY main}.
The proof will make essential use of the calculations of  \cite{Ho1, Ho2, BKY}.

We assume $n\ge 3$ throughout \S \ref{ss:big harmonic intersection}. 
This allows us to make use of the distinguished harmonic forms 
\[
f_m \in H_{2-n}(\overline{\omega}_L)^\Delta
\]
(for $m>0$) characterized by (\ref{simple harmonic}).

\begin{theorem}\label{thm:big harmonic lift}
Assume that the discriminants of $\kk/\Q$ and $F/\Q$ are odd and relatively prime, and fix  a positive integer $m$.
If $f  =f_m$ is the harmonic form above, and $\widehat{\mathcal{Z}}$ is the linear function (\ref{fake theta}),
  then
\[
 \frac{ n\cdot  [  \widehat{\mathcal{Z}} ( f )  :  \mathcal{Y}_\mathrm{big}  ]  }  {  \deg_\C( \mathcal{Y}_\mathrm{big} )    }
  +  2  c_f^+(0,0)   \frac{  \Lambda'(0,\chi_E )   }{    \Lambda(0,\chi_E )   }
 =  -   \frac{d}{ds} \langle E(s) , \xi(f)  \rangle_\mathrm{Pet} \big|_{s=0} .
\]
\end{theorem}

For the form $f=f_m$  we have
\[
\widehat{\mathcal{Z}}(f) =  \widehat{\mathcal{Z}}_\Kra^\tot(m)   = \big( \mathcal{Z}_\Kra^\tot(m) ,  \Theta^\mathrm{reg}(f_m)   \big)
  \in  \widehat{\mathrm{Ch}}^1( \mathcal{S}_\Kra^*) ,
\]
where the Green function $\Theta^\mathrm{reg}(f_m)$  for the divisor $\mathcal{Z}_\Kra^\tot(m)$  is constructed in \cite[\S 7]{BHKRY} as a  regularized theta lift.
The arithmetic degree appearing in Theorem \ref{thm:big harmonic lift}  decomposes as
\begin{align}\label{intersection decomp}
[  \widehat{\mathcal{Z}} ( f )  :  \mathcal{Y}_\mathrm{big}  ]     &  =
\sum_{ \mathfrak{p} \subset \co_\kk } \log(\mathrm{N}(\mathfrak{p} ) )  \sum_{ y \in  ( \mathcal{Z}^\tot_\Kra(m) \cap   \mathcal{Y}_\mathrm{big}   )   ( \F_\mathfrak{p}^\alg ) }
\frac{ \mathrm{length}(\co_y)  }{| \Aut(y) | }  \\
& \quad +
\sum_{ y\in \mathcal{Y}_\mathrm{big}(\C) }   \frac{ \Theta^\mathrm{reg}(f_m )(y) }{ |\Aut(y) | }   \nonumber
\end{align}
 where $\F_\mathfrak{p} = \co_\kk/\mathfrak{p}$,  and $\co_y$ is the \'etale local ring of
\begin{equation}\label{naive intersection}
\mathcal{Z}^\tot_\Kra(m) \cap   \mathcal{Y}_\mathrm{big}   \define  \mathcal{Z}^\tot_\Kra(m) \times_{\mathcal{S}^*_\Kra}   \mathcal{Y}_\mathrm{big}
\end{equation}
 at $y$.  The final summation is over all complex points of $\mathcal{Y}_\mathrm{big}$, viewed as an $\co_\kk$-stack.
We will see that the terms on the right hand side of (\ref{intersection decomp}) are intimately related to the Eisenstein series coefficients  $a_F(\alpha)$ of \S \ref{ss:weird L}.

We first study the structure of the stack-theoretic intersection (\ref{naive intersection}).
Suppose  $S$ is a connected $\co_\Phi$-scheme, and
\[
(A_0,A) \in   \big(  \mathcal{M}_{(1,0)} \times_{\co_\kk} \mathcal{CM}_\Phi \big)(S)
\]
is an $S$-point.
The $\co_\kk$-module $\Hom_{\co_\kk}(A_0,A)$ carries an $\co_\kk$-hermitian form $\langle-,-\rangle$ defined by  \cite[(2.5.1)]{BHKRY}.
The construction of this hermitian form only uses the underlying point of $\mathcal{S}_\Kra$, and not the action $\co_E \to \End_{\co_\kk}(A)$.   As in  \cite[\S 3.2]{Ho2}, the extra action of $\co_E$ makes $\Hom_{\co_\kk}(A_0,A)$ into a projective $\co_E$-module, and there is a totally positive definite $E$-hermitian form $\langle-,-\rangle_\mathrm{big}$ on
\begin{equation}\label{nearby hermitian}
\mathscr{V}(A_0,A) = \Hom_{\co_\kk}(A_0,A) \otimes_\Z\Q
\end{equation}
characterized by the relation
\[
\langle x_1,x_2\rangle = \mathrm{Tr}_{E/\kk} \langle x_1,x_2\rangle_\mathrm{big}.
\]
for all $x_1,x_2\in \Hom_{\co_\kk}(A_0,A)$.

Fix an $\alpha\in F_+$.   Recalling that
\begin{equation}\label{y big}
\mathcal{Y}_\mathrm{big} \subset  \mathcal{M}_{(1,0)} \times_{\co_\kk} \mathcal{CM}_\Phi
\end{equation}
as an open and closed substack, for any $\co_\Phi$-scheme $S$  let $\mathcal{Z}_\mathrm{big}(\alpha)(S)$ be  the groupoid of triples $(A_0,A ,x)$,  in which
\begin{itemize}
\item
$(A_0,A)\in \mathcal{Y}_\mathrm{big}(S)$,
\item
 $x\in \Hom_{\co_\kk}(A_0,A)$ satisfies $\langle x,x\rangle_\mathrm{big}=\alpha$.
 \end{itemize}
 This  functor is represented by an $\co_\Phi$-stack $ \mathcal{Z}_\mathrm{big}(\alpha)$,
 and the evident forgetful morphism
\[
\mathcal{Z}_\mathrm{big}(\alpha) \to \mathcal{Y}_\mathrm{big}
\]
 is finite and unramified.

 This construction is entirely analogous to the construction of the special divisors
$
\mathcal{Z}^\tot_\Kra(m) \to \mathcal{S}_\Kra
$
of \cite{BHKRY}.   In fact, directly from the definitions, if $S$ is an $\co_\Phi$-scheme  an $S$-point
\[
(A_0,A, x)  \in \big(  \mathcal{Z}^\tot_\Kra(m) \cap  \mathcal{Y}_\mathrm{big}  \big)(S)
\]
 consists of a pair  $(A_0,A) \in \mathcal{Y}_\mathrm{big}(S)$ and an $x\in \Hom_{\co_\kk}(A_0,A)$ satisfying $m=\langle x,x\rangle$.
From this it is clear that there is an isomorphism
\begin{equation}\label{stack decomp}
\mathcal{Z}^\tot_\Kra(m) \cap  \mathcal{Y}_\mathrm{big}
\iso   \bigsqcup_{ \substack{  \alpha\in F_+ \\ \mathrm{Tr}_{F/\Q}(\alpha)=m }}  \mathcal{Z}_\mathrm{big}(\alpha),
\end{equation}
defined by sending the  triple $(A_0,A,x)$ to the same triple, but now viewed as an $S$-point of the stack
$\mathcal{Z}_\mathrm{big}(\alpha)$ determined by $\alpha = \langle x,x\rangle_\mathrm{big}$.

\begin{proposition}\label{prop:cycle degree}
For each $\alpha \in F_+$ the stack $\mathcal{Z}_\mathrm{big}(\alpha)$ is either empty, or has  dimension $0$ and is supported at a single prime of $\co_\Phi$.  Moreover,
\begin{enumerate}
\item
If $ | \mathrm{Diff}( \mathscr{C} , \alpha)  | >1$ then $\mathcal{Z}_\mathrm{big}(\alpha)=\emptyset$.
\item
Suppose that $\mathrm{Diff}(\mathscr{C} , \alpha ) = \{ \mathfrak{p}\}$ for a single prime $\mathfrak{p} \subset \co_F$, let $\mathfrak{q} \subset \co_E$ be the unique prime above it, and denote by $\mathfrak{q}_\Phi \subset \co_\Phi$ the corresponding prime under the isomorphism $\varphi^\mathrm{sp} : E \iso E_\Phi$.  Then  $\mathcal{Z}_\mathrm{big}(\alpha)$  is supported at the prime $\mathfrak{q}_\Phi$, and satisfies
\[
 \sum_{  y \in \mathcal{Z}_\mathrm{big}(\alpha) (\F_{\mathfrak{q}_\Phi}^\alg) }
 \frac{ 1  }{ |\Aut(y) | }
 = \frac{h_\kk}{w_\kk} \cdot \rho( \alpha \mathfrak{d}_F \mathfrak{p}^{-\epsilon_\mathfrak{p}} ),
\]
where $\F_{\mathfrak{q}_\Phi}$ is the residue field of  $\mathfrak{q}_\Phi$,  and $\epsilon_\mathfrak{p}$ and $\rho$ are as in  Proposition \ref{prop:fourier}.
Moreover,  the \'etale local rings at all geometric points  \[y\in  \mathcal{Z}_\mathrm{big}(\alpha) (\F_{\mathfrak{q}_\Phi}^\alg)\] have the same length
\[
\mathrm{length} (\co_y)  =    \ord_\mathfrak{p}( \alpha \mathfrak{p} \mathfrak{d}_F)
\cdot
 \begin{cases}
1/2  & \mbox{if $E_\mathfrak{q} / F_\mathfrak{p}$ is unramified} \\
1 & \mbox{otherwise.}
\end{cases}
\]
\end{enumerate}
\end{proposition}

\begin{proof}
This is essentially contained in \cite[\S 3]{Ho1}.  In that work we studied the $\co_\Phi$-stack   $\mathcal{Z}_\Phi(\alpha)$ classifying triples $(A_0,A,x)$ exactly as in the definition of  $\mathcal{Z}_\mathrm{big}(\alpha)$, except we allowed the pair $(A_0,A)$ to be any point of $\mathcal{M}_{(1,0)} \times_{\co_\kk} \mathcal{CM}_\Phi$ rather than  a point of the substack (\ref{y big}).
Thus we have a cartesian  diagram
\[
\xymatrix{
{ \mathcal{Z}_\mathrm{big}(\alpha) } \ar[r]\ar[d]  &   {\mathcal{Z}_\Phi(\alpha) } \ar[d] \\
{  \mathcal{Y}_\mathrm{big} }  \ar[r]  &  { \mathcal{M}_{(1,0)} \times_{\co_\kk} \mathcal{CM}_\Phi .}
}
\]
As the bottom horizontal arrow is an open and closed immersion, so is the top horizontal arrow.  In other words, our $\mathcal{Z}_\mathrm{big}(\alpha)$ is a union of connected components of the stack $\mathcal{Z}_\Phi(\alpha)$ of \cite{Ho1}.

\begin{lemma}
Each $\mathcal{Z}_\Phi(\alpha)$ has dimension $0$.
If $y$ is a geometric point of $\mathcal{Z}_\Phi(\alpha)$ corresponding to a triple $(A_0,A,x)$ over $k(y)$, then
$k(y)$ has nonzero characteristic,   $A_0$ and $A$ are supersingular, and the $E$-hermitian space (\ref{nearby hermitian}) has dimension one.    Moreover, if $\mathfrak{p} \subset \co_F$ denotes the image of $y$ under the composition
\begin{equation}\label{prime image}
 \mathcal{Z}_\Phi(\alpha) \to \Spec(\co_\Phi) \iso \Spec(\co_E) \to \Spec(\co_F)
\end{equation}
(the isomorphism is $\varphi^\mathrm{sp} : E\iso E_\Phi$), then $\mathfrak{p}$ is nonsplit in $E$, and the following are equivalent:
\begin{itemize}
\item The geometric point $y$ factors through the open and closed substack \[ \mathcal{Z}_\mathrm{big}(\alpha) \subset \mathcal{Z}_\Phi(\alpha).\]
\item The $E$-hermitian space (\ref{nearby hermitian})   is isomorphic to $\mathscr{V}$ everywhere locally except at $\mathfrak{p}$ and $\varphi^\mathrm{sp}|_F$.
\end{itemize}
\end{lemma}

\begin{proof}
This is an easy consequence of \cite[Proposition 3.4.5]{Ho1} and  \cite[Proposition 3.5.2]{Ho1}.
The only part that requires  explanation is the final claim.

Fix a connected component
\[
\mathcal{B} \subset  \mathcal{M}_{(1,0)} \times_{\co_\kk} \mathcal{CM}_\Phi.
\]
As in \cite[\S 3.4]{Ho1}, for  each complex point  $y=(A_0,A) \in \mathcal{B} (\C)$ one can construct from  the Betti realizations of $A_0$ and $A$  an $E$-hermitian space
\[
\mathscr{V}(\mathcal{B})= \Hom_\kk \big(  H_1(A_0(\C) , \Q) , H_1(A(\C) ,\Q)    \big)
\]
 of dimension $1$.   This hermitian space has signature $(0,1)$ at $\varphi^\mathrm{sp}|_F$, and signature $(1,0)$ at all other archimedean places of $F$.  Moreover, as in Remark \ref{rem:betti hermitian}, this hermitian space depends only on the connected component $\mathcal{B}$, and not on the particular complex point $y$.  The open and closed substack
\[
\mathcal{Y}_\mathrm{big} \subset  \mathcal{M}_{(1,0)} \times_{\co_\kk} \mathcal{CM}_\Phi
\]
can be characterized as the union of all  components $\mathcal{B}$ for which $\mathscr{V}(\mathcal{B}) \iso \mathscr{V}$.

So, suppose we have a geometric point $y=(A_0,A,x)$ of $\mathcal{Z}_\Phi(\alpha)$,  and denote by
\[
\mathcal{B} \subset  \mathcal{M}_{(1,0)} \times_{\co_\kk} \mathcal{CM}_\Phi
\]
the connected component containing the underlying point $y=(A_0,A)$.
The content of \cite[Proposition 3.4.5]{Ho1} is that the hermitian space   (\ref{nearby hermitian})  is isomorphic to $\mathscr{V}(\mathcal{B})$ everywhere locally except at $\mathfrak{p}$ and $\varphi^\mathrm{sp}|_F$.
From this we deduce the equivalence of the following statements:
\begin{itemize}
\item The geometric point $y\to  \mathcal{Z}_\Phi(\alpha)$ factors through $\mathcal{Z}_\mathrm{big}(\alpha).$
\item The underlying point $y \to   \mathcal{M}_{(1,0)} \times_{\co_\kk} \mathcal{CM}_\Phi$  factors through $\mathcal{Y}_\mathrm{big}$.
\item  The hermitian spaces $\mathscr{V}(\mathcal{B})$ and  $\mathscr{V}$ are isomorphic.
\item The $E$-hermitian space (\ref{nearby hermitian})   is isomorphic to $\mathscr{V}$ everywhere locally except at $\mathfrak{p}$ and $\varphi^\mathrm{sp}|_F$.
\end{itemize}
  \end{proof}

Now suppose that $\mathcal{Z}_\mathrm{big}(\alpha)$ is nonempty.  If we fix a geometric point
$y= (A_0,A,x)$ as above, the vector  $x\in \Hom_{\co_\kk}(A_0,A) $
satisfies $\langle x,x\rangle_\mathrm{big}=\alpha$, and hence (\ref{nearby hermitian}) represents $\alpha$.  The above lemma now implies that $\mathscr{V}$ represents  $\alpha$ everywhere locally except at
$\mathfrak{p}$ and $\varphi^\mathrm{sp}|_F$, where $\mathfrak{p}$ is the image of $y$ under (\ref{prime image}).
  From this it follows first $\mathrm{Diff}(\mathscr{C} , \alpha ) = \{ \mathfrak{p}\}$, and then that all geometric points of $\mathcal{Z}_\mathrm{big}(\alpha)$ have the same image under (\ref{prime image}),  and   lie above the same prime  $\mathfrak{q}_\Phi \subset \co_\Phi$  characterized  as in the statement of Proposition \ref{prop:cycle degree}.
In particular, if $ | \mathrm{Diff}(\mathscr{C} , \alpha)  | >1$ then $\mathcal{Z}_\mathrm{big}(\alpha)=\emptyset$.

It remains to prove part  (2) of the proposition.  For this we need the following lemma.

\begin{lemma}
Assume   that $\mathrm{Diff}(\mathscr{C} , \alpha ) = \{ \mathfrak{p}\}$ for some prime $\mathfrak{p}\subset \co_F$, and let  $\mathfrak{q}\subset \co_E$ be the unique prime above it.
The open and closed substack $\mathcal{Z}_\mathrm{big}(\alpha) \subset \mathcal{Z}_\Phi(\alpha)$ is equal to the union of all connected components of $ \mathcal{Z}_\Phi(\alpha)$ that are supported at the prime $\mathfrak{q}_\Phi$.
\end{lemma}

\begin{proof}
We have already seen that every geometric point of $\mathcal{Z}_\mathrm{big}(\alpha)$ lies above the prime $\mathfrak{q}_\Phi$, and so it suffices to prove that every geometric point of $\mathcal{Z}_\Phi(\alpha)$ lying above the prime $\mathfrak{q}_\Phi$ factors through $\mathcal{Z}_\mathrm{big}(\alpha)$.  Let $y \to \mathcal{Z}_\Phi(\alpha)$ be such a point.

 If $y$ corresponds to the triple $(A_0,A,x)$, then $x\in \Hom_{\co_\kk}(A_0,A)$ satisfies $\langle x,x\rangle_\mathrm{big}=\alpha$, and hence (\ref{nearby hermitian}) represents $\alpha$.  But the assumption that $\mathrm{Diff}(\mathscr{C} , \alpha ) = \{ \mathfrak{p}\}$ implies that $\mathscr{V}$ represents $\alpha$ everywhere locally except at $\mathfrak{p}$ and $\varphi^\mathrm{sp}|_F$, and it follows from this that $\mathscr{V}$ and (\ref{nearby hermitian}) are isomorphic locally everywhere except at $\mathfrak{p}$ and $\varphi^\mathrm{sp}|_F$.  By the previous lemma, this implies that $y$ factors through $\mathcal{Z}_\mathrm{big}(\alpha)$.
\end{proof}

With this last lemma in hand, all parts of (2) follow from the corresponding statements for $\mathcal{Z}_\Phi(\alpha)$ proved in \cite[Theorem 3.5.3]{Ho1} and \cite[Theorem 3.6.2]{Ho1}.

%
%
%
%
%
%
%
%
\end{proof}

\begin{proposition}\label{prop:arakelov degree}
For every $\alpha\in F_+$ we have
\[
\sum_{ \mathfrak{p} \subset \co_\kk }
\frac{ n  \cdot \log(\mathrm{N}(\mathfrak{p}))  }{\deg_\C (\mathcal{Y}_\mathrm{big}) }
 \sum_{  y \in \mathcal{Z}_\mathrm{big}(\alpha) (\F_{\mathfrak{p}}^\alg) }
 \frac{ \mathrm{length} (\co_y)  }{ |\Aut(y) | }
 = -  \frac{ a_F(\alpha,0) }{  \Lambda(0,\chi_E)   } ,
\]
where the inner sum is over all $\F_\mathfrak{p}^\alg$-points of  $\mathcal{Z}_\mathrm{big}(\alpha)$, viewed as an $\co_\kk$-stack.
\end{proposition}

\begin{proof}
Combining Propositions \ref{prop:class number}, \ref{prop:fourier}, and \ref{prop:cycle degree} shows that
\[
\sum_{ \mathfrak{q}_\Phi \subset \co_\Phi }
\frac{ n \cdot \log(\mathrm{N}(\mathfrak{q}_\Phi))  }{\deg_\C (\mathcal{Y}_\mathrm{big}) }
 \sum_{  y \in \mathcal{Z}_\mathrm{big}(\alpha) (\F_{\mathfrak{q}_\Phi}^\alg) }
 \frac{ \mathrm{length} (\co_y)  }{ |\Aut(y) | }
 = -  \frac{ a_F(\alpha,0) }{  \Lambda(0,\chi_E)   } ,
\]
where the inner sum is over all $\F_{\mathfrak{q}_\Phi}^\alg$ points of  $\mathcal{Z}_\mathrm{big}(\alpha)$, viewed as an $\co_\Phi$-stack.
The claim follows by collecting together all primes $\mathfrak{q}_\Phi\subset \co_\Phi$ lying above a common prime $\mathfrak{p}\subset \co_\kk$.
\end{proof}

\begin{proposition}\label{prop:green eval}
The regularized theta lift $\Theta^\mathrm{reg}(f_m )$ satisfies
\begin{multline*}
\frac{n}{ \deg_\C(\mathcal{Y}_\mathrm{big}) } \sum_{ y\in \mathcal{Y}_\mathrm{big}(\C) }   \frac{ \Theta^\mathrm{reg}(f_m )(y) }{ |\Aut(y) | } \\
     =      -  \frac{d}{ds} \langle E(s) , \xi(f_m)  \rangle_\mathrm{Pet} \big|_{s=0}
   +  \sum_{ \substack{ \alpha\in F_+ \\ \mathrm{Tr}_{F/\Q}(\alpha) =m   }}   \frac{a_F(\alpha,0) }{\Lambda(0,\chi_E)}
-2 c_{f_m}^+(0,0) \cdot  \frac{\Lambda'(0,\chi_E)}{ \Lambda(0,\chi_E)} .
\end{multline*}
\end{proposition}

\begin{proof}
This is a special case of the main result of \cite{BKY}.
This requires some explanation, as that work deals with cycles on Shimura varieties of type $\GSpin$, rather than the unitary Shimura varieties under current consideration.

Recall that we have an $F$-quadratic space $(\mathscr{V},\mathscr{Q})$ of rank two, and a $\Q$-quadratic space $(V,Q)$ whose underlying $\Q$-vector space
\[
V = \Hom_\kk(W_0,W)
\]
is equal to $\mathscr{V}$, and whose quadratic form is (\ref{trace form}).
As in \cite[\S 2]{BKY} or \cite[\S 5.3]{AGHMP-2} this data determines a commutative diagram 
\[
\xymatrix{
{ 1 } \ar[r] &  {  \mathbb{G}_m } \ar[r]  \ar@{=}[d]  &  {   T_\GSpin }  \ar[r] \ar[d]  &{  T_\SO } \ar[r] \ar[d]   & {  1  }  \\
{ 1 }  \ar[r] &  {  \mathbb{G}_m } \ar[r]   &  {   \GSpin(V)  }  \ar[r]   &{  \SO(V) } \ar[r]   & {  1  },
}
\]
with exact rows, of algebraic groups over $\Q$.
 The torus $T_\SO=\mathrm{Res}_{F/\Q} \SO(\mathscr{V})$ has $\Q$-points
\[
T_\SO(\Q) = \{ y\in E^\times : y\overline{y} =1  \},
\]
while the torus $T_\GSpin$ has $\Q$-points
\[
T_\GSpin(\Q) = E^\times / \mathrm{ker}( \mathrm{Norm} : F^\times \to \Q^\times ) .
\]
The map $T_\GSpin \to T_\SO$ is $x\mapsto x/\overline{x}$.
To these groups one can associate morphisms of Shimura data
\[
\xymatrix{
 {  ( T_\GSpin , \{ h_\GSpin\} )  }  \ar[r] \ar[d]  &   { ( T_\SO, \{ h_\SO\} )  }     \ar[d]  \\
 {  ( \GSpin(V) , \mathcal{D}_\GSpin )   }  \ar[r]   &  {  ( \SO(V) , \mathcal{D}_\SO ) . }
}
\]
In the top row  both data have reflex field $E_\Phi$.  In the bottom row both data have reflex field $\Q$.

Let  $K_\SO \subset \SO(V)(\A_f)$ be any compact open subgroup that  stabilizes the lattice $L\subset V$, and  fix  any compact open subgroup $K_\GSpin \subset \GSpin(V)(\A_f)$ contained in the preimage of $K_\SO$.
The Shimura data in the bottom row, along with these compact open subgroups, determine Shimura varieties
$
M_\GSpin \to M_\SO.
$
These are $\Q$-stacks  of dimension $2n-2$.

The Shimura data in the top row, along with the  compact open subgroups $K_\GSpin \cap T_\GSpin(\A_f)$ and $K_\SO \cap T_\SO(\A_f)$,  determine  Shimura varieties
$
Y_\GSpin \to Y_\SO.
$
These are $E_\Phi$-stacks of dimension $0$, but we instead view them  as stacks over $\Spec(\Q)$, so that there is a commutative diagram
\begin{equation}\label{gspin descent}
\xymatrix{
{  Y_\GSpin } \ar[r]\ar[d]    & {   Y_\SO  }  \ar[d]  \\
{  M_\GSpin }  \ar[r]  &  {  M_\SO .}
}
\end{equation}

Assume that the compact open subgroup $K_\SO$ acts trivially on the quotient $L'/L$.
For every form $f\in H_{2-n}(\overline{\omega}_L)$, one can find in \cite[Theorem 3.2]{BKY} the construction of a divisor $Z_\GSpin(f)$ on $M_\GSpin$, along with a Green function
$
 \Theta^\reg_{\GSpin}(f)
$
for that divisor, constructed as a regularized theta lift.
Up to change of notation, \cite[Theorem 1.1]{BKY} asserts that
\begin{align}\nonumber
\frac{ n  }{\deg_\C (Y_\GSpin )}
\sum_{ y\in Y_\GSpin(\C) }  \frac{  \Theta^\reg_{\GSpin}(f,y)   }{ |\Aut(y) |}
  & =
  -  \frac{d}{ds} \langle E(s) , \xi(f)  \rangle_\mathrm{Pet} \big|_{s=0}  \\
   & \quad +     \sum_{   \substack{ m \ge 0 \\  \mu \in L'/L  } }  \frac{a(m,\mu) c_f^+(-m,\mu) }{\Lambda(0,\chi_E)}
   \label{BKYmain}
\end{align}
where  the coefficients $a(m) \in S_L$ are defined by
\[
a(m) = \sum_{ \substack{ \alpha\in F_+ \\ \mathrm{Tr}_{F/\Q}(\alpha) =m   }} a_F(\alpha)
\]
if $m>0$, and by $a(0)=a_F(0)$.

It is not difficult to see,  directly from the constructions,  that both the divisor $Z_\GSpin(f)$ and the Green function $\Theta^\reg_\GSpin(f)$ descend to the quotient $M_\SO$.  If we call these descents $Z_\SO(f)$ and $\Theta^\reg_\SO(f)$,  it is a formal consequence
of the commutativity of (\ref{gspin descent}) that the equality (\ref{BKYmain}) continues to hold if all subscripts $\GSpin$ are replaced by $\SO$.

Moreover, suppose that our form $f\in H_{2-n}(\overline{\omega}_L)$ is invariant under the action of the finite group $\Delta$ of \S \ref{ss:first small derivative}, as is true for the form  $f_m$ of (\ref{simple harmonic}).  In this case one can see,  directly from the definitions, that the divisor $Z_\SO(f)$ and the Green function $\Theta^\reg_\SO(f)$
descend to the orthogonal Shimura variety determined by the maximal compact open subgroup
\[
K_\SO= \{ g \in \SO(V)(\A_f) : gL=L \}.
\]
From now on we fix this choice of $K_\SO$.

Specializing (\ref{BKYmain}) to the form $f=f_m$, and using the formula for $a(0)=a_F(0)$ found in Proposition \ref{prop:constant eisenstein}, we obtain
\begin{align}\nonumber
\frac{n}{ \deg_\C(   Y_\SO  ) } \sum_{ y\in Y_\SO(\C) }   \frac{ \Theta_\SO^\mathrm{reg}(f_m )(y) }{ |\Aut(y) | }
   &  =      -  \frac{d}{ds} \langle E(s) , \xi(f_m)  \rangle_\mathrm{Pet} \big|_{s=0}   \\
& \quad     +    \frac{a(m,0) }{\Lambda(0,\chi_E)}   -2 c_{f_m}^+(0,0) \cdot  \frac{\Lambda'(0,\chi_E)}{ \Lambda(0,\chi_E)} . \label{SOBKY}
\end{align}

As in \cite[\S 2.1]{BHKRY}, our group $G \subset \GU(W_0) \times \GU(W)$ acts on $V$ in a natural way, defining a homomorphism $G\to \SO(V)$.  On the other hand, Remark \ref{rem:split group} shows  that $T_\mathrm{big} \iso \mathrm{Res}_{\kk/\Q} \mathbb{G}_m  \times T_\SO$, and projection to the second factor defines a morphism $T_\mathrm{big} \to T_\SO$.   We obtain morphisms of Shimura data
\[
\xymatrix{
 {  ( T_\mathrm{big} , \{ h_\mathrm{big}\} )  }  \ar[r] \ar[d]  &   { ( T_\SO, \{ h_\SO\} )  }     \ar[d]  \\
 {  (  G , \mathcal{D} )   }  \ar[r]   &  {  ( \SO(V) , \mathcal{D}_\SO ) , }
}
\]
which induce morphisms of $\kk$-stacks
\[
\xymatrix{
{  \mathcal{Y}_{ \mathrm{big} / \kk} } \ar[r]\ar[d]    & {   Y_{\SO/\kk}   }  \ar[d]  \\
{  \mathcal{S}_{\Kra/\kk}  }  \ar[r]  &  {  M_{\SO/\kk}  .}
}
\]

The Green function $\Theta^\reg(f_m)$ on $\mathcal{S}_{\Kra/\kk}$ defined in \cite[\S 7.2]{BHKRY}  is simply  the pullback of the Green function $\Theta^\reg_\SO(f_m)$ via the bottom horizontal arrow.   It follows easily that
\[
\frac{n}{ \deg_\C(   Y_\SO  ) } \sum_{ y\in Y_\SO(\C) }   \frac{ \Theta_\SO^\mathrm{reg}(f_m )(y) }{ |\Aut(y) | }
=
\frac{n}{ \deg_\C(\mathcal{Y}_\mathrm{big}) } \sum_{ y\in \mathcal{Y}_\mathrm{big}(\C) }   \frac{ \Theta^\mathrm{reg}(f_m )(y) }{ |\Aut(y) | },
\]
and comparison with (\ref{SOBKY})  completes the proof of Proposition \ref{prop:green eval}.
\end{proof}

\begin{proof}[Proof of Theorem \ref{thm:big harmonic lift}]
Combining the decomposition (\ref{stack decomp}) with Proposition \ref{prop:arakelov degree} shows that
\[
\sum_{ \mathfrak{p} \subset \co_\kk }
\frac{n \log(\mathrm{N}(\mathfrak{p} ) )  } { \deg_\C(\mathcal{Y}_\mathrm{big}) }
 \sum_{ y \in ( \mathcal{Z}^\tot_\Kra(m) \cap   \mathcal{Y}_\mathrm{big}  ) ( \F_\mathfrak{p}^\alg ) }
\frac{ \mathrm{length}(\co_y)  }{| \Aut(y) | }
=    \sum_{ \substack{ \alpha\in F_+ \\ \mathrm{Tr}_{F/\Q}(\alpha) =m   }} \frac{- a_F(\alpha,0)}{\Lambda(0,\chi_E) }.
\]
Plugging this formula and   the archimedean calculation of Proposition \ref{prop:green eval}  into (\ref{intersection decomp}) leaves
\[
 \frac{ n  \cdot   [  \widehat{\mathcal{Z}} ( f_m )  :  \mathcal{Y}_\mathrm{big}  ]   }{ \deg_\C(\mathcal{Y}_\mathrm{big}) }
 =
 - 2 c_{f_m}^+(0,0) \cdot  \frac{\Lambda'(0,\chi_E)}{ \Lambda(0,\chi_E)}
  -  \frac{d}{ds} \langle E(s) , \xi(f_m)  \rangle_\mathrm{Pet} \big|_{s=0} ,
 \]
 as desired.
\end{proof}


\subsection{The proof of Theorem \ref{Theorem B}}
\label{ss:big derivative}


We now use Theorem \ref{thm:big harmonic lift} to  prove  a special case of Theorem \ref{Theorem D}, and then prove Theorem \ref{Theorem B}.  We assume  $n\ge 3$.

Recall the differential operator 
\[
\xi :H_{2-n}(\omega_L) \to S_n( \overline{\omega}_L) 
\]
 of \S \ref{ss:first small derivative}.  Its kernel  is the subspace
\[
M^!_{2-n}(\omega_L) \subset H_{2-n}(\omega_L)
\]
of weakly holomorphic forms.

\begin{lemma}\label{lem:awesome f}
In the notation of \S \ref{ss:first small derivative}, there exists a $\Delta$-invariant form $f\in M^!_{2-n}(\omega_L)$ such that $c_f^+(0,0)\neq 0$, and
\[
\widehat{\mathcal{Z}} (f) +  c_f^+(0,0) \cdot \widehat{\mathcal{Z}}_\Kra^\tot(0) =0  .
\]
 \end{lemma}

\begin{proof}
Denote by
\[
S_{2-n}^{!,\infty}( \Gamma_0(D) ,\chi^n_\kk ) \subset M_{2-n}^!( \Gamma_0(D) ,\chi^n_\kk )
\]
the subspace of  forms that vanish at all  cusps other than $\infty$, and choose  any form
\[
f_0 (\tau) = \sum_{ \substack{ m\in \Z \\ m \gg -\infty} } c_0(m)\cdot q^m \in S_{2-n}^{!,\infty}( \Gamma_0(D) ,\chi^n_\kk )
\]
such that  $c_0(0) \neq 0$. The existence of such a form can be proved as in \cite[Lemma 4.11]{BBK}.
As in (\ref{vector induction}) there is an induced form
\[
f(\tau) = \sum_{\gamma\in \Gamma_0(D) \backslash \SL_2(\Z) } ( f_0 |_{2-n} \gamma)(\tau) \cdot   \omega_L(\gamma^{-1})\phi_0
\in M_{2-n}^!(\omega_L)^\Delta ,
\]
which we claim  has the desired properties.

Indeed, the proof of Proposition  \ref{prop:good harmonic choice} shows that $c_f^+(0,0) = c_0(0)$, and that
$f = \sum_{m>0} c_0(-m) f_m$.  In particular,
\[
\widehat{\mathcal{Z}}  (f) = \sum_{m>0} c_0(-m) \cdot  \widehat{\mathcal{Z}}_\Kra^\tot(m) \in
\widehat{\mathrm{Ch}}_\C^1( \mathcal{S}_\Kra^* ).
\]

Given any modular form
\[
g(\tau) = \sum_{m\ge 0} d(m) \cdot q^m \in M_n(D,\chi^n_\kk ),
\]
summing the residues of the meromorphic form $f_0(\tau) g(\tau) d\tau$ on $X_0(D)(\C)$ shows that
\[
 \sum_{m\ge 0} c_0(-m) \cdot  d(m) = 0.
\]
Thus  the modularity of the generating series (\ref{modular divisors}) implies the second equality in
\[
\widehat{\mathcal{Z}} (f) +  c_0(0) \cdot \widehat{\mathcal{Z}}_\Kra^\tot(0)
=
 \sum_{m\ge 0} c_0(-m) \cdot  \widehat{\mathcal{Z}}_\Kra^\tot(m) = 0.
\]
\end{proof}

We can now prove Theorem \ref{Theorem D} under some additional hypotheses.  These hypotheses will be removed in \S \ref{s:colmez}.

\begin{theorem}\label{thm:big cm volume}
If the discriminants of $\kk/\Q$ and $F/\Q$ are odd and relatively prime, then
\[
[ \widehat{\bm{\omega}} : \mathcal{Y}_\mathrm{big} ]
=  \frac{-2}{n} \cdot  \deg_\C(\mathcal{Y}_\mathrm{big})  \cdot \frac{  \Lambda'(0,\chi_E)}{\Lambda(0,\chi_E)}.
\]
\end{theorem}

\begin{proof}
If we choose $f$ as in Lemma \ref{lem:awesome f} then $\xi(f)=0$, and so Theorem \ref{thm:big harmonic lift} simplifies to
\[
- n  c_f^+(0,0) \cdot  \frac{    [  \widehat{\mathcal{Z}}^\tot_\Kra(0)  :  \mathcal{Y}_\mathrm{big}  ]  }  {  \deg_\C( \mathcal{Y}_\mathrm{big} )    }
  + 2  c_f^+(0,0)  \cdot   \frac{  \Lambda'(0,\chi_E )   }{    \Lambda(0,\chi_E )   }
 =  0.
\]
An application of Proposition \ref{prop:big no error} completes the proof.
\end{proof}

The following is Theorem \ref{Theorem B} in the introduction.

\begin{theorem}
Assume that the discriminants of $\kk/\Q$ and $F/\Q$ are odd and relatively prime,
and let $g \in S_n(\Gamma_0(D) , \chi^n )$ and  $ \tilde{g} \in S_n(\overline{\omega}_L)$ be related by (\ref{vector induction}).
The central derivative of the Petersson inner product (\ref{fake L})  is   related to  the arithmetic theta lift (\ref{ATL}) by
\[
[ \widehat{\theta}(g) : \mathcal{Y}_\mathrm{big} ]
= \frac{-1}{n} \cdot  \deg_\C (\mathcal{Y}_\mathrm{big})  \cdot   \frac{d}{ds} \langle E(s) , \tilde{g} \rangle_\mathrm{Pet} \big|_{s=0} .
\]
\end{theorem}

\begin{proof}
If we choose $f$ as in Proposition \ref{prop:good harmonic choice}, then $\xi(f)=\tilde{g}$ and
\[
[ \widehat{\theta}(g) : \mathcal{Y}_\mathrm{big}]
= [  \widehat{\mathcal{Z}} ( f )  : \mathcal{Y}_\mathrm{big}] + c_f^+(0,0) \cdot   [ \widehat{\mathcal{Z}}_\Kra^\tot(0): \mathcal{Y}_\mathrm{big}].
\]
Proposition \ref{prop:big no error} and Theorem \ref{thm:big cm volume} allow  us to rewrite this as
\begin{align*}
[ \widehat{\theta}(g) : \mathcal{Y}_\mathrm{big}]
&  = [  \widehat{\mathcal{Z}} ( f )  : \mathcal{Y}_\mathrm{big}] - c_f^+(0,0) \cdot   [ \widehat{\bm{\omega}}: \mathcal{Y}_\mathrm{big}] \\
&=
 [  \widehat{\mathcal{Z}} ( f )  : \mathcal{Y}_\mathrm{big}] +  \frac{2}{n}  \cdot c_f^+(0,0)  \cdot  \deg_\C(\mathcal{Y}_\mathrm{big})  \cdot \frac{  \Lambda'(0,\chi_E)}{\Lambda(0,\chi_E)},
\end{align*}
and comparison with Theorem \ref{thm:big harmonic lift} completes the proof.
\end{proof}


\section{Faltings heights of CM abelian varieties}
\label{s:colmez}


In  \S \ref{s:colmez} we assume  $n\ge 2$, and  study Theorems \ref{Theorem C}  and \ref{Theorem D}  of the introduction.
As in \S \ref{ss:intro big GZ}, let  $F$ be a totally real field of degree $n$, set
\[
E=\kk\otimes_\Q F,
\]
and let $\Phi \subset \Hom(E,\C)$ be a CM type  of signature $(n-1,1)$.   We fix a triple $(\mathfrak{a}_0,\mathfrak{a},i_E)$ as in \S \ref{ss:big cm}.


\subsection{Some metrized line bundles}


By virtue of the inclusion (\ref{kramer inclusion}), there is a universal pair $(A_0,A)$ over $\mathcal{S}_\Kra$ consisting of an elliptic curve $\pi_0: A_0\to \mathcal{S}_\Kra$ and an abelian scheme $\pi: A \to \mathcal{S}_\Kra$ of dimension $n$.

Endowing the Lie algebras of $A_0$ and $A$ with their Faltings (a.k.a.~Hodge) metrics gives rise to metrized line bundles
\[
\Lie(A_0)   \in \widehat{\mathrm{Pic}}( \mathcal{S}_\Kra ),\qquad
 \det(\Lie(A) )  \in \widehat{\mathrm{Pic}}( \mathcal{S}_\Kra ).
\]
A vector  $\eta$  in the fiber
\[
\det(\Lie(A_s))^{-1} \iso  \bigwedge\nolimits^n \mathrm{Fil}^1 H_\dR^1(A_s) \subset \bigwedge\nolimits^n \ H_\dR^1(A_s)
\]
at a complex point $s\in\mathcal{S}_\Kra(\C)$ has norm
\begin{equation}\label{integration}
\| \eta  \|_s^2 = \Big| \int_{A_s(\C)}  \eta \wedge \overline{ \eta }\ \Big|.
\end{equation}
The metric on $\Lie(A_0)$ is defined similarly.

 We now recall some notation from \cite[\S 1.8]{BHKRY}.  Fix a $\pi \in \co_\kk$ such that $\co_\kk = \Z + \Z \pi$.
 If $S$ is any $\co_\kk$-scheme, define
 \begin{align}
\epsilon_S   & = \pi \otimes 1 - 1 \otimes i_S(\overline{\pi}) \in  \co_\kk \otimes_\Z \co_S \label{idempotents}\\
\overline{\epsilon}_S   & =  \overline{\pi} \otimes 1 - 1\otimes i_S( \overline{\pi} ) \in  \co_\kk \otimes_\Z \co_S \nonumber ,
\end{align}
where $i_S : \co_\kk \to \co_S$ is the structure map.  We usually just write $\epsilon$ and $\overline{\epsilon}$, when the scheme $S$ is clear from context.

\begin{remark}\label{rem:eigensplitting}
  If $N$ is an $\co_\kk \otimes_\Z \co_S$-module then $N/ \overline{\epsilon}  N$ is the maximal quotient of $N$ on which $\co_\kk$ acts through the structure morphism $i_S : \co_\kk \to \co_S$, and  $N/\epsilon N$ is the maximal quotient on which $\co_\kk$ acts through the conjugate of the structure morphism.  If $D \in \co_S^\times$ then
\[
N = \epsilon  N \oplus \overline{\epsilon} N,
\]
and the summands are the maximal submodules on which $\co_\kk$ acts through the structure morphism and its conjugate, respectively.
\end{remark}

As in \cite[\S 2.2]{BHKRY},  the relative de Rham homology $H_1^{\dR}(A)$ is a rank $2n$ vector bundle on $\mathcal{S}_\Kra$ endowed with an action of $\co_\kk$ induced from that on $A$.  In fact, it is locally free of rank $n$ as an $\co_\kk\otimes_\Z \co_{\mathcal{S}_\Kra}$-module, and
\[
\mathcal{V} = H_1^{\dR}(A) / \overline{\epsilon} H_1^{\dR}(A)
\]
is a rank $n$ vector bundle.  We make $\det(\mathcal{V})$ into a metrized line bundle by declaring that a local section $\eta$ of its inverse
\[
\det(\mathcal{V})^{-1} \iso \bigwedge\nolimits^n \epsilon H^1_{\dR}(A) \subset H_{\dR}^n(A)
\]
has norm  (\ref{integration}) at a complex point $s\in \mathcal{S}_\Kra(\C)$.

As the exceptional divisor $\mathrm{Exc} \subset \mathcal{S}_\Kra$ of \cite[\S 2.3]{BHKRY} is supported in characteristics dividing $D$, the line bundle $\co(\mathrm{Exc})$ is canonically trivial in the generic fiber.  We endow it with the trivial metric.  That is to say, the constant function $1$, viewed as a section of $\co(\mathrm{Exc})$, has norm $\| 1 \|^2 = 1$.

Recall that the line bundle $\bm{\omega}$ of \cite[\S 2.4]{BHKRY} was endowed with a metric in \cite[\S 7.2]{BHKRY}, defining
\[
\widehat{\bm{\omega}} \in \widehat{\mathrm{Pic}}( \mathcal{S}_\Kra ).
\]
For any positive real number $c$, denote by
\[
\co\langle c \rangle \in \widehat{\mathrm{Pic}}( \mathcal{S}_\Kra )
\]
the trivial bundle $\co_{\mathcal{S}_\Kra}$ endowed with the constant metric $\| 1 \|^2 =c$.

\begin{proposition}\label{prop:bundle shuffle}
There is an isomorphism
\[
\co \langle 8 \pi^2 e^\gamma D^{-1}\rangle ^{\otimes 2}  \otimes     \widehat{\bm{\omega}}^{\otimes 2}    \otimes \det(\Lie(A))  \otimes  \Lie(A_0)^{\otimes 2} \iso   \co(\mathrm{Exc})  \otimes \det(\mathcal{V})
\]
of metrized line bundles on $\mathcal{S}_\Kra$.
\end{proposition}

\begin{proof}
In \cite[\S 2.4]{BHKRY} we defined a line bundle $\bm{\Omega}_\Kra$ on $\mathcal{S}_\Kra$ by
\[
\bm{\Omega}_\Kra = \mathrm{det}(\Lie(A))^{-1}  \otimes \Lie(A_0)^{\otimes -2}  \otimes   \mathrm{det}(\mathcal{V}) ,
\]
and in  \cite[Theorem 2.6.3]{BHKRY} we constructed an isomorphism
\[
\bm{\omega}^{\otimes 2} \iso \bm{\Omega}_\Kra \otimes \co(\mathrm{Exc}) .
\]
This defines the desired isomorphism
\begin{equation}\label{integral lines}
\bm{\omega}^{\otimes 2}    \otimes \det(\Lie(A))  \otimes  \Lie(A_0)^{\otimes 2} \iso   \co(\mathrm{Exc})  \otimes \det(\mathcal{V})
\end{equation}
on underlying line bundles, and it remains to compare the metrics.

In the complex fiber this can be made more explicit.
At any complex point $s\in \mathcal{S}_\Kra(\C)$ the Hodge short exact sequence admits a canonical splitting
\[
H_1^{\dR}(A_s) = F^0(A_s) \oplus \Lie(A_s),
\]
where $F^0(A_s) = \Fil^0 H^\dR_1(A_s)$ is the nontrivial step in the Hodge filtration.
When combined with the decomposition of Remark \ref{rem:eigensplitting} we obtain
\[
H_1^{\dR}(A_s)  =
 \underbrace{\epsilon  F^0(A_s)}_{1}\oplus   \underbrace{ \overline{\epsilon} F^0(A_s) }_{n-1}\oplus
\underbrace{ \epsilon  \Lie(A_s) }_{n-1}  \oplus  \underbrace{\overline{\epsilon} \Lie(A_s)}_{1}
\]
where the subscripts indicate the dimensions as $\C$-vector spaces.    There is a similar decomposition
\[
H_1^{\dR}(A_{0s})  =
 \underbrace{\epsilon  F^0(A_{0s})}_{0}\oplus   \underbrace{ \overline{\epsilon} F^0(A_{0s}) }_{1}\oplus
\underbrace{ \epsilon  \Lie(A_{0s}) }_{1}  \oplus  \underbrace{\overline{\epsilon} \Lie(A_{0s})}_{0}
\]

Denote by
\begin{equation}\label{deRham pairing}
\psi  : H_1^\dR(A_s) \times  H_1^\dR(A_s) \to \C
\end{equation}
 the alternating pairing  determined by the principal polarization on $A_s$.
The two direct summands
\[
\epsilon F^0(A_s) \oplus \overline{\epsilon } \Lie(A_s) \subset H_1^\dR(A_s)
\]
are interchanged by complex conjugation.
We endow  both $\epsilon F^0(A_s)$ and $\overline{\epsilon } \Lie(A_s)$   with the metric
\begin{equation}\label{polarization metric}
\|  b  \|_s^2 =    \left| \frac{ \psi ( b , \overline{b}  ) } { 2\pi i }  \right| ,
\end{equation}
 so that the pairing
 \begin{equation}\label{line duality}
 \psi : \epsilon F^0(A_s) \otimes \overline{\epsilon } \Lie(A_s) \to \co \langle 4\pi^2\rangle^{-1}_s
  \end{equation}
    is  an isometry.

For $a,b\in \overline{\epsilon} \Lie(A_s)$, define
$
p_{a\otimes b} : \epsilon F^0(A_s)  \to  \overline{\epsilon} \Lie(A_s)
$
by
\begin{equation}\label{little p}
p_{a\otimes b}( e ) = \psi (\overline{\epsilon} a , e ) \cdot \overline{\epsilon} b = - D \psi (a,e) \cdot b.
\end{equation}
The factor of $-D$ comes from the observation that $\overline{\epsilon}$ acts on $\overline{\epsilon}\Lie(A_s)$ as $\pm \sqrt{-D}$, where the sign depends on the choice of $\pi$ used in (\ref{idempotents}).

We now define $P_{a\otimes b}$ by the commutativity of
\begin{equation}\label{diagram switch}
\xymatrix{
{ \det( \mathcal{V}_s )      }  \ar[rr]^{P_{a\otimes b} } \ar[d]_\iso &  & {  \det( \Lie(A_s) )  }\ar[d]^\iso \\
{   \epsilon F^0(A_s) \otimes \det(\epsilon \Lie(A_s) )  }   \ar[rr]_{ p_{a\otimes b} \otimes\mathrm{id} } & &  {  \overline{\epsilon} \Lie(A_s) \otimes   \det(\epsilon \Lie(A_s) )  .   }
}
\end{equation}
This defines the isomorphism
\begin{equation}\label{P iso}
  ( \overline{\epsilon} \Lie(A_s) )^{\otimes 2}
  \map{P}  \Hom \big( \mathrm{det}(\mathcal{V}_s) , \mathrm{det}(\Lie(A_s))  \big)
\end{equation}
of \cite[Lemma 2.4.5]{BHKRY}.

\begin{lemma}
The isomorphism (\ref{P iso}) defines  an isometry
  \[
    \mathrm{det}(\mathcal{V}_s)
\iso   \co\langle 2\pi D^{-1} \rangle_s^{\otimes 2} \otimes    ( \epsilon F^0(A_s) )^{\otimes 2}  \otimes  \mathrm{det}(\Lie(A_s)).
\]
\end{lemma}

\begin{proof}
Fix an isomorphism $\bigwedge\nolimits^{2n} H_1(A_s(\C) , \Z) \iso \Z$ and extend it to a $\C$-linear isomorphism
\[
\mathrm{vol} : \bigwedge\nolimits^{2n} H^\dR_1(A_s) \iso \C.
\]

Under the de Rham comparison isomorphism
$
H_1(A_s(\C) , \C)  \iso H_1^\dR(A_s),
$
the pairing (\ref{deRham pairing}) restricts to a perfect pairing
\[
\psi : H_1(A_s(\C) , \Z) \times H_1(A_s(\C) , \Z) \to 2\pi i \Z.
\]
It follows that there is a unique element
$
\Psi = \alpha \wedge \beta \in \bigwedge\nolimits^2 H_1(A_s(\C) , \Z)
$
such that
\[
2\pi i \cdot  \psi(a,b) = \psi(\alpha,a)\psi(\beta,b) - \psi(\alpha,b)\psi(\beta,a)
\]
for all $a,b\in H_1(A_s(\C) , \Z)$.  The  map
\[
 \Big( \bigwedge\nolimits^{n-1} H_1(A_s(\C) , \Z) \Big)  \otimes   \Big(  \bigwedge\nolimits^{n-1} H_1(A_s(\C) , \Z)  \Big)
\to \Z
\]
defined by $a\otimes b \mapsto \mathrm{vol}( \Psi \wedge a\wedge b)$ is a perfect pairing of $\Z$-modules.

We now metrize the line
\[
\det(\epsilon \Lie(A_s)) \subset  \bigwedge\nolimits^{n-1} \epsilon H^\dR_1(A_s)
\]
by $\| \mu \|^2 =|  \mathrm{vol}( \Psi\wedge \mu \wedge \overline{\mu} ) | $.
With this definition, the vertical arrows  in (\ref{diagram switch}) are isometries.

Using (\ref{line duality}) and  (\ref{little p}), one sees that the map
 \[
 p_{a\otimes b} \in \Hom( F^0(A_s) , \overline{\epsilon } \Lie(A_s))
 \]
 satisfies $\| p_{a\otimes b} \| = 2\pi D\cdot  \|a \otimes b\| $, and hence also  $\| P_{a\otimes b} \| = 2\pi D\cdot  \|a \otimes b\| $.
 This proves that the isomorphism $P$ defines an isometry
 \[
\co\langle 2\pi D\rangle _s^{\otimes 2} \otimes   ( \overline{\epsilon} \Lie(A_s) )^{\otimes 2}
 \iso \Hom \big( \mathrm{det}(\mathcal{V}_s) , \mathrm{det}(\Lie(A_s))  \big)  .
 \]
The isomorphism (\ref{line duality}) allows us to rewrite this as
\[
  \mathrm{det}(\mathcal{V}_s)
 \iso \co \langle 2\pi D^{-1} \rangle_s^{\otimes 2} \otimes   (  \epsilon F^0(A_s) )^{\otimes 2}  \otimes  \mathrm{det}(\Lie(A_s))  .
\]
 \end{proof}

The  proof of \cite[Proposition 2.4.2]{BHKRY} gives an isomorphism
\begin{equation}\label{tautological}
\bm{\omega}_s \iso \Hom( \Lie(A_{0s})  ,  \epsilon  F^0(A_s) ) \subset \epsilon  V_\C
\end{equation}
where
\[
V =\Hom_\kk \big(H_1 (A_{0s}(\C)  ,\Q ) , H_1 ( A_s(\C) ,\Q)\big).
\]
As in \cite[\S 2.1]{BHKRY},  there is a $\Q$-bilinear form $[\cdot\, ,\cdot] : V\times V\to \Q$ induced by the polarizations on $A_{0s}$ and $A_s$.
If we  extend this  to a $\C$-bilinear form on
\[
V_\C = \Hom_{\kk\otimes \C}  \big(H^\dR_1 (A_{0s} ) , H^\dR_1 ( A_s)  \big)
\]
then the  metric on $\bm{\omega}_s$ is  defined, as in  \cite[\S 7.2]{BHKRY}, by
\[
 \| x \|^2 =   \frac{  | [x,\overline{x}] | } { 4\pi e^\gamma }
\]
for any $x\in  \Hom( \Lie(A_{0s})  ,  \epsilon  F^0(A_s) )$.

On the other hand, we have defined the Faltings metric on $\Lie(A_{0s})$, and defined a metric  on $\epsilon F^0(A_s)$ by (\ref{polarization metric}).  The following lemma shows that (\ref{tautological}) respects the metrics, up to scaling by a factor of $4\pi e^\gamma$.

\begin{lemma}
The isomorphism (\ref{tautological}) defines an isometry
\[
\co \langle 4 \pi e^\gamma\rangle_s \otimes \widehat{\bm{\omega}}_s \iso \Hom( \Lie(A_{0s})  ,  \epsilon  F^0(A_s) ) .
\]
\end{lemma}

\begin{proof}
The alternating form
\[
\psi_0 : H_1^\dR(A_{0s}) \times H_1^\dR(A_{0s}) \to \C
\]
analogous to (\ref{deRham pairing}) restricts to a perfect pairing
\[
\psi_0 : H_1 (A_{0s}(\C),\Z ) \times H_1 (A_{0s}(\C),\Z) \to 2\pi i \Z,
\]
and hence the Faltings metric on $\Lie(A_{0s}) = \epsilon H_1^\dR(A_{0s})$ is
\[
\| a  \|^2 = (2\pi)^{-1} | \psi_0( a, \overline{a} )| .
\]

From the definition of the bilinear form on $V$, one can show that
\[
[ x, \overline{x}] \cdot   \psi_0(a, \overline{a} )  =   \psi(xa,  \overline{xa}   )
\]
for all $x\in \epsilon V_\C$.  Comparing with the metric on $\epsilon F^0(A_s)$ shows that
\[
4 \pi  e^\gamma  \cdot \| x\|^2 \cdot  \| a \|^2 =   (2\pi )^{-1} \cdot  | \psi(xa,  \overline{xa}   ) |  =   \| xa \|^2,
\]
for all $x\in \bm{\omega}_s$ and $a\in \Lie(A_{0s})$, as claimed.
\end{proof}

The two lemmas provide us with isometries
\begin{align*}
 \mathrm{det}(\mathcal{V}_s)
 & \iso \co \langle 2 \pi D^{-1} \rangle_s^{\otimes 2} \otimes    ( \epsilon F^0(A_s) )^{\otimes 2}  \otimes  \mathrm{det}(\Lie(A_s))  \\
&\iso \co \langle 8 \pi^2 e^\gamma D^{-1} \rangle_s^{\otimes 2} \otimes
\widehat{\bm{\omega}}_s^{\otimes 2} \otimes   \Lie(A_{0s})^{\otimes 2} \det(\Lie(A_s))
 \end{align*}
 and the composition agrees with the isomorphism (\ref{integral lines}).  This completes the proof of Proposition \ref{prop:bundle shuffle}.
 \end{proof}

Recall  the  big CM cycle  $\pi : \mathcal{Y}_\mathrm{big} \to \mathcal{S}_\Kra^*$  of Definition \ref{def:big cycle}.  All of the  metrized line bundles on $\mathcal{S}_\Kra$ appearing in Proposition \ref{prop:bundle shuffle} can be extended to the toroidal compactification $\mathcal{S}_\Kra^*$ (with possible $\log$-singularities along the boundary) so as to define classes in the codimension one arithmetic Chow group.  However, we don't actually need this. Indeed, we can define a homomorphism
\[
[-:\mathcal{Y}_{\mathrm{big}} ] : \widehat{\mathrm{Pic}}( \mathcal{S}_\Kra ) \to \R
\]
as the composition
\[
\widehat{\mathrm{Pic}}( \mathcal{S}_\Kra ) \map{\pi^*} \widehat{\mathrm{Pic}}( \mathcal{Y}_\mathrm{big})
\iso \widehat{\mathrm{Ch}}^1( \mathcal{Y}_\mathrm{big}) \map{\widehat{\deg}} \R.
\]
As the big CM cycle does not meet the boundary of the toroidal compactification, the composition
\[
 \widehat{\mathrm{Ch}}^1( \mathcal{S}_\Kra^*) \iso  \widehat{\mathrm{Pic}}( \mathcal{S}_\Kra^* )
 \to \widehat{\mathrm{Pic}}( \mathcal{S}_\Kra)  \map{[-:\mathcal{Y}_{\mathrm{big}} ] } \R
\]
 agrees with the arithmetic degree along $\mathcal{Y}_\mathrm{big}$ of Definition \ref{def:big cycle}.

\begin{remark}\label{trivial archimedean}
Directly from the definitions, and recalling Remark \ref{rem:no one-half}, the metrized line bundle $ \co\langle c\rangle$  satisfies
\[
[ \co\langle c\rangle  : \mathcal{Y}_\mathrm{big} ] =  \sum_{ y\in \mathcal{Y}_\mathrm{big}(\C) } - \log\| 1\|^2 = - \log(c) \cdot \deg_\C(\mathcal{Y}_\mathrm{big}).
\]
\end{remark}


\subsection{The Faltings height}


Recall from \S \ref{ss:big cm} the moduli stack $\mathcal{CM}_\Phi$ of abelian varieties over $\co_\Phi$-schemes with complex multiplication by   $\co_E$ and  CM type $\Phi$.

Suppose $A\in \mathcal{CM}_\Phi(\C)$.
Choose a model of $A$ over a number field $L \subset \C$ large enough that  the N\'eron model $\pi:\mathcal{A} \to \Spec(\co_L)$ has everywhere good reduction.
Pick a nonzero rational section $s$ of the line bundle $\pi_*\Omega^{\mathrm{dim}(A)}_{ \mathcal{A}/\co_L }$ on $\Spec(\co_L)$,  and define
\[
h^\Falt _\infty ( A, s ) = \frac{-1}{ 2 [ L:\Q] }  \sum_{ \sigma : L \to \C}
\log \big|   \int_{ \mathcal{A}^\sigma(\C) } s^\sigma \wedge \overline{s^\sigma}\,  \big|,
\]
and
\[
h^\Falt_f(A,s) = \frac{1}{  [L:\Q] }  \sum_{ \mathfrak{p} \subset \co_L}  \ord_\mathfrak{p}(s) \cdot  \log  \mathrm{N}(\mathfrak{p})   .
\]
By a result   of Colmez  \cite{Colmez}, the \emph{Faltings height}
\[
h^\Falt_{(E,\Phi)} = h^\Falt_f(A, s) + h^\Falt _\infty ( A,s)
\]
depends only on the pair $(E,\Phi)$.

\begin{proposition}\label{prop:faltings bundles}
The arithmetic degree of $\Lie(A)$ along $\mathcal{Y}_\mathrm{big}$ satisfies
\[
[  \det(\Lie(A)): \mathcal{Y}_\mathrm{big} ]    =   - 2  \deg_\C( \mathcal{Y}_\mathrm{big} )  \cdot    h^\Falt_{(E,\Phi)} .
\]
Similarly, recalling the Faltings height $h^\Falt_\kk$ of (\ref{chowla-selberg}),
\[
[  \Lie(A_0): \mathcal{Y}_\mathrm{big} ]    =  - 2  \deg_\C( \mathcal{Y}_\mathrm{big} )  \cdot  h^\Falt_\kk   .
\]
\end{proposition}

\begin{proof}
Suppose  we are given a morphism  $y : \Spec(\co_L) \to  \mathcal{Y}_\mathrm{big}$ for some finite extension  $L/E_\Phi$.
The restriction of $A$ to $\co_L$ has complex multiplication by $\co_E$ and CM type $\Phi$, and comparing the definition of the Faltings height with the definition of $\widehat{\deg}$ found in \cite[\S 3.1]{Ho2},  shows that the composition
\[
\widehat{\mathrm{Pic}}( \mathcal{S}_\Kra ) \map{\pi^*}  \widehat{\mathrm{Ch}}^1( \mathcal{Y}_\mathrm{big}) \map{y^*}  \widehat{\mathrm{Ch}}^1( \Spec(\co_L) )
\map{ \widehat{\deg} } \R
\]
sends $\Lie(A)^{-1}$ to $[L:\Q] \cdot h^\Falt_{(E,\Phi)}$.

We may choose $L$ in such a way that the $\co_\kk$-stack
\[
\mathcal{Y}_\mathrm{big} \times_{\Spec(\co_\Phi)} \Spec(\co_L)
\] 
admits a finite \'etale cover  by a disjoint union  $Y_\mathrm{big} = \bigsqcup \Spec(\co_L)$ of, say, $m$  copies of $\Spec(\co_L)$, and then
\[
 \frac{ [ \Lie(A) : \mathcal{Y}_\mathrm{big} ] }{ \deg_\C(\mathcal{Y}_\mathrm{big})  } =
  \frac{ [ \Lie(A) : Y_\mathrm{big} ] }{ \deg_\C(Y_\mathrm{big} )}
  = - \frac{ m [L:\Q]  \cdot  h^\Falt_{(E,\Phi)}  }  { m  [ L : \kk]  } = -2\cdot h^\Falt_{(E,\Phi)} .
\]
This proves the first equality, and the proof of the second  is similar.
\end{proof}


\subsection{Gross's trick}
\label{ss:gross}


The goal of \S \ref{ss:gross} is to compute the degree of the metrized line bundle $\det(\mathcal{V})$ along the big CM cycle.
The impatient reader may skip directly to Proposition \ref{prop:junk bundle two} for the answer.
However, the strategy of the calculation is simple enough that we can explain it in a few sentences.

It is an observation of Gross \cite{Gross} that the metrized line bundle $\det(\mathcal{V})$ behaves, for all practical purposes, like the trivial bundle $\co_{\mathcal{S}_\Kra}$ endowed with the constant metric $\| 1 \|^2 = \exp(-c)$ for a certain period $c$.
This is made more precise in Theorem \ref{thm:gross} and Corollary \ref{cor:almost constant} below.
A priori, the constant $c$ is something mysterious, but one can evaluate it by computing the degree of $\det(\mathcal{V})$ along \emph{any} codimension $n-1$ cycle that one chooses.  We choose a cycle  along which  the universal abelian scheme $A \to \mathcal{S}_\Kra$ degenerates to a product of CM elliptic curves.
Using this, one can express the value of $c$ in terms of the Faltings height $h_\kk^\Falt$ appearing in  (\ref{chowla-selberg}).
The degree of $\det(\mathcal{V})$ along $\mathcal{Y}_\mathrm{big}$ is readily computed from this.

To carry out this procedure,  the first step is to construct a cover of $\mathcal{S}_\Kra(\C)$ over which the line bundle $\det(\mathcal{V})$ can be trivialized analytically.
Fix a positive integer $m$, let  $K(m) \subset K$ be the compact open subgroup of \cite[Remark 2.2.3]{BHKRY}, and consider the finite \'etale cover
\[
\xymatrix{
{  \mathrm{Sh}_{K(m)} ( G,\mathcal{D})  (\C)   }   \ar@{=}[rr]  \ar[d]  &  &   {  G(\Q) \backslash \mathcal{D} \times G(\A_f) / K(m) }  \ar[d]     \\
 {   \mathrm{Sh}( G,\mathcal{D})  (\C)   }   \ar@{=}[rr]  &   & {  G(\Q) \backslash \mathcal{D} \times G(\A_f) / K }  .
}
\]
This cover has a moduli interpretation, exactly as with $\mathcal{S}_\Kra$ itself, but with additional level $m$ structure.  This allows us to construct a regular integral model
$\mathcal{S}_\Kra(m)$ over $\co_\kk[1/m]$ of $ \mathrm{Sh}_{K(m)} ( G,\mathcal{D})$, along with a  finite \'etale morphism
\[
\mathcal{S}_\Kra(m) \to \mathcal{S}_{\Kra / \co_\kk [1/m] } .
\]
We use the notation  $\det(\mathcal{V})$ for both the metrized line bundle on $\mathcal{S}_\Kra$, and for its pullback to $\mathcal{S}_\Kra(m)$.

The following results extends a theorem of Gross \cite[Theorem 1]{Gross} to integral models.

\begin{theorem}\label{thm:gross}
Suppose $m \ge 3$,    let $\Z^\alg \subset \C$ be the subring of all algebraic integers, and fix a connected component
\[
\mathcal{C} \subset  \mathcal{S}_\Kra(m)_{/\Z^\alg[1/m] }.
\]
The line bundle $\det(\mathcal{V})$ admits a nowhere vanishing section
\[
\eta  \in H^0(   \mathcal{C}  , \det(\mathcal{V})  ).
\]
Such a section is unique up to scaling by $\Z^\alg[1/m]^\times$, and its norm $\| \eta \|^2$ is constant on $\mathcal{C}(\C)$.
\end{theorem}

\begin{proof}
 For some $g\in G(\A_f)$ we have a complex uniformization
\[
 \Gamma\backslash \mathcal{D} \map{ z\mapsto (z,g) } \mathcal{C}(\C) \subset  \mathrm{Sh}_{K(m)} ( G,\mathcal{D})  (\C)  ,
\]
where $\Gamma= G(\Q) \cap g K(m) g^{-1}$, and under this uniformization the total space of the vector bundle $\det(\mathcal{V})$ is isomorphic to
$ \Gamma\backslash (  \mathcal{D} \times \C )$,  where the action of $\Gamma$ on $\C$ is via  the composition
\[
\Gamma \subset G(\Q) \to \GL(W) \map{\det}\kk^\times \subset \C^\times.
\]

The compact open subgroup $K(m)$ is constructed in such a way that there is a $\co_\kk$-lattice  $g \mathfrak{a} \subset W(\kk)$ stabilized by $\Gamma$, and such that $\Gamma$ acts trivially on $g \mathfrak{a} / m g \mathfrak{a}$.  This implies that the above composition actually takes values in the subgroup
\[
\{ \zeta \in \co_\kk^\times : \zeta \equiv 1\pmod{m\co_\kk} \},
\]
which is trivial by our assumption that $m\ge 3$.  In other words, the vector bundle $\det( \mathcal{V})$ becomes (non-canonically) trivial after restriction to $\mathcal{X}(\C)$.
In fact, the argument of \cite[Theorem 1]{Gross} shows that one can find a trivializing section $\eta$ that is algebraic and defined over $\Q^\alg\subset \C$, and that such  a section is unique up to scaling by $(\Q^\alg)^\times$ and has constant norm $\|\eta\|^2$.

All that remains to show is that $\eta$ may be chosen so that it extends to a nowhere vanishing section over $\Z^\alg[1/m]$.
The key is to recall from \cite[\S 2.3]{BHKRY} that $\mathrm{Sh}(G,\mathcal{D})$ has a second integral model $\mathcal{S}_\Pap$ over $\co_\kk$, which is normal with geometrically normal fibers.
It is related  to the first by a surjective morphism $\mathcal{S}_\Kra \to \mathcal{S}_\Pap$, which restricts to an isomorphism over $\co_\kk[1/D]$.
It has a moduli interpretation very similar to that of $\mathcal{S}_\Kra$, which allows us to do two things.  First, there is a canonical descent of the vector bundle $\mathcal{V}$ to $\mathcal{S}_\Pap$, defined again by $\mathcal{V} = H_1^\dR(A) / \overline{\epsilon} H_1^\dR(A)$, but where now $(A_0,A)$ is the universal pair over $\mathcal{S}_\Pap$.
Second, we can add level $K(m)$ structure to obtain a cartesian diagram
\[
\xymatrix{
{  \mathcal{S}_\Kra(m)  }  \ar[r]  \ar[d]  &  {   \mathcal{S}_{\Kra / \co_\kk [1/m] }   } \ar[d] \\
{  \mathcal{S}_\Pap(m)  }  \ar[r]   &  {   \mathcal{S}_{\Pap / \co_\kk [1/m] }   }
}
\]
of $\co_\kk[1/m]$-stacks with \'etale horizontal arrows.

In particular, $\mathcal{S}_\Pap(m)$ is normal with  geometrically normal fibers, from which it follows that the above diagram extends to
\[
\xymatrix{
{ \mathcal{C} } \ar[r] \ar[d] &  {  \mathcal{S}_\Kra(m)_{/\Z^\alg[1/m]}   }  \ar[r]  \ar[d]  &  {   \mathcal{S}_{\Kra / \Z^\alg [1/m] }   } \ar[d] \\
{ \mathcal{B} } \ar[r] & {  \mathcal{S}_\Pap(m)_{/\Z^\alg[1/m]}  }  \ar[r]   &  {   \mathcal{S}_{\Pap / \Z^\alg [1/m] }   }
}
\]
for some connected component $\mathcal{B} \subset  \mathcal{S}_\Pap(m)_{/\Z^\alg[1/m]} $ with irreducible fibers.

Now fix a number field $L\subset \C$ containing $\kk$ large enough that the section $\eta$ and the components $\mathcal{C}$ and $\mathcal{B}$ are defined over $\co_L[1/m]$.
Viewing $\eta$ as a rational section of the line bundle $\det(\mathcal{V})$ on $\mathcal{B}$, its divisor is a finite sum of vertical fibers of $\mathcal{B}$, and so there is a fractional $\co_L[1/m]$-ideal $\mathfrak{b}\subset L$ such that
\[
\mathrm{div}(\eta) =  \sum_{\mathfrak{q}\mid \mathfrak{b}}  \ord_\mathfrak{q}(\mathfrak{b}) \cdot \mathcal{B}_\mathfrak{q},
\]
where $\mathcal{B}_\mathfrak{q}$ is the mod $\mathfrak{q}$ fiber of $\mathcal{Y}$.
By enlarging $L$ we may assume that $\mathfrak{b}$ is principal, and hence $\eta$ can be rescaled by an element of $L^\times$ to have trivial divisor on $\mathcal{B}$.
But then $\eta$ also has trivial divisor on $\mathcal{C}$, as desired.
\end{proof}

\begin{corollary}\label{cor:almost constant}
Let $\mathcal{A}\subset \mathcal{S}_\Kra$ be a connected component.
There is a constant $c=c_\mathcal{A} \in \R$ with the following property: for any finite extension $L/\kk$ and any morphism $\Spec(\co_L) \to \mathcal{A}$, the image of $\det(\mathcal{V})$
under
\begin{equation}\label{test degree}
\widehat{\Pic} (\mathcal{S}_\Kra) \to \widehat{\Pic} (\mathcal{A}) \to  \widehat{\Pic} ( \Spec(\co_L) )  \map{\widehat{\deg}} \R
\end{equation}
is equal to $ c\cdot [L:\kk]$.
\end{corollary}

\begin{proof}
Fix an  integer $m \ge 3$.   The open   and closed substack
\[
\mathcal{A}(m) = \mathcal{A} \times_{ \mathcal{S}_\Kra} \mathcal{S}_\Kra(m).
\]
of $ \mathcal{S}_\Kra(m)$, may be disconnected, so we fix one of its connected components
$\mathcal{A}(m)^\circ \subset \mathcal{A}(m)$.    This is an $\co_\kk[1/m]$-stack, which may become disconnected after base change to $\Z^\alg[1/m]$.
Fix one connected component
\[
\mathcal{C} \subset \mathcal{A}(m)^\circ_{/\Z^\alg[1/m]}.
\]
and let $\eta\in H^0(\mathcal{C} , \det(\mathcal{V} ))$ be a trivializing section as in Theorem \ref{thm:gross}.

Choose a finite Galois extension $M/\kk$ contained in $\C$,  large enough that $\mathcal{C}$ and $\eta$ are defined over $\co_M[1/m]$.
For each $\sigma\in \Gal(M/\kk)$ we obtain a trivializing section
\[
\eta^\sigma \in H^0( \mathcal{C}^\sigma ,  \det(\mathcal{V} ) )
\]
which, by Theorem \ref{thm:gross}, has constant norm $\|\eta^\sigma\|$.

Let $\R(m)$ be the quotient of $\R$ by the $\Q$-span of $\{ \log(p) : p\mid m\}$, and define
\[
c(m) = \frac{-1}{ [ M :\kk]  }  \sum_{ \sigma\in \Gal(M/\kk) } \log \|\eta^\sigma\|^2 \in \R(m).
\]
This is independent of the choice of $M$, and also independent of $\eta$ by the uniqueness claim of Theorem \ref{thm:gross}.
Moreover, for any number field $L/\kk$ and any morphism
\[
\Spec(\co_L[1/m]) \to \mathcal{A}(m)^\circ,
\]
the image of $\det(\mathcal{V})$ under
\[
\widehat{\Pic}   (\mathcal{A}(m)^\circ) \to \widehat{\Pic}  ( \Spec(   \co_L[1/m] ) ) \map{\widehat{\deg}} \R(m)
\]
is equal to $c(m) \cdot [L:\kk]$.

Now suppose we are given some $\Spec(\co_L) \to \mathcal{A}$ as in the statement of the corollary.
After possible enlarging $L$,  this morphism admits a lift
\[
\xymatrix{
& {   \mathcal{A}(m)^\circ  }  \ar[d]  \\
{  \Spec(\co_L[1/m]) }  \ar[r]  \ar@{-->}[ru]  & {   \mathcal{A}_{/\co_\kk[1/m] }   ,}
}
\]
and from this it is easy to see that the image of $\det(\mathcal{V})$ under the composition of (\ref{test degree}) with $\R\to \R(m)$ is equal to $c(m)\cdot [L:\kk]$.

In particular, the image of $\det(\mathcal{V})$ under the composition of (\ref{test degree}) with the diagonal embedding
\[
\R\hookrightarrow   \prod_{ m \ge 3 } \R(m)
\]
is equal to the tuple of constants $c(m)\cdot [ L:\Q]$.  What this proves is that there is a unique $c\in \R$ whose image under the diagonal embedding is the tuple of constants $c(m)$, and that this is the $c$ we seek.
\end{proof}

\begin{proposition}\label{prop:constant eval}
The constant $c=c_\mathcal{A}$ of Corollary \ref{cor:almost constant} is independent of $\mathcal{A}$, and is equal to
\[
c = (4-2n)  h^\Falt_\kk +  \log ( 4\pi^2 D ),
\]
where $h^\Falt_\kk$ is the Faltings height (\ref{chowla-selberg}).
\end{proposition}

\begin{proof}
Recall that we have fixed a triple $(\mathfrak{a}_0,\mathfrak{a} , i_E)$ as in \S \ref{ss:big cm}.
Fix a $g\in G(\A_f)$ in such a way that the map
\[
\mathcal{D} \map{ z\mapsto (z,g) } \mathrm{Sh}(G,\mathcal{D})(\C)
\]
factors through $\mathcal{A}(\C)$,  and a decomposition of $\co_\kk$-modules
\[
g \mathfrak{a} = \mathfrak{a}_1\oplus \cdots \oplus \mathfrak{a}_n
\]
in which each $\mathfrak{a}_i$ is projective of rank $1$.  Define elliptic curves over the complex numbers by
\[
A_i(\C)   =  g \mathfrak{a}_i \backslash \mathfrak{a}_{i\C} / \overline{\epsilon} \mathfrak{a}_{i\C}  .
\]
for $0 \le  i <n$, and
\[
A_{n}(\C)  = g \mathfrak{a}_n \backslash \mathfrak{a}_{n \C} /  \epsilon \mathfrak{a}_{n \C} .
\]

 Endow the abelian variety  $A=A_1\times \cdots \times A_n$ with the diagonal action of $\co_\kk$, and the principal polarization induced by the perfect symplectic form on $g\mathfrak{a}$,  as in the proof of \cite[Proposition 2.2.1]{BHKRY}.
The pair $(A_0,A)$ then corresponds to a point $(z,g) \in  \mathcal{A}(\C)$.

As each  $A_i$ has complex multiplication by $\co_\kk$, we may choose a number field $L$ containing $\kk$ over which all of these elliptic curves are defined and have everywhere good reduction.  If we denote again by $A_0,\ldots, A_n$ and $A$ the N\'eron models over $\Spec(\co_L)$,  the pair $(A_0,A)$  determines a morphism
 \[\Spec(\co_L) \to \mathcal{A} \subset  \mathcal{S}_\Kra.\]

The pullback of $\mathcal{V}$ to $\Spec(\co_L)$ is the rank $n$ vector bundle
\[
\mathcal{V}|_{\Spec(\co_L)} \iso \mathcal{V}_1 \oplus \cdots \oplus  \mathcal{V}_n ,
\]
where $\mathcal{V}_i = H_1^\dR(A_i) / \overline{\epsilon} H_1^\dR(A_i)$.
We endow
$
\mathcal{V}_i^{-1} \iso \epsilon H^1_{\dR}(A_i )
$
with the metric (\ref{integration}), so that
\[
\det(\mathcal{V})|_{\Spec(\co_L)}  \iso  \mathcal{V}_1 \otimes \cdots \otimes \mathcal{V}_n
\]
is an isomorphism of metrized line bundles.

The following two lemmas relate the images of $\mathcal{V}_1,\ldots, \mathcal{V}_n$ under the arithmetic degree
\begin{equation}\label{test degree 2}
\widehat{\Pic}( \Spec(\co_L))  \map{\widehat{\deg}} \R
\end{equation}
to the Faltings height $h_\kk^\Falt$.

\begin{lemma}
For $1\le i < n$,  the arithmetic degree (\ref{test degree 2}) sends
\[
\mathcal{V}_i \mapsto - [L:\Q] \cdot h_\kk^\Falt.
\]
\end{lemma}

\begin{proof}
The action of $\co_\kk$ on $\Lie(A_i)$ is through the inclusion $\co_\kk \to \co_L$, and hence,
as in \cite[Remark 2.3.5]{BHKRY}, the quotient map \[H_1^\dR(A_i)  \to \Lie(A_i)\] descends to an isomorphism of line bundles
$
\mathcal{V}_i \iso  \Lie(A_i).
$
If we endow  $\Lie(A_i)^{-1}$ with the Faltings metric  (\ref{integration}) then this isomorphism respects the metrics, and
the claim  follows  as in the proof of Proposition \ref{prop:faltings bundles}.
\end{proof}

\begin{lemma}
The arithmetic degree (\ref{test degree 2}) sends
\[
 \mathcal{V}_n \mapsto  [L:\Q]  \cdot \big(   h_\kk^\Falt -   \frac{1}{2} \log(4\pi^2 D) \big).
 \]
\end{lemma}

\begin{proof}
The action of $\co_\kk$ on $\Lie(A_i)$ is through the complex conjugate of the inclusion $\co_\kk \to \co_L$,  from which it follows that the Hodge short exact sequence takes the form
\[
\xymatrix{
{ 0 }  \ar[r]  &  { F^0(A_n) } \ar[r] \ar@{=}[d]  &  {   H_1^\dR(A_n) }  \ar[r] \ar@{=}[d]  & { \Lie(A_n )  } \ar[r]\ar@{=}[d]  & {0} \\
{ 0 }  \ar[r]  &  { \epsilon   H_1^\dR(A_0)  } \ar[r]  &  {   H_1^\dR(A_n) }  \ar[r]  & {  H_1^\dR(A_n) / \epsilon H_1^\dR(A_n)  } \ar[r] & {0.}
}
\]
In particular,  the endomorphism $\epsilon$ on  $H_1^\dR(A_n)$  descends to an isomorphism  $\mathcal{V}_n \iso F^0(A_n)$.

Let
\[
\psi_n :H_1^\dR(A_n) \otimes H_1^\dR(A_n) \to \co_L
\]
be the perfect pairing induced by the principal polarization on $A_n$, and define a second pairing
$
\Psi( x,y) = \psi_n( \epsilon x,y) .
$
It follows from the previous paragraph that this descends to a perfect pairing
\[
\Psi : \mathcal{V}_n \otimes \Lie(A_n) \iso \co_L.
\]
However, if  we endow $\Lie(A_n)^{-1}$ with the Faltings metric (\ref{integration}), then this pairing is not a duality between metrized line bundles.

Instead, an argument as in the proof of Proposition \ref{prop:bundle shuffle} shows that
\[
\Psi : \mathcal{V}_n \otimes \Lie(A_n) \iso \co_L \left\langle \frac{1}{2\pi \sqrt{D}} \right\rangle.
\]
is an isomorphism of metrized line bundles.  With this isomorphism in hand, the remainder of the proof is exactly as in the previous lemma.
\end{proof}

The two lemmas show that the image of $\det(\mathcal{V})$ under (\ref{test degree}) is
\[
\sum_{i=1}^n \widehat{\deg} ( \mathcal{V}_i) =  [L:\Q]  \cdot \Big(  (2-n) \cdot h_\kk^\Falt - \frac{1}{2} \log(4\pi^2 D) \Big)
\]
as claimed.  This completes the proof of Proposition \ref{prop:constant eval}.
\end{proof}

\begin{proposition}\label{prop:junk bundle two}
The metrized line bundle  $\det(\mathcal{V}) $ satisfies
\[
[ \det(\mathcal{V}) : \mathcal{Y}_\mathrm{big} ]   =    \deg_\C( \mathcal{Y}_\mathrm{big} ) \cdot  \Big( (4-2n)  h^\Falt_\kk +  \log (4\pi^2 D)  \Big) .
\]
\end{proposition}

\begin{proof}
As in the proof of Proposition \ref{prop:faltings bundles}, we may fix a finite extension $L/E_\Phi$ and a finite \'etale cover
$
Y_\mathrm{big} = \bigsqcup \Spec(\co_L)
$
of the $\co_\kk$-stack
\[
\mathcal{Y}_\mathrm{big} \times_{\Spec(\co_\Phi)} \Spec(\co_L)
\] 
 by, say, $m$ copies of $\Spec(\co_L)$.  Corollary \ref{cor:almost constant} then implies
\[
  \frac{  [ \det(\mathcal{V}) : \mathcal{Y}_\mathrm{big} ]  } {  \deg_\C( \mathcal{Y}_\mathrm{big} )   }   =
   \frac{  [ \det(\mathcal{V}) : \mathrm{Y}_\mathrm{big} ]  } {  \deg_\C( \mathrm{Y}_\mathrm{big} )   }
   = \frac{ cm\cdot [ L :\kk]  }{ m\cdot [L:\kk]} =  c.
\]
Appealing to the evaluation of the constant $c$ found in Proposition \ref{prop:constant eval} completes the proof.
\end{proof}


\subsection{Theorems \ref{Theorem C} and  \ref{Theorem D}}


We can now put everything together, and relate the arithmetic degree of $\widehat{\omega}$ along $\mathcal{Y}_\mathrm{big}$ to the Faltings height $h^\Falt_{(E,\Phi)} $.

\begin{proposition}\label{prop:faltings relation}
The metrized line bundle $\widehat{\bm{\omega}}$ satisfies
\[
  \frac{ [ \widehat{\bm{\omega}} : \mathcal{Y}_\mathrm{big} ]} {   \deg_\C(\mathcal{Y}_\mathrm{big}) }
  = h^\Falt_{(E,\Phi)}      +  \frac{n-4}{2}  \cdot  \frac{  \Lambda'(0,\chi_\kk)  }{ \Lambda(0,\chi_\kk) } + \frac{n}{4}  \log(16\pi^3e^\gamma)   .
\]
\end{proposition}

\begin{proof}
Proposition \ref{prop:bundle shuffle} shows that
\begin{align*}
 2  \cdot [ \co \langle 8\pi^2 e^\gamma D^{-1}\rangle  \otimes \widehat{\bm{\omega}}: \mathcal{Y}_\mathrm{big} ]
 & +  [ \mathrm{det}( \Lie(A))   : \mathcal{Y}_\mathrm{big} ]   +  2 \cdot   [ \Lie(A_0) : \mathcal{Y}_\mathrm{big} ]   \\
& =  [  \co(\mathrm{Exc})  : \mathcal{Y}_\mathrm{big} ]   + [  \det(\mathcal{V})  : \mathcal{Y}_\mathrm{big} ]   .
\end{align*}
Proposition \ref{prop:faltings bundles} and Remark \ref{trivial archimedean} imply that the left hand side is equal to
\[
2 \cdot  [ \widehat{\bm{\omega}} : \mathcal{Y}_\mathrm{big} ]
 - 2  \deg_\C(\mathcal{Y}_\mathrm{big}) \cdot  \left(  \log(8\pi^2 e^\gamma D^{-1})  + h^\Falt_{(E,\Phi)}
   +2 \cdot h^\Falt_\kk \right),
\]
while Proposition \ref{prop:junk bundle two} shows that  the right hand side is equal to
\[
2  \deg_\C( \mathcal{Y}_\mathrm{big} ) \cdot  \big(    (2-n)  h^\Falt_\kk + \log (2\pi D)  \big).
\]
Note that we have used here the equality
\[
[  \co(\mathrm{Exc})  : \mathcal{Y}_\mathrm{big} ]  =[  ( \mathrm{Exc} , 0)  : \mathcal{Y}_\mathrm{big} ]   =   \deg_\C( \mathcal{Y}_\mathrm{big} )  \cdot \log(D) .
\]
from the proof of Proposition \ref{prop:big no error}.

Combining these formulas yields
\[
   \frac{ [ \widehat{\bm{\omega}} : \mathcal{Y}_\mathrm{big} ]} {   \deg_\C(\mathcal{Y}_\mathrm{big}) }
     = h^\Falt_{(E,\Phi)}   + (4-n)  h^\Falt_\kk  + \log(16\pi^3 e^\gamma)    ,
\]
and substituting the value (\ref{chowla-selberg}) for $h_\kk^\Falt$ completes the proof.
\end{proof}

It is clear  from Proposition \ref{prop:faltings relation}  that Theorems \ref{Theorem C} and Theorem \ref{Theorem D} are equivalent.
As Theorem  \ref{Theorem C} is proved  in \cite{Yang-Yin}, this completes the proof of Theorem \ref{Theorem D}.

On the other hand, we proved Theorem \ref{Theorem D} in \S \ref{ss:big derivative} under the assumption that $n\ge 3$ and the discriminants of $\kk$ and $F$ are odd and relatively prime, and so this gives a new proof of Theorem  \ref{Theorem C} under these hypotheses.



\bibliographystyle{smfalpha}


\newcommand{\etalchar}[1]{$^{#1}$}
\providecommand{\bysame}{\leavevmode\hbox to3em{\hrulefill}\thinspace}
\providecommand{\MR}{\relax\ifhmode\unskip\space\fi MR }
\providecommand{\MRhref}[2]{%
  \href{http://www.ams.org/mathscinet-getitem?mr=#1}{#2}
}
\providecommand{\href}[2]{#2}

\end{document}